\documentclass[11pt]{amsart}

\usepackage[hmargin=3cm,vmargin=3cm]{geometry}

\usepackage[latin1]{inputenc}
\usepackage{xspace,amssymb,amsfonts,euscript}
\usepackage{amsthm,amsmath,amscd,latexsym}
\usepackage{palatino, euscript}
\usepackage[all]{xy}
\usepackage{tikz, tikz-cd}
\usetikzlibrary{matrix, arrows, calc}
\tikzset{frontline/.style={preaction={draw=white,-,line width=6pt}},}  %%% for 3d commutative diagrams
\tikzset{anchorbase/.style={baseline={([yshift=-0.5ex]current bounding box.center)}},
    tinynodes/.style={font=\tiny,text height=0.75ex,text depth=0.15ex},
    smallnodes/.style={font=\scriptsize,text height=0.75ex,text depth=0.15ex},
    >={Latex[length=1mm, width=1.5mm]}
}
\usepackage[dvips]{epsfig}
\usepackage{pinlabel}

\newcommand{\ig}[2]{\vcenter{\xy (0,0)*{\includegraphics[scale=#1]{fig/#2}} \endxy}}

\RequirePackage{color}
\definecolor{myred}{rgb}{0.75,0,0}
\definecolor{mygreen}{rgb}{0,0.5,0}
\definecolor{myblue}{rgb}{0,0.25,0.65}
\definecolor{references}{rgb}{0,0,1}

\RequirePackage[pdftex,
colorlinks = true,
urlcolor = references, % \href{...}{...} external (URL)
citecolor = references, % \cite{...}
linkcolor = references, % \ref{...} and \pageref{...}
]
 {hyperref}
 
\newtheorem{thm}{Theorem}[section]
\newtheorem{lemma}[thm]{Lemma}
\newtheorem{theorem}[thm]{Theorem}

\newtheorem{prop}[thm]{Proposition}
\newtheorem{proposition}[thm]{Proposition}
\newtheorem{cor}[thm]{Corollary}
\newtheorem{corollary}[thm]{Corollary}

\newtheorem*{prop*}{Proposition}
\newtheorem*{lemma*}{Lemma}

\theoremstyle{definition}
\newtheorem{defn}[thm]{Definition}
\newtheorem{definition}[thm]{Definition}

\newtheorem{notation}[thm]{Notation}

\newtheorem{example}[thm]{Example}

\theoremstyle{remark}
\newtheorem{remark}[thm]{Remark}

\newtheorem{question}[thm]{Question}

\numberwithin{equation}{section}

\def\BB{{\mathbf B}}    
\def\CB{{\mathbf C}}

    \def\FC{{\mathcal{F}}}
    
    \def\HC{{\mathcal{H}}}

    \def\KC{{\mathcal{K}}}

\def\PB{{\mathbf P}}

\def\XB{{\mathbf X}}

\def\AS{{\EuScript A}}
\def\BS{{\EuScript B}}
\def\CS{{\EuScript C}}
\def\DS{{\EuScript D}}

\def\MS{{\EuScript M}}
\def\NS{{\EuScript N}}

\def\QS{{\EuScript Q}}

\def\ZS{{\EuScript Z}}

\def\a{\alpha}
\def\b{\beta}

\def\d{\delta}

\def\e{\varepsilon}
\def\l{\lambda}

\let\phi=\varphi
\let\tilde=\widetilde

\usepackage{bbm}

\def\R{{\mathbbm R}}
\def\Z{{\mathbbm Z}}

\def\1{\mathbbm{1}}
\renewcommand{\k}{\mathbbm{k}}

\newcommand{\inv}{^{-1}}

\newcommand{\ot}{\otimes}
\newcommand{\pa}{\partial}
\newcommand{\co}{\colon}

\renewcommand{\to}{\rightarrow}

\newcommand{\refequal}[1]{\xy {\ar@{=}^{#1}
(-1,0)*{};(1,0)*{}};
\endxy}

\newcommand{\Bim}{\textbf{Bim}}

\newcommand{\Hom}{{\rm Hom}}

\newcommand{\End}{{\rm End}}

\newcommand{\Res}{{\rm Res}}

\newcommand{\id}{{\rm id}}

\newcommand{\fin}{{\rm fin}}

\newcommand{\res}{\textrm{res}}

\newcommand{\FT}{FT}

\newcommand{\Cone}{{\rm Cone}}
\newcommand{\startdotred}{\ig{.5}{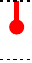}}

\newcommand{\finaldotred}{\ig{.5}{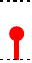}}

\newcommand{\linered}{\ig{.5}{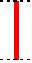}}

\newcommand{\barbred}{\ig{.5}{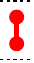}}

\newcommand{\brokenred}{\ig{.5}{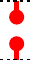}}

\newcommand{\splitred}{\ig{.5}{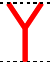}}

\DeclareMathOperator{\Br}{Br}

\renewcommand{\Bar}{\operatorname{Bar}}

\newcommand{\Diag}{\HC}

\DeclareMathOperator{\Ch}{Ch}

\newcommand{\Tot}{\operatorname{Tot}}
\newcommand{\Id}{\id}
\newcommand{\one}{\1}

\newcommand{\uHom}{\underline{\Hom}}
\newcommand{\uEnd}{\underline{\End}}
\renewcommand{\d}{\delta}
\renewcommand{\a}{\alpha}
\newcommand{\tw}{\operatorname{tw}}
\renewcommand{\tw}{\operatorname{tw}}

\renewcommand{\l}{\lambda}

\newcommand{\ttau}{\boldsymbol\tau}

\newcommand{\HT}{\mathrm{HT}}
\newcommand{\Obj}{\operatorname{Obj}}
\newcommand{\sep}{\, |\, }

\renewcommand{\FT}{\operatorname{FT}}

\newcommand{\FZ}{F}
\newcommand{\GZ}{G}
\newcommand{\Fdg}{\mathbf{F}}
\newcommand{\Gdg}{\mathbf{G}}

\newcommand{\etaZ}{\eta}
\newcommand{\etadg}{\boldsymbol{\eta}}

\newcommand{\ip}[1]{\left\langle #1 \right\rangle}
\newcommand{\oone}{\mathbf{I}}
\newcommand{\premod}[2]{#1\text{-premod}(#2)}
\newcommand{\Qmod}[1]{#1^{\operatorname{bim}}}
\newcommand{\dgMod}[2]{#2\text{-dMod}^{#1}}
\newcommand{\bimhom}[2]{B^{#1}_{#2}}
\newcommand{\HH}{\operatorname{HH}}

\newcommand{\fm}{\mathfrak{m}}
\newcommand{\bard}{\alpha}
\newcommand{\ibim}[1]{\mathbf{I}_{#1}}
\newcommand{\mergebim}[1]{\boldsymbol{\mu}_{#1}}
\newcommand{\splitbim}[1]{\boldsymbol{\Delta}_{#1}}
\newcommand{\ubim}[1]{\boldsymbol{\eta}_{#1}}
\newcommand{\coubim}[1]{\boldsymbol{\e}_{#1}}

\definecolor{electricindigo}{rgb}{0.44, 0.0, 1.0}
\definecolor{spirodiscoball}{rgb}{0.06, 0.75, 0.99}

\DeclareMathOperator{\ext}{ext}
\DeclareMathOperator{\Flat}{Flat}
\DeclareMathOperator{\pr}{pr}

\newcommand{\Matt}{\Xi}
\newcommand{\Coalg}{\mathbf{Coalg}}

\begin{document}

\begin{abstract}
The Drinfeld centralizer of a monoidal category $\AS$ in a bimodule category $\MS$ is the category $\ZS(\AS,\MS)$ of objects in $\MS$ for which the left and right actions by objects of $\AS$ coincide, naturally.  In this paper we study the interplay between Drinfeld centralizers of $\AS$ and its homotopy category $\KC^b(\AS)$; culminating with our ``lifting lemma'', which provides a sufficient condition for an object of $\ZS(\AS,\KC^b(M))$ to lift to an object of $\ZS(\KC^b(\AS),\KC^b(\MS))$.

The central application of this lifting lemma is a proof of some folklore facts about conjugation by Rouquier complexes in the Hecke category: the centrality of the full twist, and related properties of half twists and Coxeter braids.

%These folklore results are proven using an abstract result concerning the lift of natural transformations of functors from $\Fun_\k(\AS,\KC^b(\BS))$ to $\Fun_\k(\KC^b(\AS),\KC^b(\BS))$, which should be of independent interest.

We also prove stronger, homotopy coherent versions of these statements, stated using the notion of the $A_\infty$-Drinfeld centralizer, which we believe is new.
\end{abstract}

%\title{Extending isomorphisms of functors to homotopy categories, or how to prove something is in the Drinfeld center}

\title{Drinfeld centralizers and Rouquier complexes}

\author{Ben Elias} \address{University of Oregon, Eugene}

\author{Matthew Hogancamp} \address{Northeastern University, Boston}

\maketitle

\setcounter{tocdepth}{1}
\tableofcontents

%This paper concerns the interplay between centralizers of bimodule categories, and homotopy centralizers of their homotopy categories.  We begin by recalling the relevant notions of centralizer, and the most common examples of interest.  We introduce some homological algebra tools along the way that may be of more general interest.

%%%%%%%%%%%%%%%%%%%%%%%%%
\section{Introduction}
\label{s:intro}
%%%%%%%%%%%%%%%%%%%%%%%%%

\subsection{Drinfeld centralizers}

Let $A$ be a ring and $M$ an $(A,A)$-bimodule.  Define the \emph{centralizer of $A$ in $M$} to be the subset $Z(A,M)\subset M$ consisting of elements $m\in M$ such that $am=ma$ for all $a\in A$.  Note that this coincides with the zeroth Hochschild cohomology $Z(A,M)=\HH^0(A,M)$.
The center of $A$ is the same as $Z(A,A)$. More generally if $A\subset B$ is an inclusion of algebras then $Z(A,B)$ coincides with the centralizer of $A$ in $B$ in the usual sense.

We are interested in the same notion but one categorical level higher.  For this we will need  a monoidal category $\AS$ (note, all categories in this paper will be assumed linear over a fixed commutative ground ring $\k$). We will denote the monoidal structure on $\AS$ by $\star$ and the monoidal identity by $\one$. Let $\MS$ be a $\star$-bimodule\footnote{In the literature a $\star$-bimodule category is typically just called a bimodule category. See Remark \ref{rmk:bimoduleargh}.} category over $\AS$, i.e.~$\MS$ comes equipped with functors
\[
\AS\otimes_\k \MS\buildrel\star \over \to \MS, \qquad \MS \otimes_\k \AS \buildrel\star \over \to \MS,
\]
satisfying the the usual unit and associativity constraints up to natural isomorphism.  See \cite[Chapter 7]{EGNO15} for details.

If $Z$ is an object of $\MS$, then $Z \star (-)$ is a functor from $\AS$ to $\MS$, as is $(-) \star Z$.

\begin{definition}\label{def:intro centralizer}
The \emph{Drinfeld centralizer of $\AS$ in $\MS$}, denoted $\ZS(\AS,\MS)$, is the category whose objects are pairs $(Z,\tau)$, in which $Z\in \MS$ is an object of $\MS$, and $\tau$ is a natural isomorphism of functors $\tau \colon Z\star (-) \to (-) \star Z$ satisfying the additional constraint that
\begin{equation}
\tau_{X\star Y} = (\id_X\star \tau_Y)\circ (\tau_X\star \id_Y)
\end{equation}
for all $X,Y\in \AS$. Here $\tau_X \colon Z \star X \to X \star Z$ is the morphism induced by $\tau$.
\end{definition}

See \S \ref{ss:relative drinfeld} for more details, as well as a description of morphisms in $\ZS(\AS,\MS)$.  Note that naturality of $\tau$ means that each morphism $f:X\rightarrow Y$ in $\AS$ fits into a commutative square
\begin{equation} \label{centralsquare} 
\begin{tikzpicture}
\node (a) at (0,0) {$Z\star X$};
\node (b) at (3,0) {$Z\star Y$};
\node (c) at (0,2) {$X\star Z$};
\node (d) at (3,2) {$Y\star Z$};
\path[-stealth,thick]
(a) edge node[above] {$\id_Z\star f$} (b)
(c) edge node[above] {$f\star \id_Z$} (d)
(a) edge node[left] {$\tau_X$} (c)
(b) edge node[right] {$\tau_Y$} (d); 
\end{tikzpicture}.
\end{equation}

\begin{remark} \label{remark:twistedcenter} Let $\Phi$ be a monoidal automorphism of $\AS$. It is also interesting to study the $\Phi$-twisted Drinfeld centralizer, whose obtains are pairs $(Z,\tau)$ such that $\tau$ is a natural isomorphism from $Z \star (-)$ to $\Phi(-) \star Z$. Note that the $\Phi$-twisted Drinfeld centralizer is just another example of a Drinfeld centralizer, agreeing with $\ZS(\AS,{}^{\Phi} \MS)$ for the $\star$-bimodule category ${}^{\Phi} \MS$ where the left action of $\AS$ is twisted by $\Phi$. \end{remark}

% If the categories $\AS$ and $\MS$ are not merely $\k$-linear but possess structure (e.g. a grading) then one defines the ``enriched relative Drinfeld center'' of $\AS$ in $\MS$ in precisely the same way, but where the notion of natural transformation depends on the enrichment.  Passing to the homotopy categories of finite complexes (viewed as a category enriched in $\Z$-graded $\k$-modules with appropriate Koszul sign rule) yields a notion of $\ZS(\KC^b(\AS),\KC^b(\MS))$.  Objects of this category are complexes $Z\in \KC^b(\MS)$ equipped with a given (natural-up-to-homotopy) family of homotopy equivalences $Z\star X\simeq X\star Z$ for all $X\in \KC^b(\AS)$.

If $\MS$ is a $\star$-bimodule category over $\AS$ then the category $\Ch^b(\MS)$ of bounded chain complexes\footnote{For the purposes of this part of the introduction, $\Ch^b(\MS)$ is a $\k$-linear category whose morphisms are chain maps. Later we will notate this category as $Z^0(\Ch^b(\MS))$, reserving $\Ch^b(\MS)$ for the dg category whose morphisms are all $\k$-linear maps (not necessarily chain maps). We begin using dg categories in \S\ref{ss:intro infty centralizer} and make all proper definitions in \S\ref{s:dg}.} over $\MS$ is a $\star$-bimodule category over $\Ch^b(\AS)$. Taking the quotient by
nulhomotopic maps, the homotopy category $\KC^b(\MS)$ is a $\star$-bimodule category over $\KC^b(\AS)$. This yields a notion of $\ZS(\KC^b(\AS),\KC^b(\MS))$. Objects of this category are
complexes $Z\in \KC^b(\MS)$ equipped with a family of homotopy equivalences $Z\star X\simeq X\star Z$ for all $X\in \KC^b(\AS)$, which satisfy \eqref{centralsquare} up to homotopy. This is typically a much more fruitful object of study than $\ZS(\AS, \MS)$ itself. There are several known situations where $\ZS(\AS,\MS)$ is trivially small, whereas $\ZS(\KC^b(\AS), \KC^b(\MS))$ is quite interesting, such as when $\AS = \MS = \Diag$ is the Hecke category, see \S\ref{ss:intro rouquier}.

\begin{remark}
A typical central element in an algebra is a linear combination of basis elements, often with signs. To categorify a linear combination with signs one commonly uses complexes of objects and takes the Euler characteristic. Note that $\ZS(\Ch^b(\AS), \Ch^b(\MS))$ is much smaller than $\ZS(\KC^b(\AS), \KC^b(\MS))$ in practice: it is much easier for $Z \star X$ and $X \star Z$ to be homotopy equivalent than to be isomorphic as complexes. \end{remark}

For the practical reader, a very natural question that arises now is: how does one show that a given complex $Z\in \KC^b(\MS)$ admits the structure of an object in the Drinfeld centralizer $\ZS(\KC^b(\AS),\KC^b(\MS))$? To provide a morphism $\tau_X$ for each complex $X \in \KC^b(\AS)$ seems like a prohibitive amount of data.

The problem of constructing objects in $\ZS(\AS,\KC^b(\MS))$ is much more tractable, provided one is given a monoidal presentation of $\AS$ by generators and relations. That is, to show that $Z \in \KC^b(\MS)$ centralizes objects of $\AS$ we need only provide one morphism $\tau_X$ for each generating object $X$ of $\AS$, and check one square \eqref{centralsquare} for each generating morphism. Many related computations have been done in the Hecke category (see e.g. \cite{MMVFlat}), though this becomes more difficult outside of type A.

In this paper our goal is to provide tools which minimize the work needed to construct central objects in general. The following theorem is an abstract method of constructing objects in $\ZS(\AS,\KC^b(\MS))$. To state it, let $Z\in \KC^b(\MS)$ be given, and let $L_Z,R_Z\colon \AS\to \KC^b(\MS)$ denote the functors $L_Z(X)=Z\star X$ and $R_Z(X)=X\star Z$.
Theorem \ref{thm:intro Phi and Z} is more elementary and quite different in flavor to what we accomplish in the rest of the paper.

\begin{theorem}\label{thm:intro Phi and Z}
Retain the setup above.  Assume that $L_Z$ and $R_Z$ are fully faithful and have the same essential image.  Then there is a monoidal autoequivalence $\Phi\colon \AS\to \AS$ such that $Z$ can be equipped with the structure of an object in the Drinfeld centralizer $\ZS(\AS,\KC^b({}^\Phi\MS))$, where ${}^\Phi\MS$ denotes $\MS$ with left action twisted by $\Phi$.
\end{theorem}

This is restated and proven as Theorem \ref{thm:Phi and Z}. One should think that $\Phi$ is $R_Z^{-1} L_Z$. Once one knows that an autoequivalence $\Phi$ exists, the work that
remains is to compute the autoequivalence $\Phi$. The object $Z$ will be in the ordinary Drinfeld centralizer if and only if $\Phi$ is isomorphic to the identity functor. For example, \cite{EBigraded} classifies the relevant autoequivalences of the Hecke category, making it easy to pin down the autoequivalence
$\Phi$ with a minimal number of computations.

A more weighty issue is the problem of lifting from $\ZS(\AS,\KC^b(\MS))$ to $\ZS(\KC^b(\AS),\KC^b(\MS))$, which is an essentially an ``obstruction theory'' problem.
It can be solved under appropriate ext-vanishing conditions, giving a sufficient condition for lifting. The following result is the main tool in our applications.  As is common in homological algebra, for two complexes $X, Y$ in $\Ch^b(\AS)$, we let $\uHom(X,Y)$ denote the hom complex, a chain complex of $\k$-modules whose elements are linear maps (not chain maps) of various degrees, and whose $0$-th homology is the space of morphisms in $\KC^b(\AS)$. See \eqref{eq:uHom} for details. 

% \begin{remark}
% Observe that there is a functor $\KC^b(\ZS(\AS,\MS))\to \ZS(\KC^b(\AS),\KC^b(\MS))$.  However, this functor is in practice very far from being an equivalence of categories since $\ZS(\AS,\MS)$ may be quite trivial even when $\ZS(\KC^b(\AS),\KC^b(\MS))$ is not.
% \end{remark}
%
% Thus, one should not expect very many objects of $\ZS(\KC^b(\AS),\KC^b(\MS))$ to come from objects of $\ZS(\AS,\MS)$.  Instead, it is much more fruitful to attempt to construct objects in $\ZS(\KC^b(\AS),\KC^b(\MS))$ by lifting objects in $\ZS(\AS,\KC^b(\MS))$.

%More precisely, we consider the following question.
%
%\begin{question}
%Let $\AS$ be a $\k$-linear monoidal category and $\MS$ is an $(\AS,\AS)$-bimodule category. Suppose $Z\in \KC^b(\MS)$ centralizes objects of $\AS$, i.e.~ there is an object $(Z,\tau)$ in $\ZS(\AS,\KC^b(\MS))$.  When is it possible to lift $(Z,\tau)$ to an object of $(Z,\ttau)$ of the centralizer $\ZS(\KC^b(\AS),\KC^b(\MS))$?  How unique is the lift?
%\end{question}

% The problem of showing that a given $Z\in \KC^b(\MS)$ centralizes objects of $\AS$ is fairly straightforward, provided one is given a presentation of $\AS$ by generators and relations.  The problem of lifting from $\ZS(\AS,\KC^b(\MS))$ to $\ZS(\KC^b(\AS),\KC^b(\MS))$ is an essentially an ``obstruction theory'' problem, and can be solved under appropriate ext-vanishing conditions. In this paper, the following will be the main tool we use.

\begin{theorem}\label{thm:intro centralizer}
Suppose $Z\in \KC^b(\MS)$ is such that the complex $\uHom(Z\star X,Y\star Z)$ has zero homology in negative cohomological degrees for all $X,Y\in \AS$. Then any $(Z,\tau)\in \ZS(\AS,\KC^b(\MS))$ lifts to an object $(Z,\ttau)\in \ZS(\KC^b(\AS),\KC^b(\MS))$.
\end{theorem} 

Let us crudely explain the homological issue at hand. We know how $Z$ commutes past objects, and need to build morphisms which commute $Z$ past complexes. Suppose that $Z$ commutes
with objects $X$ and $Y$ of $\AS$, so that we have isomorphisms $\tau_X$ and $\tau_Y$, and suppose that \eqref{centralsquare} commutes up to homotopy for some $f \colon X \to Y$.
Now consider the two-term complex $\Cone(f)$, whose lone nonzero differential is the map $f$ from $X$ to $Y$. The top row of \eqref{centralsquare} can be loosely viewed as the complex
$\Cone(f) \star Z$; both terms $X \star Z$ and $Y \star Z$ are complexes, so $\Cone(f) \star Z$ is better thought of as the total complex of (the bicomplex given by) the top row.
The same goes for the bottom row and $Z \star \Cone(f)$ (ignoring also the Koszul sign rule for tensor products). Then the vertical arrows $(\tau_X, \tau_Y)$ seem like
they should provide a chain map ``$\tau_X + \tau_Y$'' that could serve as $\ttau_{\Cone(f)}$. The issue is that $\tau_X + \tau_Y$ is not actually a chain map, since
\eqref{centralsquare} does not actually commute; it only commutes up to homotopy. Choosing such a homotopy $h$ (which is an element of $\uHom(Z \star X, Y \star Z)$ in degree
$-1$), one can produce a genuine chain map ``$\tau_X + \tau_Y + h$'' sending $Z \star \Cone(f) \to \Cone(f) \star Z$. This chain map depends on the choice of $h$! When
$\uHom(Z \star X, Y \star Z)$ has zero homology in degree $-1$, the choice of $h$ is irrelevant up to homotopy, so we pick one and call the resulting chain map $\ttau_{\Cone(f)}$.
When considering the functoriality of $\ttau_{\Cone(f)}$ or constructing maps $\ttau$ for iterated cones, we must venture deeper into negative degrees within $\uHom(Z \star X', Y'
\star Z)$ for various $X', Y' \in \AS$.

This theorem is restated and proven as Theorem \ref{thm:lifting drinfeld}.  The proof uses an abstract lifting lemma, stated and proved as Theorem \ref{thm:lifting lemma}, which addresses the following question. Let $\AS, \BS$ be $\k$-linear additive categories. A functor $\AS \to \Ch^b(\BS)$ extends naturally to a functor $\Ch^b(\AS) \to \Ch^b(\BS)$ and descends to a functor $\KC^b(\AS) \to \KC^b(\BS)$. Given two functors $F, G \colon \AS \to \Ch^b(\BS)$ and a natural transformation between the corresponding functors $\AS \to \KC^b(\BS)$, when does this lift to a natural transformation between the corresponding functors $\KC^b(\AS) \to \KC^b(\BS)$? In the case of Theorem \ref{thm:lifting drinfeld}, the two functors in question are $L_Z$ and $R_Z$. Again, we are able to lift any natural transformation in the absence of negative extensions between $F(X)$ and $G(Y)$, for $X, Y \in \AS$.

\subsection{Applications to Rouquier complexes}
\label{ss:intro rouquier}

Let $W$ be a Coxeter group with a finite set of simple reflections $S$.  In \S \ref{s:diag Hecke} we recall the notion of a realization of $(W,S)$, which roughly speaking is the data of a representation $V$ (free over $\k$) on which $S$ acts by reflections.  Associated to $(W,S,V)$ one has the \emph{diagrammatic Hecke category} $\Diag=\Diag(W,S,V)$.  This is a $\Z$-graded $\k$-linear strict monoidal category, 
%with objects given by formal expressions $B_{s_1}\cdots B_{s_r}$ where $s_i\in S$ for all $i$, and morphisms
explicitly\footnote{One relation remains inexplicit for Coxeter groups containing type $H_3$.} described by diagrammatic generators and relations.  In case $W=S_n$ we will denote the diagrammatic Hecke category by $\Diag_n$.

We also have the braid group $\Br:=\Br(W,S)$ associated to the Coxeter system. To each braid $\b\in \Br$ we have an asssociated \emph{Rouquier complex} $F(\b)\in \KC^b(\Diag)$.
Rouquier proved the existence of a canonical homotopy equivalence $F(\b_1) \star F(\b_2) \cong F(\b_1 \cdot \b_2)$ for any two braids $\b_1$ and $\b_2$. The Hecke category plays a ubiquitous role in geometric and categorical representation theory, and in particular, many categorical actions of braid groups arise from actions of the
Hecke category via Rouquier complexes, though often they are described in other languages (e.g. shuffling functors on category $\mathcal{O}$, or zigzag algebras) which obfuscates the appearance of the Hecke category. See e.g. \cite[Section 11]{Rouq06}, \cite{StroppelTL}, \cite{KhovanovSeidel}, as well as \cite{KhoRoz08, Khov07}.

The functor $F(\b) \star (-)$ is invertible, with inverse $F(\b^{-1}) \star (-)$. This invertiblility implies that that $L_{F(\b)}$ and $R_{F(\b)}$ are fully faithful. It also
implies (for all $X,Y\in \AS$) that the complex $\uHom(X \star F(\b),Y\star F(\b))$ is isomorphic to $\uHom(X,Y)$, and thus has zero homology in negative cohomological degrees. If
$F(\b) \star X \cong X \star F(\b)$ for all $X$ (with no assumptions of naturality) then we are primed to use Theorems \ref{thm:intro Phi and Z} and \ref{thm:intro centralizer}.

By this method we establish some folklore results on Rouquier complexes of Coxeter braids, half twists, and full twists. For more details on the theorems below, see \S \ref{s:diag Hecke}. More generally, these tools can be used to prove that conjugation by (the Rouquier complex of) a given braid does what one expects it would do, and does it functorially.

\begin{theorem}
\label{thm:intro half twist}
Let $(W,S)$ be a finite Coxeter group, which does not contain $H_3$ as a parabolic subgroup. Let $\HT$ be the Rouquier complex associated to the positive braid lift of the longest element, often called the \emph{half twist}.  Then $\HT$ lifts to an object of $\ZS(\KC^b(\Diag),{}^\Phi\KC^b(\Diag))$, where $\Phi$ is the composition of a Dynkin diagram automorphism and a sign twist.
\end{theorem}

\begin{cor}
\label{cor:intro full twist} 
Let $\FT=\HT^2$ be the Rouquier complex associated to the full twist braid.  Then $\FT_n$ lifts to an object of $\ZS(\KC^b(\Diag))$.
\end{cor}

We discuss the existing literature on the full twist below, but the literature is relatively silent on the half twist and its associated automorphism $\Phi$. While the Dynkin diagram automorphism is expected behavior, we are unaware of this sign twist appearing in the literature. The literature also has little to say outside of type $A$.

Interestingly, one can prove the results above with almost no direct computation. Let $\Phi$ be whatever automorphism is associated with conjugation by the half twist. Composing $\Phi$ with the Dynkin diagram automorphism associated to the longest element we get a monoidal autoequivalence of $\Diag$ preserving the objects. Such autoequivalences were classified in \cite[Chapter 3]{EBigraded}, and it is easy to classify them further up to isomorphism. Up to isomorphism, such an autoequivalence is determined by how it acts on $V$ (living inside the endomorphism ring of the monoidal identity), which is a simple and well-known computation. However, we do some extra work to show how one can also explicitly construct the structure map $\tau$ using ``canonical morphisms'' between Rouquier complexes.

\begin{remark} The presentation for the Hecke category is not fully explicit in type $H_3$, due to a missing ``Zamolodchikov'' equation. Consequently, the paper \cite{EBigraded}
assumes that $W$ does not contain $H_3$ as a parabolic subgroup, and for some of our results we must make the same assumption. If the missing Zamolodchikov relation is found and
the results of \cite{EBigraded} extended to that setting, then our results can be extended for free. \end{remark}

We have a relative version of Theorem \ref{thm:intro half twist} which is useful.  To state it, let $(W,S)$ be a finite Coxter group, and let $I\subset S$ be given.  Let $W_I\subset W$ be the (parabolic) subgroup.  Let $w_0(S)$ and $w_0(I)$ denote the longest elements of $W$ and $W_I$, respectively, and let $\HT_S$ and $\HT_I$ denote the associated Rouquier complexes.  Let $w_0(S/I):= w_0(S)w_0(I)$ with associated Rouquier complex $\HT_{S/I}$, so that
\begin{equation}
\HT_S \simeq \HT_{S/I}\star \HT_{I}.
\end{equation}

\begin{example}
In type $A$, consider the symmetric group $S_{n+m}$ with its parabolic subgroup $S_n\times S_m$ (so that $S=\{1,\ldots,m+n-1\}$ and $I=S\setminus\{n\}$.  Then
\begin{equation}\label{eq:type A rel HT}
\HT_{S/I} \ \ := \ \ F\left(\,
\begin{tikzpicture}[scale=1.5,baseline=.5cm]
\draw[very thick]
(1,0) -- (0,1)
(1.3,0) -- (.3,1)
(1.6,0) -- (.6,1);
\draw[line width=5pt,white]
(0,0) -- (1,1)
(.2,0) -- (1.2,1)
(.4,0) -- (1.4,1)
(.6,0) -- (1.6,1);
\draw[very thick]
(0,0) -- (1,1)
(.2,0) -- (1.2,1)
(.4,0) -- (1.4,1)
(.6,0) -- (1.6,1);
\node at (0,-.1) {\scriptsize$1$};
\node at (.3,-.1) {\scriptsize$\cdots$};
\node at (.6,-.1) {\scriptsize$n$};
\node at (1,-.1) {\scriptsize$1$};
\node at (1.3,-.1) {\scriptsize$\cdots$};
\node at (1.6,-.1) {\scriptsize$m$};
\end{tikzpicture}\,\right)
\end{equation}
(where $F(\b)$ denotes the Rouquier complex of a braid $\b$, as usual).
\end{example}

\begin{corollary}\label{cor:intro relative half twist}
Retain the setup above.  Let $\Diag_S:=\Diag(W,S,V)$ and $\Diag_I:=\Diag(W_I,I,V)$.  Then $\HT_{S/I}$ lifts to an object of $\ZS(\KC^b(\Diag_I),{}^{\Phi}\KC^b(\Diag_S))$. Here $\Phi$ is a composition of the Dynkin diagram automorphisms for $I$ and $S$, see Theorem \ref{thm:relative HT} for details.
\end{corollary}

\begin{remark}\label{rmk:LMGRSW}
In the case $W=S_{n+m}$ with $W_I=S_n\times S_m$ an $\infty$-categorical version of Corollary \ref{cor:intro relative half twist} was proven in \cite{LMGRSW}. They work with Soergel bimodules (rather than the diagrammatic incarnation of the Hecke category) and show that the Rouquier complexes appearing in \eqref{eq:type A rel HT} are responsible for the (fully homotopy coherent) braiding in a braided monoidal $(2,2)$-category structure on the Hecke categories $\{\Diag(S_n)\}_{n\geq 0}$.

In this paper we prove an $A_\infty$ lift of Corollary \ref{cor:intro relative half twist}. We do not have a direct comparison with the work of \emph{loc.~cit.} since we adopt the setting of dg categories rather than $\infty$-categories. We stress that $A_\infty$ algebras and modules are expressed in dg categorical language (complexes, chain maps, and so on) and not in $\infty$-categorical language.
\end{remark}

\begin{remark}
During the preparation of this paper, we learned that Stroppel and Wedrich \cite{StropWed} have been independently developing a notion of $A_\infty$ twisted Drinfeld centralizers (on the level of objects), and have shown that complexes of the form \eqref{eq:type A rel HT} give examples of such.  We have coordinated the posting of our preprints in order to emphasize the independence of our approaches.
\end{remark}

The following special case of Corollary \ref{cor:intro relative half twist} is particularly important in topological applications.  Let $C_{n+1}$ be the Rouquier complex associated to the so-called \emph{Coxeter braid} $\sigma_1\cdots \sigma_n$.  Observe that $C_{n+1}=\HT_{S/I}$ where $S=\{1,\ldots,n\}$ is the set of simple reflections of $S_{n+1}$ and $I=S\setminus \{n\}$. 
\begin{cor}
\label{cor:intro coxeter}
The complex $C_{n+1}\in \KC^b(\Diag_{n+1})$ lifts to an object of $\ZS(\KC^b(\Diag_n),\KC^b(\Diag_{n+1}))$, where $\Diag_{n+1}$ has the $(\Diag_n,\Diag_n)$ $\star$-bimodule structure given by
\begin{equation}
B \cdot M\cdot B' := (\one_1\boxtimes B)\star M\star (B'\boxtimes \one_1).
\end{equation}
\end{cor}

\begin{remark}
We announced the results in this section many years ago, though this paper accrued dust on the shelf as we hoped to be able to construct the flattening functor, see the next remark. Though the long delay is regrettable, one silver lining is that in the intervening years we were able to formulate and prove $A_\infty$ versions of our main results (see \S \ref{ss:intro infty centralizer} for a summary).

In the meantime, an explicit study of conjugation by $C_{n+1}$ was performed by Mackaay, Miemietz, and Vaz \cite[\S 4.2]{MMVFlat}. They compute the action of conjugation on
morphisms, and provide a handy diagrammatic calculus which mixes ordinary Soergel calculus with a new color representing the object $C_{n+1}$. It is a very clean presentation of
the computations, so while we originally planned on including many of these explicit computations as examples, we now refer instead to \cite{MMVFlat}. In this paper we focus more
on general methods to simplify such calculations. Using our techniques, only a fraction of the computations in \cite[\S 4]{MMVFlat} are actually required. Note also that \cite{MMVFlat} primarily treats the tensoring with $C_{n+1}$ as a functor from $\Diag$ to $\KC^b(\Diag)$, and does not address the issue of lifting. \end{remark}

\begin{remark} One of the original motivations for this paper remains still out of reach: the flattening functor conjectured in \cite{EGaitsgory}. Let $W_{\ext}$ denote the
extended affine Weyl group in type $\tilde{A}_{n}$, whose braid group $\Br_{\ext}$ can be visualized as braids with $n+1$ strands on a cylinder. Let
$W_{\fin}$ denote the finite Weyl group of type $A_n$. There is a group homomorphism $\Br_{\ext} \to \Br_{\fin}$ which squashes the cylinder flat. The extended affine braid group contains a crossingless braid which rotates the strands around the cylinder, and this flattens to the Coxeter braid $C_{n+1}$.

There is also an extended affine Hecke category $\Diag_{\ext}$, with Rouquier complexes associated to each cylindrical braid (up to canonical homotopy equivalence), see \cite[\S 3,
\S 4.4]{EGaitsgory}. The technology above and in \cite{EBigraded} can be adapted to provide a painless proof that the flattening homomorphism lifts to a functor $\Flat \colon
\Diag_{\ext} \to \KC^b(\Diag_{\fin})$; this was also proven in \cite[\S 5]{MMVFlat}. However, it is far from obvious that $\Flat$ extends to a functor $\KC^b(\Diag_{\ext}) \to
\KC^b(\Diag_{\fin})$, which is conjectured in \cite[p9]{EGaitsgory}. This is a question similar to that addressed by our lifting lemma above (it is about lifting functors, not
lifting natural transformations, but the homological obstructions overlap). Unlike the functor of tensoring with a Rouquier complex, the functor $\Flat$ is not invertible and the
vanishing of negative exts is still conjectural.
 \end{remark}

\subsection{The $A_\infty$-Drinfeld centralizer}
\label{ss:intro infty centralizer}
The paper \cite{GHW} introduced a notion of derived (or dg) horizontal trace of a $\k$-linear (or dg) monoidal category, as well as a dg version of the Drinfeld center.  The paper \cite{GHW} goes on to show that the dg Drinfeld center acts on the dg horizontal trace, and suggests that this action ought to be responsible for various cabling and satellite operations in link homology.

However, the dg Drinfeld center introduced in \cite{GHW} is still too strict for many purposes. For instance we do not know of a direct proof that our favorite Rouquier complexes (for instance, full twists) admit structures in the dg Drinfeld center as it appears in \cite{GHW}.  For this reason, we are motivated to find a more flexible homotopy theoretic notion of Drinfeld center (or centralizer).  We discuss this next (see Definition \ref{def:intro infty drinfeld} and Remark \ref{rmk:GHW vs EH} for a comparison).

Recall that to lift a complex $Z$ in $\KC^b(\MS)$ to an object of $\ZS(\AS,\KC^b(\MS))$, we need to provide the data of chain maps $\tau_X$ for all $X \in \AS$ for
which \eqref{centralsquare} commutes up to homotopy for each morphism $f$ in $\AS$. However, it would be better for many purposes to include these homotopies (and appropriate higher homotopies) as part of the structure of $Z$.

To expand on this slightly, from the homotopy theory perspective, one ought to choose, for each morphism $f$, a homotopy $h_f$ which witnesses the homotopy commutativity of \eqref{centralsquare}, which we may write (abbreviating somewhat) as $d(h_f) = f\cdot \tau - \tau \cdot f$.  The chosen homotopies $h_f$ for various $f$'s should not be independent.  At the very least, the $h_f$'s should linear in $f$, so that $h_{f_1+f_2}=h_{f_1}+h_{f_2}$.   Moreover, if $f,g$ are composable morphisms then we can construct a priori two different homotopies which witness $(f\circ g)\cdot \tau \simeq \tau\cdot (f\circ g)$, namely, $h_{f\circ g}$ and $h_f\cdot g + f\cdot h_g$.  The philosophy of ``homotopy coherence'' suggests that these two different homotopies should themselves be homotopic.  In other words, for each pair of composable morphisms $f,g$ there ought to be a degree $-2$ ``higher homotopy'' $h_{f,g}$ satisfying
\[
d(h_{f,g}) = h_f\cdot g - h_{f\circ g}+ f\cdot h_g.
\]
One can continue in this fashion, obtaining (in the most desirable circumstances) a collection of higher homotopies $h_{f_1,\ldots,f_r}$, parametrized by sequences of composable morphisms. This sequence of higher homotopies can be encoded using the bar complex of $\AS$, which is an (unbounded) complex whose chain objects parametrize sequences of composable maps.
%
%of the form ought to be linear However, when $k = f \circ g$ then both $h_1 = h_k$ and $h_2 = h_f \circ g + f \circ h_g$ would make \eqref{centralsquare} commute for $k$.
%For consistency one desires that $h_1$ and $h_2$ are homotopic, and this homotopy is additional data which satisfies additional constraints.

The second main goal of this paper is to formulate a homotopy coherent version of Drinfeld centralizer, which we call the $A_{\infty}$-Drinfeld centralizer, and to extend our lifting results to this setting. For this we need a great deal of additional abstraction and categorical technology.  Readers who are content with statements which are true ``up to homotopy'' (without including all the homotopies and higher homotopies around as part of the structure) may wish to skip ahead to \S \ref{s:dg} and \S \ref{s:diag Hecke} in this paper.  For readers who embrace the homotopy coherent philosophy above, we now shift to this perspective and summarize the salient features.

\begin{remark} 
Let us mention a concrete benefit to homotopy coherence.   In the usual Drinfeld centralizer, each hom space $\Hom_{\ZS(\AS,\MS)}(Z_1,Z_2)$ is naturally a module over $\HH_0(\AS)$.  In the $A_\infty$-Drinfeld centralizer, each hom complex
\[
\Hom_{\ZS_\infty(\Ch^b(\AS),\Ch^b(\MS))}(Z_1,Z_2)
\]
is an $A_\infty$ module over the higher Hochschild homology $\HH_\bullet(\AS)$. However, each Hom space in the naive Drinfeld centralizer $\ZS(\KC^b(\AS),\KC^b(\MS))$
%
% On the other hand, forgetting the homotopies yields objects in the naive Drinfeld centralizer $\ZS(\KC^b(\AS),\KC^b(\MS))$, and each hom space in this category
% \[
% \Hom_{\ZS(\KC^b(\AS),\KC^b(\MS))}(Z_1,Z_2)
% \]
is instead a module over $\HH_0(\KC^b(\AS))$. The ring $\HH_0(\KC^b(\AS))$ is extremely unwieldy, and in general is very different from $\HH_\bullet(\AS)$.
\end{remark}

%
%
%, and inists that the relation of ``being homotopic'' is more structure than property.  Readers who identify we work in the DG setting and make stronger claims about the nature of this additional
%homotopical data, they are willing to put up with a great deal of additional abstraction and categorical nonsense. We now shift to this perspective.

Our strategy is to rephrase the usual Drinfeld center in way that strongly suggests its homotopy coherent generalization.

We begin by giving $\MS$ the structure of a category enriched in $(\AS,\AS)$-bimodules.  Given objects $Z',Z\in \MS$, an \emph{enriched morphism} $Z'\to Z$ is a triple $(X,X',f)$ where $X,X'\in \AS$ and $f$ is a morphism $Z'\star X'\to X\star Z$ in $\MS$, visualized diagrammatically as
\[
\begin{tikzpicture}
\draw[-stealth,ultra thick, electricindigo]
(0,0) ..controls ++(0,.5) and ++(0,-.5).. (1,1.5);
\draw[-stealth,ultra thick]
(1,0) ..controls ++(0,.5) and ++(0,-.5).. (0,1.5);
\filldraw[fill=white,rounded corners] (0,.5) rectangle (1,1);
\node at (.5,.75) {\scriptsize $f$};
\node at (1,-.3) {\scriptsize $X'$};
\node at (0,1.8) {\scriptsize $X$};
\node at (0,-.3) {\scriptsize $Z'$};
\node at (1,1.8) {\scriptsize $Z$};
\end{tikzpicture}
\]
The collection of enriched morphisms $Z'\to Z$ is denoted 
\begin{equation} \Hom_{\Qmod{\MS}}(Z',Z) := \bigoplus_{X,X' \in \AS} \Hom_{\MS}(Z' \star X', X \star Z), \end{equation}
and gives rise to the enriched category $\Qmod{\MS}$.

The composition of enriched morphisms, denoted $f_2\bullet f_1$, is defined as follows.  Given objects $Z_0,Z_1,Z_2\in \MS$ and $X_1,X_2,X_1',X_2'\in \AS$, and morphisms $f_i\colon Z_{i-1}\star X_i'\to Z_i$ (for $i=1,2$) we set
\begin{equation}
f_2\bullet f_1 = (\id_{X_1}\star f_2)\circ (f_1\star \id_{X_2'})
\end{equation}
or, diagrammatically,
\begin{equation}
\begin{tikzpicture}[baseline=.75cm]
\draw[-stealth,ultra thick, electricindigo]
(0,0) ..controls ++(0,.5) and ++(0,-.5).. (1,1.5);
\draw[-stealth,ultra thick]
(1,0) ..controls ++(0,.5) and ++(0,-.5).. (0,1.5);
\filldraw[fill=white,rounded corners] (0,.5) rectangle (1,1);
\node at (.5,.75) {\scriptsize $f_2\bullet f_1$};
\node at (1,-.3) {\scriptsize $X_1'\star X_2'$};
\node at (0,1.8) {\scriptsize $X_1\star X_2$};
\node at (0,-.3) {\scriptsize $Z_0$};
\node at (1,1.8) {\scriptsize $Z_2$};
\end{tikzpicture}
\quad := \quad
\begin{tikzpicture}[baseline=1.5cm]
\draw[-stealth,ultra thick, electricindigo]
(0,0) ..controls ++(0,.5) and ++(0,0).. (.75,1.5) ..controls ++(0,0) and ++(0,-.5).. (1.5,3);
\draw[-stealth,ultra thick]
(1,0) ..controls ++(0,.5) and ++(0,-.5).. (0,2)--(0,3);
\draw[-stealth,ultra thick]
(1.5,0) -- (1.5,1) ..controls ++(0,.5) and ++(0,-.5).. (.5,3);
\filldraw[fill=white,rounded corners] (0,.75) rectangle (1,1.25);
\filldraw[fill=white,rounded corners] (.5,1.75) rectangle (1.5,2.25);
\node at (.5,1) {\scriptsize $f_1$};
\node at (1,2) {\scriptsize $f_2$};
\node at (0,-.3) {\scriptsize $Z_0$};
\node at (1,-.3) {\scriptsize $X_1'$};
\node at (1.5,-.3) {\scriptsize $X_2'$};
\node at (0,3.3) {\scriptsize $X_1$};
\node at (.5,3.3) {\scriptsize $X_2$};
\node at (1.5,3.3) {\scriptsize $Z_2$};
\end{tikzpicture}
\end{equation}
As usual, if enriched morphisms are not composable then we set the composition to be zero.

There is also an left action of $\AS$ on $\Hom_{\Qmod{\MS}}(Z',Z)$ by post-composition. For example, if $f \colon Z' \star X' \to X \star Z$ and $g \colon X \to Y$ then
\begin{equation} (g \star \id_Z) \circ f \colon Z' \star X' \to Y \star Z \end{equation} is also an enriched morphism from $Z'$ to $Z$. In this sense $\Hom_{\Qmod{\MS}}(Z',Z)$ is a
left $\AS$-module. Similarly, there is a right action of $\AS$ on $\Hom_{\Qmod{\MS}}(Z',Z)$ by pre-composition. This
structure makes $\Hom_{\Qmod{\MS}}(Z',Z)$ into an $(\AS, \AS)$-bimodule.

There is a fair bit of confusion possible here because of the standard terminology in the field. This $\AS$-module structure has nothing to
do with the monoidal structure on $\AS$, only the composition structure; $\Qmod{\MS}$ is a module over $\AS$ in the usual sense (like a module over an algebra) rather than being a
$\star$-module category over a monoidal category (which is the categorification of a module over an algebra). That $\Hom_{\Qmod{\MS}}(Z',Z)$ is an $\AS$-module means that for each
$X \in \AS$ we have a direct summand $N_X$ of $\Hom_{\Qmod{\MS}}(Z',Z)$, namely $\bigoplus_{X'} \Hom_{\MS}(Z' \star X', X \star Z)$, and for each morphism $g \colon X \to Y$ we
have a corresponding map $N_X \to N_Y$. When we want to emphasize this bimodule structure, we write $\Hom_{\Qmod{\MS}}(Z',Z)$ as $\bimhom{Z}{Z'}$. We write $\bimhom{Z}{Z'}(X,X')$ to denote the summand $\Hom_{\MS}(Z' \star X', X \star Z)$ associated to $X, X' \in \AS$.

\begin{remark} \label{rmk:bimoduleargh} We make heavy use of both $(\AS,\AS)$-bimodules and of $(\AS,\AS)$ $\star$-bimodule categories in this paper, and to help the reader disambiguate, we have broken with the typical conventions of the literature and consistently called the latter $\star$-bimodule categories. \end{remark}

Just as bimodules over any ring form a monoidal category, $\Bim_{\AS,\AS}$ has a monoidal structure $\otimes_{\AS}$. The monoidal identity is the trivial bimodule $\AS$. However,
when $\AS$ is indeed a monoidal category, then the category $\Bim_{\AS,\AS}$ of $(\AS,\AS)$-bimodules inherits a second distinct monoidal structure $\diamond$ (see \S
\ref{ss:diamond}), making it a \emph{duoidal category}. See \cite{BookStreet} for more on duoidal categories.

The tensor product $\bimhom{Z_1}{Z_2}\otimes_{\AS}\bimhom{Z_3}{Z_4}$ is easy to visualize; it can be expressed as the $\k$-module formally spanned by diagrams of the form
\begin{equation} \label{eq:normaltensor}
%\bimhom{Z_1}{Z_2}\otimes_{\AS}\bimhom{Z_3}{Z_4}(X,X')
%\cong \span\left\{ 
%\bigoplus_{Y_1, Y_2}
\begin{tikzpicture}[scale=.8,xscale=-1,baseline=1.2cm]
\node at (2,-.5) {\scriptsize$Z_2$};
\node at (1,-.5) {\scriptsize$Z_4$};
\node at (0,-.5) {\scriptsize$X'$};
\node at (2,3.5) {\scriptsize$X$};
\node at (1,3.5) {\scriptsize$Z_1$};
\node at (0,3.5) {\scriptsize$Z_3$};
\node at (.7,2.1) {\scriptsize$Y_2$};
\node at (1.3,.9) {\scriptsize$Y_1$};
\begin{scope}[yshift=-.25cm]
\draw[ultra thick]
(0,0) ..controls ++(0,.5) and ++(0,-.5).. (1,1.5);
\draw[ultra thick,electricindigo]
(1,0) ..controls ++(0,.5) and ++(0,-.5).. (0,1.5);
\filldraw[fill=white,rounded corners] (0,.5) rectangle (1,1);
\node at (.5,.75) {\scriptsize $f_2$};
\end{scope}
\begin{scope}[shift={(1,1.75)}]
\draw[ultra thick,-stealth]
(0,0) ..controls ++(0,.5) and ++(0,-.5).. (1,1.5);
\draw[ultra thick,-stealth,electricindigo]
(1,0) ..controls ++(0,.5) and ++(0,-.5).. (0,1.5);
\filldraw[fill=white,rounded corners] (0,.5) rectangle (1,1);
\node at (.5,.75) {\scriptsize $f_1$};
\end{scope}
\draw[ultra thick,electricindigo]
(2,-.25) -- (2,1.75);
\draw[ultra thick,electricindigo,-stealth]
(0,1.25) -- (0,3.25);
\filldraw[fill=white,rounded corners]
(.7,1.25) rectangle (1.3,1.75);
\node at (1,1.5) {$g$};
\end{tikzpicture}
% \right)\bigg/ \sim
\end{equation}
in which $Y_1,Y_2\in \AS$ and $f\in \Hom_\AS(Y_1,Y_2)$, modulo relations
\begin{equation}\label{hohum}
(f_1 \circ (\id_{Z_2} \star g),\id_{Y_1},f_2) \sim (f_1, g, f_2) \sim (f_1, \id_{Y_2}, (g \star \id_{Z_3}) \circ f_2).
\end{equation}
That is to say, to construct $\bimhom{Z_1}{Z_2}\otimes_{\AS}\bimhom{Z_3}{Z_4}$ we take the direct sum over $Y_1, Y_2$ of the tensor products 
\[ \Hom_{\MS}(Z_2 \star Y_2,X \star Z_1) \otimes_{\k} \Hom_{\AS}(Y_1,Y_2) \otimes_{\k} \Hom_{\MS}(Z_4 \star X',Y_1 \star Z_3),\]
and quotient by the relation \eqref{hohum}. Note that this construction makes sense even though $\MS$ need not have a monoidal structure (we need not interpret this diagram as a morphism $Z_2 \star Z_4 \star X' \to X \star Z_1 \star Z_3$).

The diamond product $\bimhom{Z_1}{Z_2}\diamond \bimhom{Z_3}{Z_4}$ is obtained by taking the tensor product $\bimhom{Z_1}{Z_2}\otimes_\k \bimhom{Z_3}{Z_4}$ and using the monoidal structure on $\AS$ to restrict the $(\AS\otimes\AS,\AS\otimes\AS)$-bimodule structure to an $(\AS,\AS)$-bimodule structure.  This is harder to appreciate, and we have written \S\ref{ss:diamond} to elucidate it. 
There is a natural \emph{composition map}
$\bimhom{Z_2}{Z_1}\diamond \bimhom{Z_1}{Z_0}\to \bimhom{Z_2}{Z_0}$ whose image is spanned by of diagrams of the form
\begin{equation}
\begin{tikzpicture}[scale=.8,baseline=1.5cm]
\begin{scope}
\node at (-.4,-.25) {\scriptsize$Z_0$};
\node at (.7,1.55) {\scriptsize$Z_1$};
\node at (2.4,3.3) {\scriptsize$Z_2$};
\node at (-.4,2) {\scriptsize$X_1$};
\node at (.7,2.7) {\scriptsize$X_2$};
\node at (1.3,.3) {\scriptsize$Y_1$};
\node at (2.4,1) {\scriptsize$Y_2$};
\draw[ultra thick, electricindigo,-stealth]
(0,0) ..controls ++(0,.5) and ++(0,-.5).. (1,1.5);
\draw[ultra thick]
(1,0) ..controls ++(0,.5) and ++(0,-.5).. (0,1.5);
\filldraw[fill=white,rounded corners] (0,.5) rectangle (1,1);
\node at (.5,.75) {\scriptsize $f_1$};
\end{scope}
\begin{scope}[shift={(1,1.5)}]
\draw[ultra thick, electricindigo]
(0,0) ..controls ++(0,.5) and ++(0,-.5).. (1,1.5);
\draw[ultra thick,-stealth]
(1,0) ..controls ++(0,.5) and ++(0,-.5).. (0,1.5);
\filldraw[fill=white,rounded corners] (0,.5) rectangle (1,1);
\node at (.5,.75) {\scriptsize $f_2$};
\end{scope}
\begin{scope}[shift={(1.5,-.5)}]
\draw[ultra thick,-stealth]
(0,-.5) --(0,0);
\draw[ultra thick,-stealth]
(0,-.5)-- (0,0)
(0,0) ..controls ++(-.2,.2) and ++(0,-.2).. (-.5,.5)
(0,0) ..controls ++(.2,.2) and ++(0,-.2).. (.5,.5);
\filldraw[fill=white,rounded corners] (-.7,0) rectangle (.7,.5);
\node at (0,.25) {\scriptsize $b$};
\node at (.4,-.3){\scriptsize $Y$};
\end{scope}
\begin{scope}[shift={(.5,3.5)},yscale=-1]
\draw[ultra thick,stealth-]
(0,-.5) -- (0,0);
\draw[ultra thick,stealth-]
(0,0) ..controls ++(-.2,.2) and ++(0,-.2).. (-.5,.5)
(0,0) ..controls ++(.2,.2) and ++(0,-.2).. (.5,.5);
\filldraw[fill=white,rounded corners] (-.7,0) rectangle (.7,.5);
\node at (0,.25) {\scriptsize $a$};
\node at (.4,-.3){\scriptsize $X$};
\end{scope}
\draw[ultra thick,-stealth]
(2,0) -- (2,1.5);
\draw[ultra thick,-stealth]
(0,1.5) -- (0,3);
\draw[ultra thick, electricindigo,-stealth]
(0,-1) -- (0,0)
(2,3) -- (2,4);
\end{tikzpicture}
%
%
%\begin{tikzpicture}
%\path[-stealth,ultra thick, electricindigo]
%(0,-1.3) edge (0,.1)
%(0,0) edge (0,1.3);
%\begin{scope}[xscale=-1]
%\draw[ultra thick,-stealth]
%(-1.9,0)--(-1.2,0)..controls++(.2,.2) and ++(-.4,0)..(-.2,.5);
%\draw[ultra thick,-stealth]
%(-1.2,0)..controls++(.2,-.2) and ++(-.4,0)..(-.2,-.5);
%\end{scope}
%\draw[ultra thick,stealth-]
%(-1.9,0)--(-1.2,0)..controls++(.2,.2) and ++(-.4,0)..(-.2,.5);
%\draw[ultra thick,stealth-]
%(-1.2,0)..controls++(.2,-.2) and ++(-.4,0)..(-.2,-.5);
%\filldraw[fill=white]
%(0,.5) circle (.3cm)
%(0,-.5) circle (.3cm)
%(-1.2,0) circle (.3cm)
%(1.2,0) circle (.3cm);
%\node at (0,.5) {\scriptsize $f_2$};
%\node at (0,-.5) {\scriptsize $f_1$};
%\node at (-1.2,0) {\scriptsize $a$};
%\node at (1.2,0) {\scriptsize $b$};
%\node at (-.8,.7) {\scriptsize $X_2$};
%\node at (-.8,-.7) {\scriptsize $X_1$};
%\node at (.8,.7) {\scriptsize $Y_2$};
%\node at (.8,-.7) {\scriptsize $Y_1$};
%\node at (-2.1,0) {\scriptsize $X$};
%\node at (2.1,0) {\scriptsize $Y$};
%\end{tikzpicture}
\end{equation}
Here $b$ is a morphism $Y \to Y_1 \star Y_2$ in $\AS$, etcetera. This composition map turns $\bigoplus_{Z,Z'} \bimhom{Z}{Z'}$ into a (locally unital) algebra object in $\Bim_{\AS,\AS}$ with respect to the diamond product, and similarly, makes $E_Z := \bimhom{Z}{Z}$ into an algebra object for any given $Z \in \MS$.

The abstractions developed above pay off when we rephrase the definition of the Drinfeld centralizer. Suppose that $(Z,\tau)$ is an object of the Drinfeld centralizer $\ZS(\AS,\MS)$. For each $X\in \AS$, the element $\tau_X\colon Z\star X\to X\star Z$ may be regarded as an enriched endomorphism of $Z$, i.e.~ an element of $\End_{\Qmod{\MS}}(Z)$. Let $\ibim{\AS}$ denote the category $\AS$, viewed as an $(\AS,\AS)$-bimodule. The centrality condition \eqref{centralsquare} can be summarized by saying that we have a map of $(\AS,\AS)$-bimodules $\nu\colon \ibim{\AS}\to \End_{\Qmod{\MS}}(Z) = E_Z$ for which $\nu(\id_X)=\tau_X$. 
The multiplicativity condition is equivalent to $\tau_{X_1}\bullet \tau_{X_2} = \tau_{X_1\star X_2}$.  In other words, the map $\nu\colon \ibim{\AS}\to E_Z$ is not just a map in $\Bim_{\AS,\AS}$, but is a map of algebra objects in $\Bim_{\AS,\AS}$.

%
% To completely precise, we let $\Bim_{\AS,\AS}$ be the category of $(\AS,\AS)$-bimodules (note, here we mean bimodules in the usual sense, ignoring the monoidal structure on $\AS$).  There is, of course, the usual monoidal structure on $\Bim_{\AS,\AS}$ given by tensoring $\otimes_\AS$.  The monoidal identity is the trivial bimodule $\AS$.
%
%
%
%
% There is a second monoidal structure, denoted $\diamond$, on $\Bim_{\AS,\AS}$ which involves the monoidal structure on $\AS$ itself .
%
% \begin{remark}
% These two monoidal structures make $\Bim_{\AS,\AS}$ into a \emph{duoidal category}.  We will not go into details here, but refer the reader to \MH{ref}.  In what follows we will be most immediately interested in the $\diamond$ monoidal structure.
% \end{remark}
%
% Now, denote the set of enriched morphisms $Z'\to Z$ by $\bimhom{Z}{Z'}$, and abbreviate $E_Z:=\bimhom{Z}{Z}$.  The previous discussion defines a category $\Qmod{\MS}$ with the same objects as $\MS$, enriched in $\Bim_{\AS,\AS}$ with its $\diamond$-monoidal structure.
%

Rephrasing, this says that an object of the Drinfeld centralizer $\ZS(\AS,\MS)$ is equivalent to an $\ibim{\AS}$-module inside the bimodule enrichment $\Qmod{\MS}$. One can also verify that morphisms in the Drinfeld center are nothing more than morphisms of $\ibim{\AS}$-modules in $\Qmod{\MS}$.   This leads us the following principle:

\noindent
\begin{center}
\begin{minipage}{5in}
(*) The Drinfeld centralizer $\ZS(\AS,\MS)$ is equivalent to the category of $\ibim{\AS}$-modules inside the bimodule enrichment $\Qmod{\MS}$.
\end{minipage}
\end{center}

To define the homotopy coherent version of the Drinfeld centralizer, we make the following generalizations:
\begin{enumerate}
\item Allow $\AS$ to be an arbitrary dg category.
\item Replace the trivial bimodule $\ibim{\AS}$ by its projective resolution $\Bar_\AS$, which is an algebra object in $\Bim_{\AS,\AS}$ using the Eilenberg-Zilber shuffle product.
\end{enumerate}
\begin{definition}\label{def:intro infty drinfeld}
If $(\AS,\star,\one_\AS)$ is a dg monoidal category and $\MS$ is a dg category with the structure of an $(\AS,\AS$) $\star$-bimodule category, then we define the \emph{$A_\infty$ Drinfeld centralizer of $\AS$ in $\MS$}, denoted $\ZS_\infty(\AS,\MS)$, to be the category of $A_\infty$ modules over $\Bar_\AS$ in $\Qmod{\MS}$.
\end{definition}

For the precise definition, see \S \ref{ss:infty drinfeld}. 

\begin{remark}\label{rmk:GHW vs EH}
By contrast, the dg Drinfeld center introduced in \cite{GHW} can be identified as the category of honest (not $A_\infty$) modules over $\Bar_{\AS}$ in $\Qmod{\MS}$.
\end{remark}

In this context, our improvement of Theorem \ref{thm:intro centralizer} is the following.

\begin{theorem}\label{thm:intro infty drinfeld}[Theorem \ref{thm:infty drinfeld}]
Let $\AS$ be a $\k$-linear monoidal category and let $\MS$ an $(\AS,\AS)$ $\star$-bimodule category.  Let $Z\in \Ch^b(\MS)$ be a complex such that the hom complex $\uHom_{\MS}(Z\star X',X\star Z)$ has zero homology in negative cohomological degrees for all $X,X'\in \AS$.  Then any object $(Z,\tau)\in\ZS(\AS,\KC^b(\MS))$ lifts to a unique object of $\ZS_\infty(\AS,\Ch^b(\MS))$, which lifts to a unique object of $\ZS_\infty(\Ch^b(\AS),\Ch^b(\MS))$.  This in turn descends to an object $(Z,\hat{\tau}) \in \ZS(\KC^b(\AS), \KC^b(\MS))$.
\end{theorem}

Just like our earlier lifting theorem, this follows from a more general result, though this time we prove it in two stages. Suppose we have two functors $\FZ, \GZ \colon \AS \to
\Ch^b(\BS)$ and a natural transformation between the corresponding functors $H^0(\FZ), G^0(\FZ) \colon \AS \to \KC^b(\BS)$. The first stage is to lift this to an
\emph{$A_{\infty}$-natural transformation} from $\FZ$ to $\GZ$, a concept defined in \S\ref{ss:inftrans} using the Bar complex. The second stage is to lift this
$A_{\infty}$-natural transformation to an $A_{\infty}$-natural transformation from $\Fdg$ to $\Gdg$, where these are the induced functors $\Ch^b(\AS) \to \Ch^b(\BS)$.

The homological algebra in this paper is largely self-contained, with one exception in the second stage of this proof. The Bar complex of $\AS$ is an enormous infinite complex whose chain object in degree $-r$ is a direct sum over all $(r+1)$-tuples of objects in $\AS$. The Bar complex of $\Ch^b(\AS)$ is more enormous still, being a direct sum over tuples of objects in $\Ch^b(\AS)$. In \cite[\S 5.1]{GHW} an intermediate complex $\Bar_{\Ch^b(\AS),\AS}$ was introduced which has the size of $\Bar_{\AS}$ but is homotopy equivalent to $\Bar_{\Ch^b(\AS)}$.  It is relatively straightforward to lift constructions involving $\Bar_{\AS}$ to constructions involving $\Bar_{\Ch^b(\AS),\AS}$, which we transfer to  $\Bar_{\Ch^b(\AS)}$ using explicit formulas proved in \cite{MattsWorkInProgress} (which were outlined in \cite[\S 5.3]{GHW}).

This paper also makes a conscious effort to make explicit many of the obnoxious signs which plague this field.

{\bf Acknowledgments:} B.E.~ is supported in part by NSF grant DMS-2201387, and his research group is supported by DMS-2039316. Both authors would like to thank Catharina Stroppel and Paul Wedrich for bringing their joint work to our attention.  M.H.~would like to thank Eugene Gorsky and Paul Wedrich, as the $A_\infty$ parts of the present paper would not have been possible without the foundational work in the collaboration \cite{GHW}. % \BE{don't forget this later}

%%%%%%%%%%%%%%%%%%%%%%%%%
\section{The lifting lemma}
\label{s:dg}
%%%%%%%%%%%%%%%%%%%%%%%%%

\subsection{DG categories}
\label{ss:dg cats}

In this section we introduce the language of complexes, dg categories, and dg functors that will be used to express our main results.

Let $\k$ be a commutative ring, fixed throughout.  A complex of $\k$-modules will also be called a \emph{dg} (short for differential graded) \emph{$\k$-module}.  Let $\dgMod{\Z}{\k}$ denote the category of dg $\k$-modules and degree zero chain maps between them.    A \emph{dg category} (short for differential graded category) is a category enriched in $\dgMod{\Z}{\k}$.  In other words $\CS$ is a dg category if the homs in $\CS$ are complexes
\[
\Hom_\CS(X,Y) \in \dgMod{\Z}{\k}
\]
and composition defines a chain map
\[
\Hom_\CS(Y,Z)\otimes_\k \Hom_\CS(X,Y)\to \Hom_\CS(X,Z)
\]
The chain map condition is equivalent to the graded Leibniz rule $d(f\circ g) = d(f)\circ g +(-1)^{|f|}f\circ d(g)$ where as usual $|f|\in \Z$ denotes the degree of $f$.  

The classic examples of dg categories are categories of complexes $\Ch^b(\AS)$ over a $\k$-linear category $\AS$.  Objects of this category are complexes $(X,\d_X)$ with the cohomological convention for differentials, as in
\[
\cdots \buildrel \d\over\rightarrow X^k \buildrel \d\over \rightarrow X^{k+1} \buildrel \d\over\rightarrow  \cdots .
\]
We frequently abuse notation and write $X=(X,\d_X)$.  For any two complexes $X$ and $Y$ and $k \in \Z$ let
\begin{equation}\label{eq:uHom}
\uHom_\AS^k(X,Y):= \prod_{i\in \Z} \Hom_{\AS}(X^i, Y^{i+k})
\end{equation}
be the space of linear maps of degree $k$, which we also call \emph{morphisms} of degree $k$ (they are not chain maps). Then
\begin{equation}
\Hom_{\Ch^b(\AS)}(X,Y):=\uHom_\AS(X,Y) := \bigoplus_{k \in \Z} \uHom_\AS^k(X,Y)
\end{equation}
is the Hom complex between $X$ and $Y$, equipped with the standard differential $d(f):=\d_Y\circ f - (-1)^{|f|}f\circ \d_X$. We adopt the convention that the symbol `$\d$' will be used for the differential on objects of $\Ch^b(\AS)$, and `$d$' will be reserved for the differential on hom complexes.  
 
% \begin{remark} In order for $\Ch^b(\AS)$ to be a dg category, the entire complexes $\uHom_{\AS}(X,Y)$ are the morphism spaces, not just the chain maps (which can be recovered as the kernel of $d$ in degree zero). \end{remark}

\begin{remark}\label{rmk:additive is dg}
Any $\k$-linear category $\AS$ can be regarded as a dg category with trivial grading and zero differential. We use this tacitly whenever we refer to $\AS$ as a dg category.
\end{remark}
 
\begin{remark}\label{rmk:other gradings}
In the definitions above we can also consider dg categories with gradings living in a more general abelian group $\Gamma$.  To formulate the definition one needs to choose a symmetric bilinear pairing $\ip{-,-}\colon \Gamma\to \Z/2\Z$, which is reponsible for the Koszul sign rule, and an element $\iota\in \Gamma$ with $\ip{\iota,\iota}=1$, which is the degree of differentials.

When we start discussing complexes over diagrammatic Hecke categories, we will take $\Gamma=\Z\times \Z$ (with gradings given by the cohomological degree and Soergel degree, respectively) and $\ip{(i,j),(i',j')} =ii'$ (mod 2), with degree of differentials given by $\iota = (1,0)$. %\MH{$q,t$ or $(t,q)$?}
\end{remark}

Finally, recall the terminology that a morphism $f\in \Hom_\CS(X,Y)$ in a dg category $\CS$ is called \emph{closed} if $d(f)=0$ and \emph{exact} if $f=d(h)$ for some $h\in \Hom_\CS(X,Y)$.  Closed morphisms $f,g$ are called \emph{homotopic}, written $f\simeq g$, if $f-g$ is exact. In the context of $\Ch^b(\AS)$, degree zero closed morphisms are the same thing as chain maps, and exact morphisms are the same thing as null-homotopic chain maps.
% In the context of $\Ch^b(\AS)$, the degree zero closed morphisms in $\Hom_{\Ch^b(\AS)}(X,Y)$ are the same thing as chain maps $X\to Y$, and exact morphisms are the same thing as null-homotopic chain maps.

For a dg category $\CS$, let $Z^0(\CS)$ (resp.~$H^0(\CS)$) be the category with the same objects as $\CS$, and morphisms given by the degree zero closed morphisms (resp.~ degree zero closed morphisms modulo homotopy). For a $\k$-linear category $\AS$ we write $\KC^b(\AS):=H^0(\Ch^b(\AS))$ and refer to this as the  \emph{homotopy category of bounded complexes over $\AS$}.  Isomorphism in $\KC^b(\AS)$ is called \emph{homotopy equivalence}, writen $\simeq$.  The homotopy category is a triangulated category, not a dg category.

\begin{remark} To avoid an undue abundance of underlines, we use the following conventions. When working with objects $X, Y$ in a dg category $\CS$, we write $\Hom(X,Y)$ or $\Hom_{\CS}(X,Y)$ for the morphism space in this dg category, which is a complex. If we wish to examine the homology groups of this complex, we will write $H^k(\Hom(X,Y))$. However, when this dg category is the category of complexes over an additive category $\AS$, we write $\uHom(X,Y)$ to help disambiguate this hom complex from other potential hom spaces (e.g. $\Hom_{\AS}(X^i, Y^j)$) which might not be complexes. Again, we do not need underlines for general dg categories because all morphism spaces are complexes, and there is nothing to disambiguate. \end{remark}

%------------------------
\subsection{Shifts and twists}
\label{ss:setup}
%------------------------

%
%To introduce a little bit of sanity, we typically refer to the differential in a Hom complex by the symbol $d$, while referring to the differentials in $X$ and $Y$ by the symbols $\d_X$ and $\d_Y$.
%
%The chain maps in $\Hom_{\(\AS)}(X,Y[k])$ agree with the closed degree $k$ morphisms in $\uHom(X,Y)$, and the nulhomotopic chain maps agree with the boundaries.

Let $\AS$ be a $\k$-linear category.  For $k\in \Z$ and $X\in \Ch^b(\AS)$, we let $X[k]$ denote the complex obtained by shifting $X$ to the \emph{left} by $k$-units, i.e.~ $X[k]^i = X^{k+i}$.  By convention the differential on $X[k]$ comes with a sign $\d_{X[k]} = (-1)^k \d_X$.  Note that chain maps $X\to Y[k]$ may be identified with closed degree $k$ elements of $\uHom(X,Y)$.

% \begin{remark}\label{remark:shift}
% It is the generally accepted convention that $[1]$ shifts complexes one unit ``against the grain'' of the differential.  For the homological convention for differentials this means that $[1]$ shifts one unit up, while for the cohomological convention this shifts one unit down.  In particular if $A\in \AS$ an object (thought of as a complex supported in degree zero) and we wish to shift it so that it lies in degree $k$ then we map $A\mapsto A[-k]$.
% \end{remark}

\begin{remark}\label{remark:shift}
It is the generally accepted convention that $[1]$ shifts complexes one unit ``against the grain'' of the differential.  Since we use the cohomological convention, $[1]$ shifts complexes one unit left. In particular if we view an object $A\in \AS$ as also being a complex supported in degree zero, then $A[-k]$ is supported in degree $+k$.
\end{remark}

If $(X,\d_X)$ and $(X,\d_X')$ are complexes with \emph{the same} underlying chain groups, then let $\a = \d_X' - \d_X$.
We may write the latter as $(X,\d_X+\a)$ and call it a \emph{twist} of $(X,\d_X)$.  We will use the following notation for twists:
\begin{equation}
(X,\d_X+\a) =: \tw_\a(X).
\end{equation}

\begin{remark}
With this notion in place, we can now write every complex as a twist of a complex with zero differential, namely 
\begin{equation} X=\tw_{\d_X}(\bigoplus_{k\in \Z} X^k[-k]). \end{equation}
\end{remark}

\begin{remark}\label{remark:bicomplexes}
Suppose $\{X^{i,j},\d_1,\d_2\}$ is a bicomplex over $\AS$.  I.e.~each $X^{i,j}$ is an object of $\AS$, and $\d_1\colon X^{i,j}\to X^{i+1,j}$, $\d_2\colon X^{i,j}\to X^{i,j+1}$ are maps such that $\d_1^2=0=\d_2^2=\d_1\d_2+\d_2\d_1$.  Then the \emph{columns} of this bicomplex are the complexes $X^{\Z,j}$ with differential given by $\d_1$, and the total complex can be expressed as
\begin{equation}
\Tot(\{X^{i,j}\},\d_1,\d_2) = \tw_{\d_2}\left(\bigoplus_{j\in \Z} X^{\Z,j}[-j]\right)
\end{equation}
Of course we can also express this total complex in terms of the \emph{rows} $X^{i,\Z}$, as
\begin{equation}
\Tot'(\{X^{i,j}\},\d_1,\d_2) = \tw_{\d_1}\left(\bigoplus_{i\in \Z} X^{i,\Z}[-i]\right).
\end{equation}
The two kinds of total complexes are isomorphic, via an isomorphism sending $X^{i,j}\to X^{i,j}$ by $(-1)^{ij}\id$.
\end{remark}

%------------------------
\subsection{Lifting DG functors}
\label{ss:lifting}
%------------------------

If $\CS$ and $\DS$ are dg categories then a \emph{dg functor} is a mapping on objects $F\colon \Obj(\CS)\to \Obj(\DS)$, and a collection of degree zero chain maps $F\colon \Hom_\CS(X,Y)\to \Hom_\DS(F(X),F(Y))$ between hom complexes, which are functorial in that $F(f \circ g) = F(f) \circ F(g)$ and $F(\id_X) = \id_{F(X)}$.  The action of a dg functor on morphisms satisfies $\deg(F(f))=\deg(f)$ and $d_\DS(F(g)) = F(d_\CS(f))$.  A \emph{degree $l$ natural transformation} between dg functors $F,G\colon \CS\to \DS$ is a family of morphisms $\phi_X\in \Hom^l_{\DS}(F(X),G(X))$ such that $G(f)\circ \phi_{X'} = (-1)^{l|f|} \phi_X\circ F(f)$ for all $f\in \Hom_\CS(X',X)$.   Natural transformations of all degrees form a complex $\uHom(F,G)$, where the differential sends $\{\phi_X\}_{X\in \CS}$ to $\{d(\phi_X)\}_{X\in \CS}$.
%The differential of the natural transformation $\{\phi_X\}_{X\in \CS}$ is the natural transformation $\{d(\phi_X)\}_{X\in \CS}$.

%Let $\AS$ and $\BS$ be $\k$-linear categories.  A \emph{dg functor} from $F\colon \Ch^b(\AS)\to \Ch^b(\BS)$ is a mapping on objects $F\colon \Obj(\Ch^b(\AS))\to \Obj(\Ch^b(\BS))$ and hom complexes $F\colon \uHom_{\AS}(X,Y)\to \uHom_{\BS}(F(X),F(Y))$ (this is to be a degree zero chain map between hom complexes for each $X,Y\in \Obj(\Ch^b(\AS))$) satisfying the usual compatibility with identity endomorphisms and composition.

Now, let $\AS,\BS$ be $\k$-linear categories.  Consider dg functors from $\AS$ (regarded as a dg category with trivial grading and differential) to $\Ch^b(\BS)$. Any such dg functor must send every morphism in $\AS$ to a closed degree zero morphism in $\Ch^b(\BS)$.  Said differently, a dg functor $\FZ:\AS\rightarrow \Ch^b(\BS)$ is the same thing as an ordinary $\k$-linear functor $\FZ\colon \AS\to Z^0(\Ch^b(\AS))$.

The purpose of this section is to review the constructions involved in the following standard result.

\begin{prop} \label{prop:liftfunctor}
There is an equivalence of categories between the category of dg functors $\AS \to \Ch^b(\BS)$, and the category of dg functors $\Ch^b(\AS) \to \Ch^b(\BS)$.
\end{prop}

Any dg functor $\FZ:\AS\rightarrow \Ch^b(\BS)$ has a dg functor extension $\Fdg:\Ch^b(\AS)\rightarrow\Ch^b(\BS)$. On the level of objects $\Fdg$ is defined as follows. Given an object $X$ of $\Ch^b(\AS)$, written out as
\[
\cdots \buildrel \d\over\rightarrow X^k \buildrel \d\over \rightarrow X^{k+1} \buildrel \d\over\rightarrow  \cdots ,
\]
we define $\Fdg(X)$ to be the total complex of the bicomplex
\[
\cdots \buildrel \FZ(\d)\over\rightarrow \FZ(X^k) \buildrel \FZ(\d)\over \rightarrow \FZ(X^{k+1})\buildrel \FZ(\d)\over\rightarrow \cdots.
\]
In other words, the $k$-th chain group of $\Fdg(X)$ is $\bigoplus_{i+j=k} \FZ(X^i)^j$, and the differential is defined componentwise to be the sum of maps of the form
\[
\FZ(\d_X^i)^j : \FZ(X^i)^j \rightarrow \FZ(X^{i+1})^j \qquad \text{ and } \qquad 
(-1)^i \d_{\FZ(X^i)}^j :  \FZ(X^i)^j\rightarrow \FZ(X^i)^{j+1}.
\]
The action of $\Fdg$ on morphisms between complexes in $\Ch^b(\AS)$ is defined componentwise: given $f\in \uHom(X,Y)$ we define $\Fdg(f)\in\uHom(\Fdg(X),\Fdg(Y))$ by
\[
\Fdg(f)|_{\FZ(X^i)^j} = \FZ(f|_{X^i})^j.
\]
We leave it to the reader to check that $f \mapsto \Fdg(f)$ defines a degree zero chain map $\uHom(X,Y)\to \uHom(\Fdg(X),\Fdg(Y))$.  Clearly $\Fdg(f\circ g) = \Fdg(f)\circ \Fdg(g)$, so that $\Fdg$ is a dg functor $\Ch^b(\AS)\to \Ch^b(\BS)$.

More compactly (see Remark \ref{remark:bicomplexes}), we can say that
\begin{equation} \Fdg(X) = \tw_{\FZ(\d)} \left( \bigoplus_i \FZ(X^i)[-i] \right). \end{equation}

Conversely, if $\Fdg:\Ch^b(\AS)\rightarrow \Ch^b(\BS)$ is any dg functor, then $\Fdg$ is completely determined by its restriction to the
additive category $\AS$ inside $\Ch^b(\AS)$.  The reason for this is that each of the relations $Y=X[1]$, $Y=X_1\oplus X_2$, and $Y=\tw_\a(X)$ is ``equational'' and hence preserved by any dg functor.  Let us explain.

If we have complex $X$, then how can we characterize $X[1]$?  Note that the identity map of $X$ may be regarded as a degree $-1$ map $\phi\colon X\to X[1]$, which is closed because
\begin{equation}
\d_{X[1]}\circ \id_{X} + \id_{X[1]}\circ \d_{X} = 0
\end{equation}
thanks to the sign rule for shifts.
An inverse to $\phi$ is constructed in the same way.  We can characterize $X[1]$ (up to isomorphism) abstractly as a pair $(Y,\phi)$ in which $Y$ is a complex and $\phi\colon X\to Y$ is a closed, degree $-1$, invertible map.  Given any such pair $(Y,\phi)$ we have $Y\cong X[1]$ canonically (composing the canonical map $X[1]\to X$ with the given map $X\to Y$).  Now, if $\Fdg\colon \Ch^b(\AS)\to \Ch^b(\BS)$ is any dg functor, and $(Y,\phi)$ satisfies the unique characterization of $X[1]$, then $(\Fdg(Y),\Fdg(\phi))$ satisfies the unique characterization of $\Fdg(X)[1]$. In other words,
\begin{equation}
\Fdg(X[1])\cong \Fdg(X)[1],
\end{equation}
canonically.

In a similar fashion, we can give a unique characterization of twists.  To give this unique characterization, let $\tw_\a(X)$ be a twist of $X$.  Let $\phi\colon X\to \tw_\a(X)$ be the degree zero, invertible, but \emph{not closed} map given by $\id_X$.  The differential of this map is
\begin{equation}
(\d_X+\a)\circ \id - \id\circ \d_X = \a.
\end{equation}
The suggests the following unique characterization of $\tw_\a(X)$, as a pair $(Y,\phi)$ where $Y$ is a complex and $\phi\colon X\to Y$ is a degree zero invertible map with $d(\phi)=\phi \circ \a$.   Indeed (we leave the verification as an exercise), this does characterize $\tw_\a(X)$ uniquely in the sense that if $(Y,\phi)$ and $(Y',\phi')$ are two such pairs (for the same $\a$), then $\phi'\circ \phi\inv \colon Y\to Y'$ is a closed degree zero isomorphism.  Any dg functor preserves twists, since if $(Y,\phi)$ satisfies the unique characterization of $\tw_\a(X)$ then $(\Fdg(Y),\Fdg(\phi))$ satisfies the unique characterization of $\tw_{\Fdg(\a)}(\Fdg(X))$.  In other words
\begin{equation}
\tw_{\Fdg(\a)}(\Fdg(X))\cong \Fdg(\tw_\a(X))
\end{equation}
canonically.

Finally, finite direct sums are also characterized equationally, using the projections and inclusions onto the summands.  Hence any $\k$-linear functor (and in particular any dg functor) will preserve finite direct sums.

Since any complex in $\Ch^b(\AS)$ is a twist of a finite direct sum of shifts of objects of $\AS$, it follows that any dg functor $\Fdg\colon \Ch^b(\AS)\to \Ch^b(\BS)$ is determined by its restriction $\Fdg|_\AS\colon \AS\to \Ch^b(\BS)$. That is to say,
\begin{equation}
\Fdg\left(\tw_{\d_X}\left(\bigoplus_{k\in \Z} X^k[-k]\right)\right) \ \cong \ \tw_{\Fdg(\d_X)} \left(\bigoplus_{k\in \Z} \Fdg(X^k)[-k]\right),
\end{equation}
canonically.

Now we discuss the lifting of natural transformations.  Supose $\FZ,\GZ:\AS\rightarrow \Ch^b(\BS)$ are two dg functors and $\etaZ :\Fdg\rightarrow \Gdg$ is a degree $l$ natural transformation.  In other words, for each object $X\in \AS$ we have a degree $l$ morphism $\etaZ_X:\FZ(X)\rightarrow \GZ(X)$ in $\Ch^b(\BS)$ (not necessarily commuting with the differentials) such that $\etaZ_Y\circ \FZ(f) = \GZ(f)\circ \etaZ_X$ for all morphisms $f\in \Hom_\AS(X,Y)$.

 We may lift $\etaZ$ to a degree $l$ natural transformation $\etadg:\Fdg\rightarrow \Gdg$ defined component-wise.  In other words, given a complex
\[
\cdots \buildrel \d\over\rightarrow X^k \buildrel \d\over \rightarrow X^{k+1} \buildrel \d\over\rightarrow \cdots
\]
in $\Ch^b(\AS)$ we define $\eta_X:\Fdg(X)\rightarrow \Gdg(X)$ to be the morphism whose component $\Fdg(X^i)^j \rightarrow \Gdg(X^i)^{j+l}$ is the $j$-th component of $\etaZ_{X^i}:\FZ(X^i)\rightarrow \GZ(X^i)$.

The dg lift $\FZ \mapsto \Fdg$ respects composition of functors up to canonical isomorphism, and likewise at the level of natural transformations, the dg lift $\etaZ\mapsto \etadg$ respects
horizontal and vertical composition of natural transformations.

%------------------------
\subsection{Monoidal structures}
\label{ss:tensoring with objects}
%------------------------

This section is a brief interlude which we discuss the operations of shift $[k]$ and twist $\tw_\a$, and how they will interact with a monoidal structure on $\AS$.
%discusses the connection between the two different ways of viewing the tensor product as a dg functor.

Let $\AS$ be a $\k$-linear additive monoidal category with tensor product denoted $\otimes$ and monoidal identity $\one\in \AS$. Then $\Ch^b(\AS)$ inherits the structure of a (dg)
monoidal category from $\AS$.  This monoidal structure is completely determined by the following rules:
\begin{equation}
\left(\bigoplus_{i\in I} X_i\right)\otimes \left(\bigoplus_{j\in J} Y_j\right) = \bigoplus_{(i,j)\in I\times J} X_i\otimes Y_j \, , \qquad X[k]\otimes Y[l] = (X\otimes Y)[k+l] \,
\end{equation}
and
\begin{equation}
\tw_{\a}(X)\otimes \tw_{\b}(Y)=\tw_{\a\otimes\id_Y+\id_X\otimes \b}(X\otimes Y),
\end{equation}
together with the Koszul sign rule for tensoring morphisms: given $f\in \uHom(X,X')$ and $g\in \uHom(Y,Y')$ we have
\begin{equation} \label{koszulsignrule}
(f\otimes g)|_{X^i\otimes Y^j} := (-1)^{i|g|}f|_{X^i}\otimes g|_{Y^j}.
\end{equation}
This implies the explicit rule for the differential on a tensor product of complexes given in \eqref{diffonYotX}.

We have the usual sign when composing tensor products:
\begin{equation} \label{signedinterchange}
(f\otimes g)\circ (f'\otimes g') = (-1)^{|g||f'|} (f\circ f')\otimes (g\circ g').
\end{equation}

\begin{remark} The formula \eqref{signedinterchange} implies that $\Ch^b(\AS)$ satisfies the signed interchange law, rather than the ordinary interchange law, whence $\Ch^b(\AS)$ is a dg monoidal category, not a monoidal category. We will typically omit the word dg in ``dg monoidal,'' leaving it as understood. \end{remark}

\begin{remark} \label{remark:signsfortensoringwithobjects}
Let $Y\in \Ch^b(\AS)$ be given.  We can tensor $Y$ with objects of $\AS$, obtaining a dg functor $L_Y\colon \AS\to \Ch^b(\AS)$ given by $L_Y=(Y\otimes -)_0$ (the subscript reminds us that we are tensoring $Y$ with objects, not complexes).  We can extend $L_Y$ to a functor $\Ch^b(\AS)\to \Ch^b(\AS)$ in two ways.  First, we can simply take the functor $\mathbf{L}_Y' = Y\otimes -$, which on objects sends $X\mapsto Y\otimes X$ and on morphisms sends $f\mapsto \id_Y\otimes f$.   We also have the functor $\mathbf{L}_Y$, which is the extension to complexes as constructed in \S \ref{ss:lifting}.  We would like to compare $\mathbf{L}_Y$ and $\mathbf{L}'_Y$.  Some signs are involved in the commutation of $Y \ot (-)$ with suspensions, making this application of Proposition \ref{prop:liftfunctor} somewhat subtle.

The differential on $\mathbf{L}_Y'(X) = Y\otimes X$ is a sum of its components
\begin{equation} \label{diffonYotX}
Y^j\otimes X^i \buildrel \d_Y\otimes \id\over \longrightarrow Y^{j+1}\otimes X^i ,\qquad\qquad Y^j\otimes X^i \buildrel (-1)^j\Id\otimes \d_X\over \longrightarrow Y^{j}\otimes X^{i+1}.
\end{equation}
Compare this with the extension of $\mathbf{L}|_\AS$ to complexes, applied to $X$, which yields
\begin{equation}\tw_{\id_Y\otimes \d_X}\left(\bigoplus_{k\in \Z} Y\otimes X^k[-k]\right).\end{equation}
%the total complex 
%\[
%\cdots \rightarrow Y\otimes X^i \rightarrow Y\otimes X^{i+1}\rightarrow \cdots
%\]
%in which the horizontal differential is $\id_Y\otimes \d_X$.  The differential on the total complex is a sum of its componentsined as follows. Given an object $X$ of $\Ch^b(\AS)$, written out as
% \[
% \cdots \buildrel \d\over\rightarrow X^k \buildrel \d\over \rightarrow X^{k+1} \buildrel \d\over\rightarrow  \cdots ,
% \]
The differential on this complex is the sum of its components
\begin{equation}
Y^j\otimes X^i \buildrel (-1)^{\textcolor{red}{i}}\d_Y\otimes \id\over \longrightarrow Y^{j+1}\otimes X^i ,\qquad\qquad Y^j\otimes X^i \buildrel \id\otimes \d_X\over \longrightarrow Y^{j}\otimes X^{i+1}.
\end{equation}
Thus, $Y\otimes X$ and $\tw_{\id_Y\otimes \d_X}\left(\bigoplus_{k\in \Z} Y\otimes X^k[-k]\right)$ are different complexes.  Nonetheless, they are isomorphic by the chain map $\psi:\mathbf{L}_Y'(X) \rightarrow \mathbf{L}_Y(X)$ defined componentwise by $\psi|_{Y^j\otimes X^i} = (-1)^{ij}\id_{Y^j\otimes X^i}$.  This isomorphism is natural in $Y$ and $X$ (naturality with respect to degree zero morphisms is clear; naturality with respect to arbitary morphisms is also true, but involves a tedious check of signs), so the functor $(Y\otimes -)$ is naturally isomorphic (but not \emph{equal}) to the functor obtained by lifting $(Y \ot -)_0$ from objects to complexes.
\end{remark}

\begin{remark}
In contrast to the situation for tensoring with $Y$ on the left, there are in fact no signs in the natural isomorphism relating $(-\otimes Y)$ and the lift of $(-\otimes Y)_0$.
\end{remark}

%In summary, we have the following lemma.
%
%\begin{lemma}
%For each $Y\in \Ch^b(\AS)$ the functors $\Ch^b(\AS)\rightarrow \Ch^b(\AS)$ given by $(-) \otimes Y$ and $Y\otimes (-)$ commute with direct sums, suspensions, and twists.  In particular both functors are completely determined by their restrictions to $\AS\rightarrow\Ch^b(\AS)$ using the construction from \S \ref{ss:lifting}, although the sign isomorphism $\psi$ is required to match $Y \ot (-)$ with the lift of $(Y \ot -)_0$. \qed
%\end{lemma}

%For the purposes of getting the correct signs, it is useful to view suspension as tensoring on the left with $\one[1]$.  Indeed, the Koszul sign rule for tensoring morphisms tells us that the differential on $\one[i]\otimes X$ differs from the differential on $\one\otimes X\cong X$ by the sign $(-1)^i$.  On the other hand tensoring with $\one[i]$ on the right does not introduce signs.

%Now, we view $X$ as $\bigoplus_{k\in \Z} \one[-k]\otimes X^k$ with a twist.  Thus, $X\otimes Y$ is $\bigoplus_{k\in \Z} \one[-k]\otimes X^k\otimes Y$ with a twist.

%On the other hand, tensoring with $Y$ on the left gives $\bigoplus_{k\in \Z} Y\otimes \one[-k]\otimes X^k$ with a twist.  In order to identify this with $\bigoplus_{k\in \Z} \one[-k]\otimes Y\otimes X^k$ with a twist, it is necessary to use the isomorphisms morphisms which commute $\one[-k]$ past $Y$.  This isomorphism multiplies the chain object $Y^l$ by $(-1)^{kl}$.  Thus, the functor $Y\otimes -$ is not \emph{equal} to the functor obtained

\subsection{Lifting natural transformations}
\label{ss:lifting2}

\begin{definition}\label{def:unobstructed}
Let $\AS$, $\BS$ be $\k$-linear categories, and let $F,G\colon \AS\to \Ch^b(\BS)$ be dg functors.  We say that the pair $(F,G)$ is \emph{unobstructed} if
\begin{equation}
H^k(\uHom_{\BS}(F(X'),G(X)))=0
\end{equation}
for all $k<0$ and all objects $X,X'\in \AS$.
\end{definition}

\begin{thm}\label{thm:lifting lemma}
Let $F,G\colon \AS\to \Ch^b(\BS)$ be dg functors, with lifts $\Fdg,\Gdg\colon \Ch^b(\AS)\to \Ch^b(\BS)$.  If $F,G\colon \AS\to \Ch^b(\BS)$ are unobstructed in the sense of Definition \ref{def:unobstructed}, then any natural transformation $\etaZ:H^0(F)\rightarrow H^0(G)$ lifts to a unique natural transformation $\etadg: H^0(\Fdg)\to H^0(\Gdg)$.
\end{thm}

We prove this theorem in the next section. The main tool used in the proof is the following abstract homological algebra lemma.  To state it, let $V$ be a complex of $\k$-modules equipped with a filtration
\[
V = \FC^0 \supset \FC^1\supset\cdots \ , 
\]
such that
\begin{equation} \label{filtconditions}
\bigcap_{k\geq 0} \FC^k = 0 \ ,\qquad \quad V = \lim_{k\to \infty} V / \FC^k.
\end{equation}
The second of these equations means that if $v_k \in \FC^k$ are arbitrary elements with $\deg v_k = l$ for all $k$ (they have the same degree), then the infinite sum
\[
v_0+v_1+\cdots
\]
is a well-defined element of $V$, also of degree $l$.  For instance, if $\FC^k=0$ for all $k$ sufficiently large, then these conditions are automatically satisfied.

\begin{lemma}\label{lemma:filtered cx0}
Retain notation above.  Assume that the following homology vanishing condition is met:
\begin{equation}
H^i(\FC^k/\FC^{k+1})\cong 0 \qquad \text{for}\qquad k\geq 1, \qquad i=0,1.
\end{equation}
Then the projection $V\rightarrow V/\FC^1$ induces an isomorphism $H^0(V)\rightarrow H^0(V/\FC^1)$.
\end{lemma}

\begin{proof}
We first show that any element of $Z^0(V/\FC^1)$ is in the image of $Z^0(V)$. Let $c_0\in V$ be a degree zero element such that $d(c_0)\in \FC^1$.  Then $-d(c_0)$ is a cycle in $\FC^1$ in degree one, descending to a cycle in $\FC^1/\FC^2$. Since $H^1(\FC^1/\FC^2)\cong 0$ there exists an element $c_1\in \FC^1$ in degree zero, such that $-d(c_0) = d(c_1)$ modulo $\FC^2$, or equivalently, $d(c_0 + c_1) \in \FC^2$.  But then $-d(c_0+c_1)$ is a cycle in $\FC^2$.  Since $H^1(\FC^2/\FC^3)\cong 0$ there exists $c_2\in \FC^2$ in degree zero, such that $d(c_0+c_1+c_2)\in \FC^3$.  Continuing in this way, we construct a sequence of elements $c_k\in \FC^k$ such that
\begin{equation}
d(c_0+c_1+c_2+\cdots + c_k)\in \FC^{k+1}.
\end{equation}
The infinite sum $c:=c_0+c_1+c_2+\cdots$ exists in $V$, and satisfies $d(c)\in \FC^k$ for all $k\geq 0$.  Thus $d(c)=0$ since $\cap_{k\geq 0}\FC^k$ is assumed to be zero. Since $c$ and $c_0$ agree modulo $\FC^1$, we have constructed the desired lift.

Since $Z^0(V) \to Z^0(V/\FC^1)$ is surjective, we immediately deduce that $H^0(V)\rightarrow H^0(V/\FC^1)$ is surjective.  

Now suppose that $c \in V$ is a degree zero element with $d(c) = 0$, whose image in $V / \FC^1$ is a boundary. Then there exists $h_0 \in V$ of degree minus one, such that $c - d(h_0) \in \FC^1$. Observe that $c - d(h_0)$ is a cycle in $\FC^1$. Since $H^0(\FC^1/\FC^2) \cong 0$, there exists $h_1 \in \FC^1$ of degree minus one, such that $c - d(h_0) - d(h_1) \in \FC^2$. Continuing in this fashion, we construct $h_k \in \FC^k$ for $k \ge 1$, such that
\begin{equation} c - d(h_0 + h_1 + \cdots + h_k) \in \FC^{k+1}.\end{equation}
The infinite sum $h:= h_0+ h_1 + \cdots$ exists in $V$ and satisfies $c - d(h) \in \FC^{k+1}$ for all $k \ge 0$. This forces
\[ c - d(h) = 0.\]
Thus $c$ is a boundary in $V$, and the map in homology $H^0(V)\rightarrow H^0(V/\FC^1)$ is injective.
\end{proof}

% Now, suppose that $c=\hat{c}_0$ and $c'=\hat{c}_0'$ are lifts of $c_0$, $c_0'$, and $c_0\simeq c_0'$ in $V/\FC^1$.   Then $c-c'$ are homotopic in $V/\FC^1$.  Thus, there exists $h_0\in V$ such that $c-c'- d(h_0) \in \FC^1$.   Observe that $c,c'$ are cycles, hence $c-c'-d(h_0)$ is a degree zero cycle in $\FC^1$.  Under the assumption $H^0(\FC^1/\FC^2)\cong 0$, there exists $h_1$ such that $c-c' - d(h_0+h_1)\in \FC^2$.  Continuing in this fashion, we construct $h_k\in \FC^k$ for $k\geq 1$, such that
% \[
% c-c' - d(h_0+h_1+\cdots+h_k) \in \FC^{k+1}.
% \]
% The infinite sum $h:=h_0+h_1+\cdots$ exists in $A$ and satisfies $c-c'-d(h)\in \FC^{k+1}$ for all $k\geq 0$.  This forces
% \[
% c-c' - d(h) = 0.
% \]
% Thus, $c\simeq c'$.  This proves that the map in homology $H^0(V)\rightarrow H^0(V/\FC^1)$ is injective.
% \end{proof}

The proof above actually establishes the following.

\begin{lemma}\label{lemma:filtered cx}
Retain notation above.  Fix an integer $l$.  If
\begin{equation}
H^1(\FC^k/\FC^{k+1})\cong 0 \qquad \text{for}\qquad k\geq l% \qquad H^0(\FC^k/\FC^{k+1})\cong 0 \qquad \text{for}\qquad k\geq 2.
\end{equation}
then the map $H^0(V)\rightarrow H^0(V/\FC^l)$ is surjective.  If
\begin{equation}
H^0(\FC^k/\FC^{k+1})\cong 0 \qquad \text{for}\qquad k\geq l% \qquad H^0(\FC^k/\FC^{k+1})\cong 0 \qquad \text{for}\qquad k\geq 2.
\end{equation}
then the map $H^0(V)\rightarrow H^0(V/\FC^l)$ is injective.\qed
\end{lemma}

\subsection{Proof of Theorem \ref{thm:lifting lemma}}
\label{subsec:proof1}

We begin with a general construction, a filtration on Hom spaces to which we can apply the lemmas of the previous section.

\begin{lemma} \label{lemma:filtered hom cx} Let $F,G:\AS\rightarrow \Ch^b(\BS)$ be dg functors, and let $\Fdg, \Gdg \colon \Ch^b(\AS)\to \Ch^b(\BS)$ be their lifts. Then for any bounded complexes $X, Y \in \Ch^b(\AS)$, the complex $\uHom(\Fdg(X),\Gdg(Y))$ has subcomplexes $\FC^k\Big(\uHom(\Fdg(X),\Gdg(Y)\Big)$ for all $k \in \Z$ (constructed in the proof), satisfying $\FC^k \supset \FC^{k+1}$ and
\begin{equation} \label{howassgrworkslemma}
\FC^k/\FC^{k+1} \cong \prod_{i\in \Z}\uHom(F(X^i), G(Y^{i+k}))[-k]
\end{equation}
as complexes. Moreover, $\FC^k = 0$ for $k$ sufficiently large.
\end{lemma}

It is not the case that $\uHom(\Fdg(X),\Gdg(Y)) = \FC^0\Big(\uHom(\Fdg(X),\Gdg(Y)\Big)$.

\begin{proof}
Given a complex $X\in \Ch^b(\AS)$, we regard $\Fdg(X)$ as a filtered complex, filtered by the homological degree internal to $X$.  That is to say, $X$ is filtered by its subcomplexes $X^{\geq i}:= \tw_{\d_X}(\bigoplus_{j\geq i} X^j[-j])$, hence $\Fdg(X)$ is filtered by subcomplexes $\Fdg(X^{\geq i})$.

Now fix complexes $X$ and $Y$ in $\Ch^b(\AS)$. The filtrations on $\Fdg(X)$ and $\Gdg(Y)$ induce a filtration on hom complexes $\uHom(\Fdg(X),\Gdg(Y))$ via
\begin{equation}
\FC^k\Big(\uHom(\Fdg(X),\Gdg(Y)\Big):=\Big\{f \in \uHom(\Fdg(X),\Gdg(X))\:\Big|\: \text{ for all } i \in \Z, f \text{ restricts to } F(X^{\geq i})\to G(Y^{\geq i+k})\Big\}.
\end{equation}
Note that $\FC^0(\uHom(\Fdg(X),\Gdg(Y))$ is the subcomplex of filtered morphisms.

For bookkeeping purposes, we will denote the component of $f\in \uHom(\Fdg(X),\Gdg(Y))$ from $F(X^i)$ to $G(Y^j)$ by $f_{ji}$ (note the reversal of order). Then $\FC^k(\uHom(\Fdg(X),\Gdg(Y)))$ consists of those morphisms $f$ such that $f_{ji}=0$ for $j<i+k$. We think of $f_{ji}$ as living in the complex $\uHom(F(X^i),G(Y^j))$.

% Then $f\in \uHom(F(X),G(Y))$ is \emph{filtered} if it has $\FC$-degree $\geq 0$.

Let $f \in \uHom^l(\Fdg(X),\Gdg(Y))$ have degree $l$. We write $f$ in terms of its components $f_{ji}\in \uHom^{l-j+i}(F(X^i),G(Y^j))$.  Informally speaking, we may write $f=\sum f_{ji}$.  Then $d(f) = \sum \tilde{d}(f_{ji})$, where $\tilde{d}(f_{ji})$ is the sum of three morphisms:
\begin{equation} \label{firstterm}
(-1)^j d(f_{ji}) \in \uHom^{l+1-j+i}(F(X^i),G(Y^j)),
\end{equation}
\begin{equation}
G(\d_Y^j)\circ f_{ji}  \in \uHom^{l-j+i}(F(X^{i}),G(Y^{j+1})),
\end{equation}
\begin{equation}
-(-1)^l f_{ji}\circ F(\d_X^{i-1}) \in \uHom^{l-j+i}(F(X^{i-1}),G(Y^{j})).
\end{equation}
Here, $d(f_{ji})$ means the differential of $f_{ji}$ as calculated in $\uHom(F(X^i),G(Y^j))$, i.e.~
\begin{equation}
d(f_{ji}) =\d_{G(Y^j)}\circ f_{ji} - (-1)^{l-j+i} f_{ji} \circ \d_{F(X^i)}. 
\end{equation}
In particular, $(d(f))_{ji}$ has contributions from $\tilde{d}(f_{ji})$ and $\tilde{d}(f_{(j-1)i})$ and $\tilde{d}(f_{j(i+1)})$. As a consequence, if $f \in \FC^k(\uHom(F(X),G(Y)))$ then $d(f) \in \FC^k(\uHom(F(X),G(Y)))$ as well; more precisely, $d(f)$ can be written as a sum of three terms with one in $\FC^k(\uHom(F(X),G(Y)))$ and two in $\FC^{k+1}(\uHom(F(X),G(Y)))$.
Regardless, this shows that $\FC^k(\uHom(F(X),G(Y)))$ is indeed a subcomplex of $\uHom(F(X),G(Y))$.

Abbreviate by writing $\FC^k = \FC^k(\uHom(F(X),G(Y))$. Then $\FC^k/\FC^{k+1}$ is spanned by the images of morphisms $f$ with $f_{ji} = 0$ unless $j = i+k$. That is,
\begin{equation} \label{howassgrworks}
\FC^k/\FC^{k+1} \cong \prod_{i\in \Z} \uHom\Big(F(X^i)[-i], G(Y^{i+k})[-i-k]\Big) \cong \prod_{i\in \Z}\uHom(F(X^i), G(Y^{i+k}))[-k].
\end{equation}
Moreover, the differential on $\FC^k/\FC^{k+1}$ comes only from the terms \eqref{firstterm}, and therefore agrees with the differential on the right-hand side of \eqref{howassgrworks}.

Because $X$ and $Y$ are bounded complexes, there exist $i \in \Z$ such that $X^{\geq i} = X$, and there exists $l \ge 0$ such that $Y^{\geq i+l} = 0$. It follows that $\FC^k = 0$ for all $k$ sufficiently large (namely $k \ge l$).
\end{proof}

\begin{lemma}\label{lemma:extending maps}
Let $F,G:\AS\rightarrow \Ch^b(\BS)$ denote an unobstructed pair of dg functors (see Definition \ref{def:unobstructed}) and let $X,Y\in \Ch^b(\AS)$ be complexes. Suppose we have a family of chain maps $g_{i,i}\colon F(X^i)\to G(Y^i)$ which fit into a diagram
\[
\begin{tikzpicture}
\node (a) at (0,0) {$\cdots$};
\node (b) at (3,0) {$F(X^i)$};
\node (c) at (6,0) {$F(X^{i+1})$};
\node (d) at (9,0) {$\cdots$};
\node (a2) at (0,-2) {$\cdots$};
\node (b2) at (3,-2) {$G(Y^j)$};
\node (c2) at (6,-2) {$G(Y^{j+1})$};
\node (d2) at (9,-2) {$\cdots$};
\path[-stealth,thick]
(a) edge node[above]{$F(\d_X^{i-1})$} (b)
(b) edge node[above]{$F(\d_X^{i})$} (c)
(c) edge node[above]{$F(\d_X^{i+1})$} (d)
(a2) edge node[above]{$G(\d_Y^{i-1})$} (b2)
(b2) edge node[above]{$G(\d_Y^{i})$} (c2)
(c2) edge node[above]{$G(\d_Y^{j+1})$} (d2)
(b) edge node[left]{$g_{i,i}$} (b2)
(c) edge node[left]{$g_{i+1,i+1}$} (c2);
\end{tikzpicture}
\]
in which the squares commute up to homotopy. Then this diagram extends to a filtered chain map $g:\Fdg(X)\rightarrow \Gdg(Y)$.  Furthermore this extension is unique up to homotopy in the sense that if $g,g':\Fdg(X)\rightarrow \Gdg(Y)$ are filtered chain maps with $g_{i,i}\simeq (g')_{i,i}$ for all $i\in \Z$, then $g\simeq g'$.
\end{lemma}

\begin{proof}
As in the previous lemma, write $\FC^k = \FC^k(\uHom(\Fdg(X),\Gdg(Y))$, and let $V=\FC^0\subset \uHom(F(X),G(Y))$ denote the subcomplex of filtered morphisms. We view $\FC^{\ge 0}$ as a filtration on $V$, and it satisfies \eqref{filtconditions} since $\FC^k = 0$ for $k$ sufficiently large. By \eqref{howassgrworkslemma} and the assumption that $F$ and $G$ are unobstructed, we see that $\FC^k/\FC^{k+1}$ has zero homology in degrees $<k$.

% We will use Lemma \ref{lemma:filtered cx}.  Abbreviate throughout by writing $\FC^k = \FC^k(\uHom(F(X),G(Y))$, and let $V=\FC^0\subset \uHom(F(X),G(Y))$ denote the subcomplex of filtered morphisms.  Observe that
% \[
% \FC^k/\FC^{k+1} = \prod_{i\in \Z} \uHom\Big(F(X^i)[-i], G(Y^{i+k})[-i-k]\Big) \cong \prod_{i\in \Z}\uHom(F(X^i), G(Y^{i+k}))[-k].
% \]
% By the obstruction-free assumption, this has zero homology in degrees $<k$.

In particular
\[
H^0(\FC^k/\FC^{k+1})\cong 0\quad \text{for} \quad k\geq 1,\qquad H^1(\FC^k/\FC^{k+1})\cong 0\quad \text{for} \quad k\geq 2.
\] 
Consider the sequence of maps
\begin{equation} \label{remarksonme}
H^0(V)\rightarrow H^0(V/\FC^2) \rightarrow H^0(V/\FC^1).
\end{equation}
It follows from Lemma \ref{lemma:filtered cx} that the first of these arrows is surjective, and the composition is injective.  

Let $c_0\in V$ denote the morphism whose only nonzero components are the $g_{i,i}$. Since each $g_{i,i}$ is a chain map we have
\begin{equation}
d(c_0) = G(\d_X)\circ c_0 - c_0\circ F(\d_X) \in \prod_{i\in \Z} \uHom(F(X^i)[-i], G(Y^{i+1})[-i-1]),
\end{equation}
so $d(c_0)$ lives in $\FC^1$. Recall that the standard differential on $\uHom(F(X^i)[-i], G(Y^{i+1})[-i-1])$ agrees with the differential on $\FC^1/\FC^2$. Our hypotheses on $g_{i,i}$ state that, for each $i$, $G(\d_X)\circ g_{i,i} - g_{i,i} \circ F(\d_X) \in \uHom(F(X^i)[-i], G(Y^{i+1})[-i-1])$ is nulhomotopic. In other words, $d(c_0)$ descends to a boundary in $\FC^1/\FC^2$. Thus, there exists $c_1\in \FC^1$ such that
\[
d(c_0+c_1) \in \FC^2.
\]
Therefore $c_0+c_1$ represents a class in $H^0(V/\FC^2)$.

From the surjectivity in \eqref{remarksonme} above, $c_0 + c_1$ can be lifted to a class $c$ in $H^0(V)$, which lifts to the desired filtered chain map $g \in Z^0(V)$.
% and this lift $c$ is completely determined by the class of $c_0$ in $H^0(V/\FC^1)$. 

Now suppose $g'$ is constructed in the same way, for $g'_{i,i}$ which are homotopic to $g_{i,i}$ for all $i$. This assumption implies that $c_0$ and $c_0'$ are homotopic in
$V/\FC^1$. Now the injectivity in \eqref{remarksonme} implies that $c_0$ and $c_0'$ lift to the same class $c \in H^0(V)$, whence $g$ and $g'$ are homotopic in $V$. This proves the
lemma.
\end{proof}

\begin{proof}[Proof of Theorem \ref{thm:lifting lemma}]
Recall the setup from the statement: we have an unobstructed pair of dg functors $F,G\colon \AS\to \Ch^b(\BS)$, with lifts $\Fdg,\Gdg\colon \Ch^b(\AS)\to \Ch^b(\BS)$.

Let $\etaZ:H^0(F)\rightarrow H^0(G)$ be a natural transformation. Our goal is to construct a lift of $\etaZ$ to a natural transformation $\etadg\colon H^0(\Fdg)\to H^0(\Gdg)$ and show that $\etadg$ is unique.

Let $X\in \Ch^b(\AS)$ be given.  Then $\etaZ$ determines a homotopy commutative diagram of chain maps:
\[
\begin{tikzpicture}
\node (a) at (0,0) {$\cdots$};
\node (b) at (3,0) {$F(X^i)$};
\node (c) at (6,0) {$F(X^{i+1})$};
\node (d) at (9,0) {$\cdots$};
\node (a2) at (0,-2) {$\cdots$};
\node (b2) at (3,-2) {$G(X^j)$};
\node (c2) at (6,-2) {$G(X^{j+1})$};
\node (d2) at (9,-2) {$\cdots$};
\path[-stealth,thick]
(a) edge node[above]{$F(\d_X^{i-1})$} (b)
(b) edge node[above]{$F(\d_X^{i})$} (c)
(c) edge node[above]{$F(\d_X^{i+1})$} (d)
(a2) edge node[above]{$G(\d_X^{j-1})$} (b2)
(b2) edge node[above]{$G(\d_X^{j})$} (c2)
(c2) edge node[above]{$G(\d_X^{j+1})$} (d2)
(b) edge node[left]{$\etaZ_{X^i}$} (b2)
(c) edge node[left]{$\etaZ_{X^{i+1}}$} (c2);
\end{tikzpicture}
\]
This extends, uniquely up to homotopy, to a chain map $\etadg_X:\Fdg(X)\rightarrow \Gdg(X)$ by Lemma \ref{lemma:extending maps}.  We have to check naturality of $\etadg_X$, up to homotopy.  So let $f:X\rightarrow Y$ be a chain map in $\Ch^b(\AS)$.  We must show that
\begin{equation}
\Gdg(f)\circ \etadg_X \simeq \etadg_Y\circ \Fdg(f).
\end{equation}
Note that $\Fdg(f)$ and $\Gdg(f)$ are filtered chain maps, as are $\etadg_X$ and $\etadg_Y$. Thus $\Gdg(f)\circ \etadg_X - \etadg_Y\circ \Fdg(f)$ is a filtered map.  The minimal $\FC$-degree component is a sum over $i\in \Z$ of chain maps
\begin{equation}
(\Gdg(f)\circ \etadg_X - \etadg_Y\circ \Fdg(f))_{ii}= G(f^i)\circ \etaZ_{X^i} - \etaZ_{Y^i}\circ F(f^i)
\end{equation}
which is null-homotopic by naturality of $\etaZ$.  Thus, $\Gdg(f)\circ \etadg_X - \etadg_Y\circ \Fdg(f)$ is null-homotopic by uniqueness of lifts.  This completes the proof.
\end{proof}

\subsection{The multilinear lifting lemma}
\label{ss:multilinear lifting}
Let us state and prove a multilinear version of Theorem \ref{thm:lifting lemma}, that will be useful in the proof of our Drinfeld center lifting lemma (Theorem \ref{thm:lifting drinfeld}).

The notion of multilinear functor is of course equivalent to the notion of an ordinary functor from a tensor product, so we begin by recalling the the notion of tensor product of dg categories. Let $\AS,\BS$ be dg categories. Their tensor product, denoted $\AS\otimes \BS$ has objects given by pairs $(X,Y)$ with $X\in \AS$, $Y\in \BS$, and morphism complexes given by
\[
\Hom_{\AS\otimes \BS}\Big((X',Y'), (X,Y)\Big) := \Hom_\AS(X',X)\otimes \Hom_\BS(Y',Y),
\]
with composition defined by
\begin{equation}
(f\otimes g)\circ (f'\otimes g') = (-1)^{|g||f'|} (f\circ f')\otimes (g\otimes g')
\end{equation}

Let $\AS_1,\ldots,\AS_r$ and $\BS$ be $\k$-linear categories. Assume also that $\BS$ is additive.  Let $F\colon \AS_1\otimes\cdots\otimes \AS_r\to \Ch^b(\BS)$ be a dg functor.  Let $\Fdg\colon \Ch^b(\AS_1)\otimes\cdots\otimes \Ch^b(\AS_r)\to \Ch^b(\BS)$ denote the dg functor defined on objects by
\begin{equation}
\Fdg\left(\tw_{\d_1}\Big(\bigoplus_{i_1} X_1^{i_1}[-i_1]\Big),\ldots, \tw_{\d_1}\Big(\bigoplus_{i_r} X_r^{i_r}[-i_r]\Big)\right):= \tw_{\boldsymbol{\d}}\left(\bigoplus_{i_1,\ldots, i_r} F(X_1^{i_1},\ldots,X_r^{i_r})[-\sum_{m=1}^r i_r]\right)
\end{equation}
where $\boldsymbol{\d} = \sum_{j=1}^r \Fdg(\id,\ldots,\d_j,\ldots,\id)$. The dg functor $\Fdg$ is defined on morphisms by
\begin{equation}
\Fdg(f_1,\ldots,f_r)|_{F(X_1^{i_1},\ldots,X_r^{i_r})} = (-1)^s F\Big(f_1|_{X_1^{i_1}}\, ,\ldots\, ,f_r|_{X_r^{i_r}}\Big).
\end{equation}
where $s=\sum_{1\leq j< j'\leq r} i_j |f_{j'}|$.

\begin{theorem}\label{thm:multilinear lifing lemma}
Let $F,G\colon \AS_1\otimes\cdots\otimes \AS_r\to \Ch^b(\BS)$ be dg functors, and consider their lifts to dg functors $\Fdg,\Gdg\colon \Ch^b(\AS_1)\otimes\cdots\otimes \Ch^b(\AS_r)\to \Ch^b(\BS)$.  Assume that
\[
H^k\left(\uHom\left(F(X_1,\ldots,X_r),G(Y_1,\ldots,Y_r)\right)\right) =0
\]
for all $k<0$ and all sequences of objects $X_i,Y_i\in \AS_i$. Then any closed degree zero natural transformation $\etaZ\colon H^0(F)\to H^0(G)$ lifts to a unique natural transformation $\etadg\colon H^0(\Fdg)\to H^0(\Gdg)$.
\end{theorem}
\begin{proof}
Let $\AS'$ be the additive closure of $\AS_1\otimes \cdots\otimes \AS_r$. We can extend functors $F,G$ additively, obtaining $F,G\colon \AS'\to \Ch^b(\BS)$.  These can be lifted to dg functors $\Fdg',\Gdg'\colon \Ch^b(\AS')\to \Ch^b(\BS)$.  There is an induced natural transformation $\etaZ'\colon H^0(F')\to H^0(G')$ which has a lift to a (unique) natural transformation $\etadg'\colon H^0(\Fdg')\to H^0(\Gdg')$ by the usual lifting lemma (Theorem \ref{thm:lifting lemma}).  Then we restrict the functors $\Fdg'$, $\Gdg'$, and the natural transformation $\etadg'$ along the canonical inclusion $\Ch^b(\AS_1)\otimes\cdots\Ch^b(\AS_r)\hookrightarrow \Ch^b(\AS')$, obtaining the existence of $\etadg$ as in the statement.

For uniqueness one can essentially repeat the argument from the proof of Theorem \ref{thm:lifting lemma}.  We omit the details as they are straightforward.
\end{proof}

\subsection{The Drinfeld centralizer}
\label{ss:relative drinfeld}

Let $\AS$ be a monoidal category and $\MS$ a $\star$-bimodule category. For any $Z\in \MS$, let $L_Z,R_Z\colon \AS\to \MS$ be the functors given by $Z\star -$ and $-\star Z$, respectively.  Let us recall and expand on Definition \ref{def:intro centralizer}.

\begin{definition}\label{def:relative drinfeld}
Retain notation as above.  The \emph{Drinfeld centralizer of $\AS$ in $\MS$} is the $\k$-linear category $\ZS(\AS,\MS)$ whose objects are pairs $(Z,\tau)$ where $Z\in \MS$ and $\tau$ is an isomorphism of functors $L_Z\to R_Z$ such that
\begin{equation}\label{eq:multiplicativity}
\tau_{X_1\star X_2}  = (\id_{X_1}\star \tau_{X_2})\circ (\tau_{X_1}\star \id_{X_2})
\end{equation}
for all $X_1,X_2\in \AS$.   A morphism in $\ZS(\AS,\MS)$ from $(Z',\tau')$ to $(Z,\tau)$ is a morphism $f\colon Z'\to Z$ in $\MS$ such that 
\begin{equation} \label{tobeamorphism} (\id_X \star f)\circ \tau_{X} = \tau'_X \circ (f\star \id_X)\end{equation}
for all $X\in \AS$.
\end{definition}

\begin{remark}
The usual Drinfeld center arises as $\ZS(\AS):=\ZS(\AS,\AS)$.  This category is (braided) monoidal and $\ZS(\AS,\MS)$ is a $(\ZS(\AS),\ZS(\AS))$ $\star$-bimodule category for all $\MS$.  We will not recall this structure as it is standard and does not play a large role in the present paper.
\end{remark}

Now we prove our main results on lifting objects of $\ZS(\AS,\KC^b(\MS))$ to $\ZS(\KC^b(\AS),\KC^b(\MS))$.

\begin{theorem}\label{thm:lifting drinfeld}
Let $\AS$ be a $\k$-linear monoidal category and $\MS$ an $(\AS,\AS)$ $\star$-bimodule category, and let $Z\in \KC^b(\MS)$ be given. Assume that $\uHom(Z\star X_1, X_2\star Z)$ has zero homology in negative degrees.  Then any $(Z,\tau)\in \ZS(\AS,\KC^b(\MS))$ has a unique lift to $(Z,\ttau)\in \ZS(\KC^b(\AS),\KC^b(\MS))$.
\end{theorem}
\begin{proof}
Assume we are given $(Z,\tau)\in \ZS(\AS,\KC^b(\MS))$ as in the statement.  Throughout the proof we will let $L_Z,R_Z\colon \AS\to \Ch^b(\MS)$ be the dg functors definend by $L_Z(X)=Z\star X$ and $R_Z(X)=X\star Z$.  We let $\mathbf{L}_Z,\mathbf{R}_Z\colon \Ch^b(\AS)\to \Ch^b(\MS)$ denote the lifts to complexes.  Note that $R_Z(X) = X\star Z$ and $L_Z(X)\cong Z\star X$ for all complexes $X\in \Ch^b(\AS)$, though the latter isomorphism involves some signs as in Remark \ref{remark:signsfortensoringwithobjects}.

By hypothesis, $\tau$ is a natural isomorphism $H^0(L_Z)\to H^0(R_Z)$.  By Theorem \ref{thm:lifting lemma} there is a unique lift to a natural isomorphism $\ttau\colon H^0(\mathbf{L}_Z)\to H^0(\mathbf{R}_Z)$.  

It remains to prove the multiplicativity condition
\begin{equation}
\ttau_{X_1\star X_2}\simeq (\id_{X_1}\star \ttau_{X_2})\circ (\ttau_{X_1}\star \id_{X_2}).
\end{equation}
For this, define $LL_Z,RR_Z\colon \AS\otimes \AS\to \Ch^b(\MS)$ to be the dg functors defined by
\begin{equation}
LL_Z(X_1,X_2) := Z\star X_1\star X_2 \ , \qquad RR_Z(X_1,X_2):=X_1\star X_2\star Z.
\end{equation}

These functors have lifts $\mathbf{LL}_Z,\mathbf{RR}_Z\colon \Ch^b(\AS)\otimes \Ch^b(\AS)\to \Ch^b(\MS)$. We have two given natural transformations $H^0(\mathbf{LL}_Z)\to H^0(\mathbf{RR}_Z)$, given by 
\begin{equation}
\etadg_{X_1,X_2}:=\ttau_{X_1\star X_2} \ , \qquad \etadg'_{X_1,X_2}:=(\id_{X_1}\star \ttau_{X_2})\circ (\ttau_{X_1}\star \id_{X_2}).
\end{equation}
These natural transformations are lifts of $\tau_{X_1\star X_2}$ and $(\id_{X_1}\star \tau_{X_2})\circ (\tau_{X_1}\star \id_{X_2})$ respectively, which are homotopic by definition of $\ZS(\AS,\KC^b(\MS))$.  Thus, uniqueness of lifts (from Theorem \ref{thm:multilinear lifing lemma}) implies that $\etadg\simeq \etadg'$, which establishes the multiplicativity condition for $\ttau$ and shows that $(Z,\ttau)$ is an object of $\ZS(\KC^b(\AS),\KC^b(\MS))$.
\end{proof}

\subsection{Fully faithful centralizing functors}

In order to use the technology developed above, we must have some way of constructing objects of $\ZS(\AS,\KC^b(\MS))$.  Many examples (see \S \ref{s:diag Hecke}) can be provided with the following relatively elementary result.

\begin{theorem}\label{thm:Phi and Z}
Let $\AS$ be a $\k$-linear monoidal category and $\MS$ an $(\AS,\AS)$ $\star$-bimodule category, and let $Z\in \KC^b(\MS)$ be given. Assume that $L_Z,R_Z\colon \AS\to \KC^b(\MS)$ are fully faithful and have the same essential image.  Then there is a monoidal autoequivalence $\Phi\colon \AS\to \AS$ such that $Z\star X\simeq \Phi(X)\star Z$, naturally in $X\in \AS$. In other words, $Z$ has the structure of an object in the Drinfeld centralizer $\ZS(\AS,{}^\Phi\KC^b(\MS))$, where ${}^\Phi\KC^b(\MS)$ denotes $\KC^b(\MS)$ with left action twisted by $\Phi$.
\end{theorem}
\begin{proof}
Without loss of generality, let us assume that $\AS$ is strict monoidal.  Let $\NS\subset \KC^b(\MS)$ be the essential image of $L_Z$ and $R_Z$.  Then $L_Z,R_Z$ restrict to equivalences of categories $\AS\to \NS$, and we may choose an inverse functor $R_Z\inv$ and define $\Phi:=R_Z\inv \circ L_Z$.  The fact that $Z\star X\simeq \Phi(X)\star Z$, naturally in $X$ is true by construction.  There is an isomorphism $\psi_{X_1,X_2}\colon \Phi(X_1\star X_2)\buildrel\cong\over\rightarrow \Phi(X_1)\star \Phi(X_2)$, defined by commutativity of the diagram
\begin{equation}\label{eq:tau and psi}
\begin{tikzpicture}
\node (a) at (0,0) {$\Phi(X_1\star X_2)\star Z$};
\node (b) at (5,0) {$Z\star X_1\star X_2$};
\node (c) at (5,2) {$\Phi(X_1)\star Z\star X_2$};
\node (d) at (0,2) {$\Phi(X_1)\star \Phi(X_2)\star Z$};
\path[-stealth,very thick]
(a) edge node[above] {$(\tau_{X_1\star X_2})\inv$} (b)
(b) edge node[right] {$\tau_{X_1}\star \id_{X_2}$} (c)
(c) edge node[above] {$\id_{\Phi(X_1)}\star \tau_{X_2}$} (d);
\path[-stealth,dashed,very thick]
(a) edge node[left] {$\psi_{X_1,X_2}\star \id_Z$} (d);
\end{tikzpicture}
\end{equation}
%\[
%\begin{tikzpicture}
%\node (a) at (0,0) {$\Phi(X_1\star X_2)\star Z$};
%\node (b) at (4,0) {$Z\star (X_1\star X_2)$};
%\node (c) at (8,0) {$(Z\star X_1)\star X_2$};
%\node (d) at (12,0) {$(\Phi(X_1)\star Z)\star X_2$};
%\node (e) at (12,2) {$\Phi(X_1)\star (Z\star X_2)$};
%\node (f) at (6,2) {$\Phi(X_1)\star (\Phi(X_2)\star Z)$};
%\node (g) at (0,2) {$(\Phi(X_1)\star \Phi(X_2))\star Z$};
%\path[-stealth,very thick]
%(a) edge node[above] {$\simeq$} (b)
%(b) edge node[above] {$\cong$} (c)
%(c) edge node[above] {$\simeq$} (d)
%(d) edge node[right] {$\cong$} (e)
%(e) edge node[above] {$\simeq$} (f)
%(f) edge node[above] {$\cong$} (g);
%\path[-stealth,dashed,very thick]
%(a) edge node[left] {$\psi_{X_1,X_2}$} (g);
%\end{tikzpicture}
%\]
Where the solid arrows alternate between homotopy equivalences of the form $R_Z(\Phi(X))\simeq L_Z(X)$, and associator isomorphisms.  Observe that, a priori the above diagram defines a homotopy equivalence $\Phi(X_1\star X_2)\simeq \Phi(X_1)\star \Phi(X_2)$.  But since $X_i$ here are objects of $\AS$, all homotopies are zero for degree reasons.  Moreover, naturality of the isomorphism $R_Z\circ \Phi\simeq L_Z$ and of the associator implies that the isomorphism $\Phi(X_1\star X_2)\cong \Phi(X_1)\star \Phi(X_2)$ is natural in $X_1,X_2\in \AS$.

To complete the proof that $\Phi$ is a monoidal functor, one would need to prove commutativity of the diagram
%\[
%\begin{tikzpicture}
%\node (a1) at (5,0) {$(\Phi(X_1)\star \Phi(X_2))\star \Phi(X_3)$};
%\node (a2) at (5,2) {$\Phi(X_1)\star (\Phi(X_2)\star \Phi(X_3))$};
%\node (b1) at (0,0) {$\Phi(X_1\star X_2)\star \Phi(X_3)$};
%\node (b2) at (0,2) {$\Phi(X_1)\star \Phi(X_2\star X_3)$};
%\node (c1) at (-5,0) {$\Phi((X_1\star X_2)\star X_3)$};
%\node (c2) at (-5,2) {$\Phi(X_1\star (X_2\star X_3))$};
%\path[-stealth,very thick]
%(c1) edge node[above] {$\cong$} (b1)
%(c2) edge node[above] {$\cong$} (b2)
%(b1) edge node[above] {$\cong$} (a1)
%(b2) edge node[above] {$\cong$} (a2)
%(a1) edge node[right] {$\cong$} (a2)
%(c1) edge node[right] {$\cong$} (c2);
%\end{tikzpicture}
%\]
%Without loss of generality, we may assume that $\AS$ is strict as a monoidal category, hence we need to prove commutativity of
\begin{equation}
\begin{tikzpicture}
\node (a) at (5,1) {$\Phi(X_1)\star \Phi(X_2)\star \Phi(X_3)$};
\node (b1) at (0,0) {$\Phi(X_1\star X_2)\star \Phi(X_3)$};
\node (b2) at (0,2) {$\Phi(X_1)\star \Phi(X_2\star X_3)$};
\node (c) at (-5,1) {$\Phi(X_1\star X_2\star X_3)$};
\path[-stealth,very thick]
(c) edge node[above] {} (b1)
(c) edge node[above] {} (b2)
(b1) edge node[above] {} (a)
(b2) edge node[above] {} (a);
\end{tikzpicture}
\end{equation}
To prove commutativity of this diagram, we prove that it remains commutative after application of $R_Z$.  Consider the diagram:
\begin{equation}
\begin{tikzpicture}
\node (3z) at (-6,2) {$\Phi(X_1 \cdot X_2 \cdot X_3) \cdot Z$};
\node (12z) at (-1.5,3) {$\Phi(X_1) \cdot \Phi(X_2 \cdot X_3) \cdot Z$};
\node (21z) at (1.5,1) {$\Phi(X_1 \cdot X_2)\cdot \Phi(X_3)\cdot Z$};
\node (111z) at (6,2) {$\Phi(X_1)\cdot \Phi(X_2)\cdot \Phi(X_3)\cdot Z$};
\node (z3) at (-6,-2) {$Z\cdot X_1\cdot X_2\cdot X_3$};
\node (1z2) at (-1.5,-1) {$\Phi(X_1)\cdot Z\cdot X_2\cdot X_3$};
\node (2z1) at (1.5,-3) {$\Phi(X_1\cdot X_2)\cdot Z\cdot X_3$};
\node (11z1) at (6,-2) {$\Phi(X_1)\cdot \Phi(X_2)\cdot Z\cdot X_3$};
%\node at (0,2) {$I$};
%\node at (0,-2) {$II$};
%\node at (3,0) {$III$};
%\node at (-3,0) {$IV$};
\path[-stealth,very thick]
(3z) edge node {} (12z)
(3z) edge node {} (21z)
(z3) edge node[above] {} (1z2)
(z3) edge node[above] {} (2z1)
(11z1) edge node {} (111z)
(2z1) edge node[above] {} (21z)
(z3) edge node[above] {} (3z)
(1z2) edge node[above] {} (11z1)
(1z2) edge node[above] {} (12z)
(21z) edge node {} (111z)
(12z) edge node {} (111z)
(2z1) edge node {} (11z1);
\end{tikzpicture}
\end{equation}
(abbreviating by writing $\cdot = \star$).  The front right square commutes obviously.  The bottom square commutes by definition of $\psi_{X_1,X_2}$.  The back right square commutes by definition of $\psi_{X_2,X_3}$.  The back left square commutes by definition of $\psi_{X_1, X_2\cdot X_3}$, and the front left square commutes by definition of $\psi_{X_1\cdot X_2,X_3}$.  Thus, the top square commutes, which proves that $\Phi$ is monoidal.

The multiplicativity condition for $Z$ is equivalent to the commutativity of the diagram \eqref{eq:tau and psi}.  This proves the theorem. 
\end{proof}

\begin{remark} \label{rmk:Phiuptoisomorphism}
There is a subtle issue at play above. The construction of $\Phi$ depended on a choice of inverse functor $R_Z\inv$. Inverse functors are well-defined up to isomorphism of functors. Let $\Phi'$ be another functor isomorphic to $\Phi$. It is false that ${}^\Phi\KC^b(\MS)$ and ${}^{\Phi'}\KC^b(\MS)$ are equivalent as $\star$-bimodule categories (though one would expect them to be Morita equivalent in the appropriate sense)! Nonetheless, $\ZS(\AS,{}^\Phi\KC^b(\MS))$ and $\ZS(\AS,{}^{\Phi'}\KC^b(\MS))$ are equivalent. Given an object $(Z,\tau) \in \ZS(\AS,{}^\Phi\KC^b(\MS))$, we can get an object $(Z,\tau') \in \ZS(\AS,{}^{\Phi'}\KC^b(\MS))$ by postcomposing $\tau$ with the isomorphism from $\Phi$ to $\Phi'$ (tensored with $\id_Z$). 
\end{remark}

The conclusion of the above remark is that one should really only consider $\Phi$ up to isomorphism of functors.

%%%%%%%%%%%%%%%%%%%%%%%%%
\section{Conjugation by Rouquier complexes}
\label{s:diag Hecke}
%%%%%%%%%%%%%%%%%%%%%%%%%

In this chapter we focus on applications to Hecke categories. Throughout, let $W$ be a Coxeter group with set of simple reflections $S\subset W$, and fix a commutative ring $\k$. 

%------------------------
\subsection{The Hecke category}
%------------------------

Most of the technical details which go into the construction of the Hecke category will not be relevant for this paper. We refer the reader to \cite{EWGr4sb} or \cite{EMTW} for more details.

A \emph{realization} is a free $\k$-module $V$ together with a collection of \emph{simple roots} $\{\alpha_s\}_{s \in S} \subset V$ and \emph{simple coroots} $\{\alpha_s^\vee\}_{s \in S} \subset V^*$, satisfying some axioms. Here $V^* := \Hom_{\k}(V,\k)$. The main axiom is that $W$ acts on $V$, where the simple reflections act by the following formula:
\begin{equation} s(v) = v - \alpha_s^\vee(v) \cdot \alpha_s. \end{equation}
We write $(W,S,V)$ to indicate a Coxeter system with a choice of realization, omitting the choice of roots and coroots from the notation. 

\begin{remark} \label{remark:assumptions} For the experts, we quickly discuss the assumptions we use, referring to \cite[\S 5, \S 6]{EWLocalized} for definitions. For the Hecke category to be well-defined we
need certain Jones-Wenzl projectors to exist and be rotatable. Precise conditions for this were given in \cite{HaziRotatable}. For the category to be cyclic we also need to assume
the realization is even-balanced. We also assume Demazure surjectivity. These assumptions are sufficient for the Hecke category to categorify the Hecke algebra with the expected
size of Hom spaces, see \cite[\S 5]{EWLocalized}. We also assume in the body of the text that the realization is odd-balanced, though we treat the odd-unbalanced case in remarks.
\end{remark}

When discussing $\Z$-graded algebras and modules, we use the symbol $(k)$ to denote the grading shift which decreases all degrees by $k$, so that if $M$ is a graded $\k$-module then a homogeneous element $m\in M$ of degree $i$ will have degree $i-k$ when regarded as an element of $M(k)$.  In other places in the literature (see e.g.~our main reference \cite{MattsWorkInProgress} for dg categorical constructions) $M(k)$ would be denoted by $q^{-k}M$.

Let $R$ be the polynomial ring whose linear terms are $V$, graded by $\deg(v)=2$ for all $v\in V$. In other words, $R=\operatorname{Sym}^\bullet(V(-2))$. The operators $\alpha_s^\vee \colon V \to \k$ extend by a twisted Leibniz rule to operators $\partial_s \colon R \to R$, satisfying
\begin{equation} \label{demazure} \alpha_s \cdot \partial_s(f) = f - sf \end{equation}
for all $f \in R$. Also, $\partial_s(f) \in R^s$, where $R^s$ denotes the subring of $s$-invariants.

There is a graded $(R,R)$-bimodule $B_s$, defined as
\[ B_s := R \otimes_{R^s} R(1). \]
%where we remind the reader that the grading shift places the element $1 \otimes 1$ in degree $-1$.
Bimodules obtained as tensor products $B_{s_1} \otimes B_{s_2} \otimes \cdots \otimes B_{s_r}$ are called \emph{Bott-Samelson bimodules}, and they form a monoidal subcategory of the category of all graded $(R,R)$-bimodules. Bimodules obtained from Bott-Samelson bimodules by taking direct sums, grading shifts, and direct summands are called \emph{Soergel bimodules}. Under fairly reasonable restrictions on $\k$ and $V$, the category of Soergel bimodules categorifies the Iwahori-Hecke algebra of $(W,S)$, where $B_s$ categorifies the so-called Kazhdan-Lusztig generator $b_s$. However, Soergel bimodules can fail to categorify the Hecke algebra (e.g. when $W$ does not act faithfully on $V$). Instead, a diagrammatic category $\Diag = \Diag(W,S,V)$ was developed in \cite{EKho, ECathedral, EWGr4sb}, presented by generators and relations, which models the category of Bott-Samelson bimodules when they behave well. It was proven in \cite{EWGr4sb} and \cite[\S 5]{EWLocalized} that $\Diag$ always categorifies the Hecke algebra (using the assumptions of Remark \ref{remark:assumptions}), whence it is now called the \emph{Hecke category}.

By definition, the category $\Diag$ is the strict monoidal category freely generated by formal objects which we also call $B_s$, for $s \in S$, with $\Z$-graded morphism spaces presented by some explicit diagrammatic generators and relations.  We will denote the monoidal operation by $\star$ or simply by horizontal juxtaposition, so that a general object may be denoted either by $B_{s_1}\star\cdots \star B_{s_r}$ or $B_{s_1}\cdots B_{s_r}$.   The monoidal identity, denoted $\one$, corresponds to the empty sequence.

The presentation for general $(W,S,V)$ can be found in \cite[\S 5]{EWGr4sb} or \cite[\S 10]{EMTW}, but we will recall some of the most revelevant features.  First, $\End_\Diag(\one)=R$.  Thus, given $X\in \Diag$ and $f\in R$ the monoidal structure provides endomorphisms $\id_X f$ and $f\id_X$, which we refer to as the right and left multiplication by $f$ on $X$, respectively. In this way, all hom spaces in $\Diag$ may be viewed as graded $(R,R)$-bimodules. 

The generating morphisms consist of:
\begin{align}
\text{dots: } \ \ 
\begin{tikzpicture}[anchorbase]
\node (a) at (0,0) {$\one$};
\node (b) at (2,0) {$B_s$};
\path[-stealth,very thick]
([yshift=2pt]a.east) edge node[above]  {$\eta_s$} ([yshift=2pt]b.west)
([yshift=-2pt]b.west) edge node[below]  {$\e_s$} ([yshift=-2pt]a.east);
\end{tikzpicture}
\ , \quad \text{ trivalent vertices: } \ \ 
\begin{tikzpicture}[anchorbase]
\node (a) at (0,0) {$B_sB_s$};
\node (b) at (2,0) {$B_s$};
\path[-stealth,very thick]
([yshift=2pt]a.east) edge node[above]  {$\mu_s$} ([yshift=2pt]b.west)
([yshift=-2pt]b.west) edge node[below]  {$\Delta_s$} ([yshift=-2pt]a.east);
\end{tikzpicture}
\\
\text{$2m_{st}$-valent vertices: } \ \ 
\begin{tikzpicture}[baseline=-.1cm]
\node (a) at (0,0) {$B_sB_t\cdots B_{s\text{ or }t}$};
\node (b) at (5,0) {$B_tB_s\cdots B_{t \text{ or }s}$};
\path[-stealth,very thick]
(a) edge node[above]  {$\pi_{s,t}$} (b);
\end{tikzpicture}
\end{align}
for all $s,t\in S$ (if $m_{s,t}=\infty$ then the generator $\pi_{s,t}$ is omitted), with degrees $\deg(\e_s)=\deg(\eta_s)=1$, $\deg(\mu_s)=\deg(\Delta_s)=-1$, $\deg(\pi_{s,t})=0$. We call this the \emph{Soergel degree} (we are in a graded category, not yet a dg category).  Some of the relations include
\begin{align}
\e_s \circ \eta_s &= \alpha_s . \label{barbell} \\ 
\id_{B_s} f - s(f) \id_{B_s} &= \partial_s(f) (\eta_s \circ \e_s). \label{polyforce} 
\end{align}
The first of these is a relation in $\End_\Diag(\one) = R$, and the second is a relation in $\End_\Diag(B_s)$.

Finally, we remark that the Hecke category $\Diag$ has a contravariant duality functor which flips diagrams upside-down. It swaps $\eta_s$ and $\e_s$, swaps $\mu_s$ and $\Delta_s$, and swaps $\pi_{s,t}$ and $\pi_{t,s}$.

%------------------------
\subsection{Rouquier complexes}
%------------------------

Let $\Diag^\oplus$ denote the category obtained from $\Diag$ by adjoining all shifts $(k)$ and finite direct sums of objects (see e.g. \cite[\S 11.2]{EMTW}). A typical object of $\Diag^\oplus$ is a formal expression $\bigoplus_{i} X_i(k_i)$ where the indexing set is finite.

%\cite{HogTwists}

\begin{notation}
We will abuse notation and write $\Ch^b(\Diag)$ instead of $\Ch^b(\Diag^\oplus)$ (and similarly $\KC^b(\Diag)$ instead of $\KC^b(\Diag^\oplus)$).
\end{notation}

When working with complexes $X, Y \in \Ch^b(\Diag)$, we write $\uHom^{i,j}(X,Y)$ for the maps of homological degree $i$ and Soergel degree $j$.  That is to say, $\Ch^b(\Diag)$ is a dg category, with gradings living in $\Z\times \Z$, as in Remark \ref{rmk:other gradings}. Just as for $\Diag$, Hom complexes are still $(R,R)$-bimodules via the monoidal structure (and no signs appear in the tensor product $\id_X \star f$ for $f \in R$, since $f \colon \one \to \one$ has homological degree zero).

Let $F_s \in \Ch^b(\Diag)$ be the Rouquier complex 
\begin{equation} F_s := (\underline{B_s}\to \one(1))\end{equation}
with differential $\e_s$. The underline indicates that $B_s$ lives in homological degree zero. The dual complex
\begin{equation} F_s\inv = (\one(-1)\to \underline{B_s}) \end{equation}
has differential $\eta_s$. The products $F_s \star F_s\inv$ and $F_s\inv \star F_s$ are homotopy equivalent to the monoidal identity $\one$ (viewed as a complex in bidegree zero).

Let $\nu_s\colon \one \to F_s$ be the inclusion of the rightmost chain object, and let $\nu_s^\vee\colon F_s\inv\to \one$ be its dual, the projection to the leftmost chain
object. So $\nu_s \in \uHom^{1,-1}(\one,F_s)$ and $\nu_s^\vee \in \uHom^{1,-1}(F_s\inv,\one)$. 

Let $\Br=\Br(W)$ be the braid group of $W$.  This is the group generated by symbols $\sigma_s$ with $s\in S$, modulo relations of the form $(\sigma_s\sigma_t)^{m_{st}}=1$ for all $s,t$. A \emph{braid word} is a tuple $(\sigma_{s_1}^{\pm},\ldots,\sigma_{s_r}^{\pm})$ of braid group generators and their inverses.  Multiplying together all the elements in a braid word $\underline{\b}$ yields a braid $\b$, and we say that $\underline{\b}$ \emph{represents} $\b$.  For each braid word $\underline{\b}\in \Br$, let $F(\underline{\b})$ denote the corresponding product of Rouquier complexes $F_s$ and $F_s\inv$.

It is a theorem of Rouquier \cite{RouqBraid-pp} that two braid words representing the same braid give rise to homotopy equivalent Rouquier complexes. In fact this homotopy equivalence is canonical, as we discuss in the next section.

\begin{lemma} \label{lem:conjpoly}
Let $\underline{\b}$ be a braid word expressing $\b \in \Br$, and let $w\in W$ be the image of $\b$ under the natural homomorphism $\Br\rightarrow W$.  Then $\id_{F(\underline{\b})} f \simeq w(f) \id_{F(\underline{\b})}$ for all $f\in R$.
\end{lemma}
\begin{proof}
The relevant homotopy is already in the literature, but it is easy enough that we have reprinted it.  

It suffices to assume that $\underline{\b}=\sigma_s^{\pm}$. Consider the following endomorphism of $F_s$ of homological degree $-1$.
\begin{equation} \label{eq:fFsFssf}
\begin{tikzpicture}
\node (a) at (0,0) {$B_s$};
\node (b) at (3,0) {$\one(1)$};
\node (c) at (0,2) {$B_s$};
\node (d) at (3,2) {$\one(1)$};
\path[-stealth,very thick]
(a) edge node[above] {$\e_s$} (b)
(c) edge node[above] {$\e_s$} (d)
(b) edge node[right=10pt] {$\eta_s \pa_s(f)$} (c);
\end{tikzpicture}
\end{equation}
This is a homotopy for $s(f) \id_{F_s} - \id_{F_s} f$, using the relations \eqref{demazure} and \eqref{barbell} and \eqref{polyforce}.
Diagrammatically, where $\e_s = \finaldotred$ and $\eta_s = \startdotred$, these relations are depicted as
\begin{equation}
\linered f - s(f) \linered  = \brokenred \pa_s(f) \ , \qquad f - s(f) = \barbred \pa_s(f).
\end{equation}

The homotopy for $F_s^{-1}$ is completely analogous, and dual to this one. 
\end{proof}

%------------------------
\subsection{Rouquier canonicity, canonical maps, residues}
\label{ss:canonical maps}
%------------------------

\begin{definition}\label{def:res}
Let $\fm\subset \Diag$ denote the ideal spanned by morphisms which factor through an object of the form $B_{s_1}\cdots B_{s_r}$ for $r\geq 1$.  Let $\Res\colon \Diag\to \Diag/\fm$ denote the quotient functor.

For an object $X \in \Diag$, let $X/\fm$ denote $\Res(X)$.
%Given complexes $X,Y\in \Ch^b(\Diag)$, we say that chain map $f\colon X\to Y$ is a $\Res$-equivalence if $\Res(f)$ is homotopy equivalence from $\Res(X)$ to $\Res(Y)$.
%
%We say that $f\in \uHom(X,Y)$ is a \emph{shifted $\Res$-equivalence} if $f$ is closed and $\Res(f)$ is invertible (but possibly not degree zero) up to homotopy.
\end{definition}

Note that $\Diag/\fm$ is a monoidal category, and $\Res$ is a monoidal functor, since $\fm$ is closed under tensoring with arbitrary morphisms in $\Diag$ on the left and right.
This ideal $\fm$ is an example of a \emph{cellular ideal}, see \cite[\S 22.1]{EMTW}. Clearly $\fm$ contains the identity maps of all objects in $\Diag$ other than $\one$, though it is not obvious a priori that the identity of $\one$ is not included in $\fm$. Indeed, it is not, which can be proven with cell theory. The
following lemma also proves this fact, and is well-known, though we do not know where to find the proof in the literature.

\begin{lemma} \label{lem:endoonemodm} The endomorphism ring of $\one/\fm$, viewed as a quotient of the endomorphism ring of $\one$, is $R/(\a_s)_{s\in S}$. In particular, the degree zero endomorphism ring is $\k$. \end{lemma}
	
\begin{proof} (Sketch) Consider any endomorphism $\phi$ of $\one$ in $\Diag$ which factors through a non-identity object. Take an extremal strand in this diagram (colored $s$), use the unit relation \cite[(5.4)]{EWGr4sb} to pull out a dot, and isotope this dot to the bottom of the picture. Now $\phi = \psi \circ \eta_s$. By the Soergel hom formula we already know that $\Hom(B_s,\one)$ is a free $R$-module (a bimodule where the left and right actions agree) generated by $\e_s$. Thus $\psi = f \e_s$ for some $f \in R$. Then $\phi = f (\e_s \circ \eta_s) = f \alpha_s$ is in the ideal generated by $\alpha_s$. \end{proof}

Thus $\Diag/\fm$ has one nonzero object $\one/\fm$, whose endomorphism ring is $R/(\a_s)_{s\in S}$.

% \begin{notation}
% We will also denote the action of $\Res$ on objects and morphisms by $X/\fm:=\Res(X)$ and $\Res(f)/\fm$, respectively.
% \end{notation}

Let $e\colon \Br\to \Z$ is the group homomorphism defined by $e(\sigma_s)=1$ for all $s\in S$.
Since each Rouquier complex $F(\underline{\b})$ has a single (shifted) copy of $\one$ it follows that (extending $\Res$ to complexes in the standard way)
\begin{equation}\label{eq:res of b}
F(\b)/\fm\cong  (\one/\fm)[-e](e)
\end{equation}
canonically, where $e=e(\b)$.  The isomorphism \eqref{eq:res of b} is canonical because $\one[-e](e)$ \emph{is} a summand of $F(\b)$ (after forgetting the differential), hence the inclusion of this summand (which is not a closed morphism in $\Ch^b(\Diag)$ but becomes closed upon applying $\Res$) provides the isomorphism \eqref{eq:res of b}.  In view of these canonical isomorphisms, we can now formulate the following.

\begin{definition}\label{def:res coeff}
Let $X,Y\in \Ch^b(\Diag)$ be complexes such that $\Res(X)\cong (\one/\fm)[a](a')$ and $\Res(Y)\cong (\one/\fm)[b](b')$ (with a preferred isomorphism in each case).  If $f\in \uHom(X,Y)$ is a morphism of bidegree $(a-b,a'-b')$ then we define $\res(f)\in \k$ to be component of $f$ between the copies of $\one$ in $X$ and $Y$.  More precisely, $\res(f)$ is defined to be the composition
\begin{equation}
\one/\fm\to (\one/\fm)[a](a')\buildrel\cong\over\longrightarrow  X/\fm \buildrel f/\fm\over\longrightarrow Y/\fm \buildrel\cong\over\longrightarrow  (\one/\fm)[b](b')\to \one/\fm
\end{equation}
where the first and last arrows are the ``identity maps'' relating $\one/\fm$ and its shifts. The composition is a degree zero endomorphism of $\one/\fm$, whence an element of $\k$ by Lemma \ref{lem:endoonemodm} (a scalar multiple of the identity map).
\end{definition}

\begin{proposition}\label{prop:res}
We have:
\begin{enumerate}
\item Given braid words $\underline{\b_0},\underline{\b_1},\underline{\b_2}$ and morphisms $f_i\in \uHom(F(\underline{\b_{i-1}}),F(\underline{\b_i}))$ of the appropriate bidegrees, then
\begin{equation}
\res(f_2\circ f_1) = \res(f_2)\res(f_1).
\end{equation}
\item Given braids $\underline{\b},\underline{\b}',\underline{\gamma},\underline{\gamma}'$ and morphisms $f\in \uHom(F(\underline{\b}'),F(\underline{\b}))$, $g\in \uHom(F(\underline{\gamma}'),F(\underline{\gamma}))$ of the appropriate bidegrees, we have
\begin{equation}
\res(f\star g) = (-1)^{e(\b')(e(\gamma)-e(\gamma'))}\res(f)\res(g).
\end{equation}
\end{enumerate}
\end{proposition}
\begin{proof}
The first statement is clear as $\Res$ is a functor. The second statement is a consequence of the Koszul sign rule, as we now explain.  Abbreviate by letting $I(\underline{\b}')=\one[-e(\underline{\b}')](e(\underline{\b}'))$ regarded as a summand of $F(\underline{\b}')$, and similarly for $I(\underline{\gamma}')$.  Then the Koszul sign rule \eqref{koszulsignrule} states that the restriction of $f\star g$ to the identity summands $I(\underline{\b}')\star I(\underline{\gamma}')$ is given by
\begin{equation}
f\star g\Big|_{I(\underline{\b}')\star I(\underline{\gamma}')}
=
(-1)^{|g|e(\b')} f|_{I(\underline{\b}')}\star g|_{I(\underline{\gamma}')}.
\end{equation}
\end{proof}

Rouquier canonicity is the statement that Rouquier complexes for different braid words are canonically homotopy equivalent (see \cite{RouqBraid-pp}). One can find the following proof of Rouquier canonicity sketched in \cite[Section 4.2]{EGaitsgory}.

\begin{lemma}[Rouquier canonicity]\label{lemma:canonicity1}
If $\underline{\b}'$ and $\underline{\b}$ are two braid words representing the same braid $\b$. We call a degree zero chain map $\nu_{\underline{\b},\underline{\b}'}\colon F(\underline{\b})\to F(\underline{\b'})$ a \emph{canonical map} if $\res(\nu_{\underline{\b},\underline{\b}'})=1$. Then there exists a canonical map, and it is unique up to homotopy. Any canonical map is a homotopy equivalence. The composition or tensor product of canonical maps is a canonical map.
\end{lemma}

\begin{proof} Because $F(\underline{\b})$ and $F(\underline{\b'})$ are homotopy equivalent, $F(\underline{\b'}) F(\underline{\b})\inv$ is homotopy equivalent to the monoidal identity.
Tensoring with $F(\underline{\b})^{-1}$ we get an homotopy equivalence of hom spaces
\begin{equation}
\uHom_{\KC^b(\Diag)}(F(\underline{\b}),F(\underline{\b}')) \simeq \uHom_{\KC^b(\Diag)}(\one,\one) \simeq R.
\end{equation}
Hence $\res$ defines an isomorphism of $\k$-algebras $H^{0,0}(\uHom(F(\underline{\b}),F(\underline{\gamma}))\to H^{0,0}(R) = \k$. The previous proposition implies the claims about composition and tensor product.
%
%
%Let $\nu\colon F(\underline{\b}')\to F(\underline{\b})$ be a homotopy equivalence.  Then $\Res(\nu)$ is invertible, hence $\res(\nu)\in \k^\times$. Replacing $\nu$ by $\nu \cdot \res(\nu)\inv$ if necessary, we may assume that $\res(\nu)=1$.  This shows existence.  Uniqueness follows since the projection $R\to R/(\a_s)_{s\in S}$ is the identity on degree zero elements.
\end{proof}

From this point on, we sometimes denote $F(\underline{\b})$ simply by $F(\b)$.  This is still a slight abuse, since the homotopy equivalence relating two different models for $F(\b)$ is canonical only up to homotopy.  Consequently, whenever we write $\uHom(F(\b),F(\b'))$, it is understood that this hom complex requires a choice of braid words representing $\b$ and $\b'$, and two choices yield homotopy equivalent hom complexes via the canonical map (and the homotopy equivalence is itself canonical up to homotopy).  In particular the homology class of a closed morphism $f\in \uHom(F(\b),F(\b'))$ is well-defined.  
%frequently we are only interested in this hom complex up to homotopy equivalence. We use $\simeq$ for homotopy equivalence of complexes, and write $f \simeq g$ if two morphisms are homotopic.

One can also develop more generally the theory of maps between inequivalent Rouquier complexes which induce the identity map upon applying $\Res$.  Let us temporarily write $\b'\leq \b$ if there exists a chain map $F(\b')\to F(\b)$ which is invertible upon applying $\Res$.  Lemma \ref{lemma:canonicity1} implies that $\leq$ is a partial order.  Clearly the relation $\b'\leq \b$ is preserved under multiplying on the left or right by any braid.  The following establishes that if $\b$ is a positive braid (i.e.~it is expressed by a word without any inverse generators $\sigma_s\inv$) then $1\leq \b$.

\begin{definition}\label{def:nu maps}
If $\b$ is a positive braid then we let $\nu_\b\in \uHom(\one,F(\b))$ denote inclusion of the right-most chain object (which is the unique copy of $\one$ inside $F(\b)$).  We denote $\nu_{\sigma_s}$ also by $\nu_s$.
\end{definition}

Note that $\nu_\b$ is closed (i.e. it is a chain map of nonzero degree). It has degree $(e(\b),-e(\b))$, and induces the isomorphism \eqref{eq:res of b} upon applying $\Res$.  Furthermore, $\res(\nu_\b)=1$ by construction.

\begin{definition}\label{def:can}
For each pair of braids $(\b'\leq \b)$, let $\mathrm{can}(\b',\b)$ be the set of homotopy classes closed morphisms $f\colon F(\b')\to F(\b)$ with $\res(f)=1$.
\end{definition}

We think of $\mathrm{can}(\b',\b)$ as the set of ``canonical maps'' from $F(\b')\to F(\b)$.  Of course, this is sensible only when this set $\mathrm{can}(\b',\b)$ is a singleton.  An interesting question that we do not address in the paper is:

\begin{question}
If $\b'\leq \b$, is $\mathrm{can}(\b',\b)$ always a singleton?
\end{question}

For instance, $\mathrm{can}(\b,\b)=\{[\id_{F(\b)}]\}$ is a singleton for any braid $\b$ by Lemma \ref{lemma:canonicity1}, and it is not hard to show that $\mathrm{can}(1,\b)=\{[\nu_\b]\}$ is a singleton for all positive braids $\b$.  For our purposes we will only need the following, which implies that $\mathrm{can}(\b\gamma,\b\sigma_s\gamma)$ a singleton for all $\b,\gamma$. Indeed, the second statement implies that the space of chain maps in this bidegree is one-dimensional up to homotopy.

%Note also that these maps are multiplicative: if $\b=\sigma_{s_1}\cdots\sigma_{s_r}$ then
%\[
%\nu_{\b} = \nu_{s_1}\star \cdots\star \nu_{s_r}.
%\]

\begin{lemma}\label{lemma:canonicity2}
For any braids $\b,\gamma\in \Br$ and any $s\in S$ we have
\begin{equation}
\res(\id_{F(\b)}\star \nu_s\star \id_{F(\gamma)}) = (-1)^{e(\b)}.
\end{equation}
Moreover, $f\in \uHom(F(\b\gamma),F(\b s \gamma))$ is any closed map with bidegree $(1,-1)$, then
\begin{equation}
f\simeq (-1)^{e(\b)}\res(f)\cdot \id_{F(\b)}\star \nu_s\star \id_{F(\gamma)}.
\end{equation}
\end{lemma}

%Recall that $\nu_s\in \uHom(\one,F(s))$ denotes the inclusion of the degree one chain object, so $\res(\nu_s)=1$ by construction.

\begin{proof} The first statement is an immediate consequence of Proposition \ref{prop:res}.  For the second statement, observe that since $F(\b)$ and $F(\gamma)$ are invertible up to homotopy we have
\begin{equation}
\uHom(F(\b\gamma),F(\b s \gamma)) \simeq \uHom(\one,F(s)).
\end{equation}
By direct computation, $\uHom(\one,F(s))$ is the two-term complex of $(R,R)$-bimodules given by
\begin{equation} \underline{R}(-1) \buildrel \alpha_s \over \to R(1), \end{equation}
whose homology is $R/(\alpha_s)(1)$ in homological degree $1$. Hence $\res$ defines an isomorphism $H^{1,-1}(\uHom(F(\b\gamma),F(\b s \gamma)))\to \k$. 
\end{proof}

%
%The fact that $\res(f\circ g)\simeq \res(f)\circ \res(g)$ and $\res(\id)\simeq \id$ implies that the system of $\nu$-maps is transitive
%\[
%\nu_{\underline{\b}'',\underline{\b}'}\circ \nu_{\underline{\b}',\underline{\b}}\simeq \nu_{\underline{\b}'',\underline{\b}}
%\]
%and satisfies $\nu_{\underline{\b},\underline{\b}}\simeq \id$.

\subsection{Conjugating braids and conjugating bimodules}
Suppose that $\b \sigma_s = \sigma_t \b$, so that Rouquier canonicity gives us a canonical homotopy equivalence $F(\b) F_s \simeq F_t F(\b)$. Our goal is to bootstrap from this a homotopy equivalence $F(\b) B_s \simeq B_t F(\b)$, and study how this homotopy equivalence behaves with respect to the morphisms in $\Diag$.

For the rest of this chapter we will frequently use cones of morphisms between chain complexes, and wish to abbreviate our notation. Given a degree zero chain map $\alpha \co X \to Y$ we have the mapping cone 
\begin{equation}
\Cone(\a):=\left( X[1] \buildrel \alpha \over \to Y \right) := \tw_{\left( \begin{array}{cc} 0 & 0 \\ \alpha & 0 \end{array} \right)} \left( X[1] \oplus Y \right).
\end{equation}
where in the twist above, we are regarding $\a$ as a degree 1 closed morphism $X[1] \to Y$.
%
%\begin{remark}
%If $f\colon X\rightarrow Y$ is a degree zero chain map then the usual mapping cone (or homotopy cofiber, or homotopy cokernel) of $f$ is $\operatorname{Cone}(f) = (X[1]\buildrel\a\over \to Y)$ in which the map $\a$ is the canonical degree 1 ``identity map'' $X[1]\to X$ followed by $f$.
%
%Dually, the cocone (or homotopy fiber, or homotopy kernel) is  $\operatorname{Cocone}(f) = (X\buildrel\a'\over \to Y[-1])$, in which $\a'$ is $f$ followed by the canonical degree 1 ``identity map'' $Y\to Y[-1]$.
%\end{remark}

\begin{prop}\label{prop:b commuting}
Let $s, t \in S$ (possibly equal) and suppose $\b$ is a braid such that $\b \sigma_s = \sigma_t \b$.  Then there is a homotopy equivalence $\phi \colon F(\b) B_s \to B_t F(\b)$ (constructed in the proof) such that the following square commutes up to homotopy:
\begin{equation} \label{phisquare}
\begin{tikzpicture}[anchorbase]
\node (a) at (0,0) {$F(\b) \star B_s$};
\node (b) at (0,2) {$B_t \star F(\b)$};
\node (c) at (3,0) {$F(\b)$};
\node (d) at (3,2) {$F(\b)$};
\path[-stealth,very thick]
(a) edge node[left] {$\phi$} (b)
(a) edge node[above] {$\id \star \e_s$} (c)
(b) edge node[above] {$\e_t\star  \id$} (d)
(c) edge node[right] {$\id$} (d);
\end{tikzpicture}
\end{equation}
\end{prop}
\begin{proof}
Throughout the proof we will employ a very common abuse of notation, viewing $\nu_s\in \uHom^{1,-1}(\one,F_s)$ as a degree $(1,0)$ map $\one(1)\to F_s$, or (when we discuss $\Cone(\nu_s)$) as a degree $(0,0)$ map $\one(1)[-1]\to F_s$.  

Observe that, since $F_s$ is a two term complex constructed from $B_s$ and $\one(1)$, it follows that $B_s$ is homotopy equivalent to a complex constructed from $F_s$ and $\one(1)$.  In slightly more details, consider the diagram
\begin{equation}
\begin{tikzpicture}
\node (B1) at (-3,0) {$B_s$};
\node (I1) at (-.5,.5) {$\one(1)$};
\node (B2) at (.5,-.5) {$B_s$};
\node (I2) at (0,2) {$\one(1)$};
\node (B3) at (3,0) {$B_s$};
\path[-stealth,very thick]
(B1) edge node[below] {$\id$} (B2)
(B1) edge node[above] {$\finaldotred$} (I1)
(I1) edge node[left] {$-\id$} (I2)
(B2) edge node[right] {$\finaldotred$} (I2)
(B2) edge node[below] {$\id$} (B3);
\end{tikzpicture}.
\end{equation}
The three middle objects represent $\Cone(-\nu_s)$, i.e.~ $(\one(1)\buildrel-\nu_s\over \longrightarrow F_s)$. The horizontal maps on the left give the homotopy equivalence $\iota_s \colon B_s \to \Cone(-\nu_s)$, and the horizontal maps on the right give an inverse homotopy equivalence $p_s \colon Cone(-\nu_s) \to B_s$. Under this homotopy equivalence, the morphism $\e_s=\finaldotred$ from $B_s$ to $\one(1)$ corresponds to the projection $\pr_{\one(1)}$ of $\Cone(-\nu_s) = \tw(\one(1) \oplus F_s)$ onto the $\one(1)$ summand (in homological degree zero). That is,
\begin{equation} \e_s = \pr_{\one(1)} \circ \iota_s, \qquad \e_s \circ p_s \simeq \pr_{\one(1)}. \end{equation}

%Consider the square
%\[
%\begin{tikzpicture}
%\node (a) at (0,0) {$F(\b)$};
%\node (b) at (3,0) {$F(\b s)$};
%\node (c) at (0,2) {$F(\b)$};
%\node (d) at (3,2) {$F(t \b)$};
%\path[-stealth,very thick]
%(a) edge node[above] {$- \id\star  \nu_{s}$} (b)
%(a) edge node[left] {$(-1)^{e(\b)} \id$} (c)
%(c) edge node[above] {$-\nu_{t} \star \id$} (d)
%(b) edge node[right] {$\nu_{t\b, \b s}$} (d);
%\end{tikzpicture}
%\]
%The square commutes up to homotopy by Proposition \ref{prop:res}, hence the vertical arrows are the components of a homotopy equivalence 
To construct $\phi$, we use the diagram
\begin{equation}\label{eq:phidef}
\begin{tikzpicture}[anchorbase]
\node (Bs) at (-4,0) {$F(\b)\star B_s$};
\node (Bt) at (-4,2) {$B_t\star F(\b)$};
\node (C1) at (0,0) {$F(\b) \star \Cone(-\nu_s)$};
\node (C2) at (0,2) {$\Cone(-\nu_t)\star F(\b)$};
\node (I1) at (3,0) {$\Big(F(\b)$};
\node (F1) at (6.5,0) {$F(\b)\star F_s\Big)$};
\node (I2) at (3,2) {$\Big(F(\b)$};
\node (F2) at (6.5,2) {$F_t\star F(\b)\Big)$};
\path[-stealth,very thick]
(I1) edge node[above] {$- \id_{F(\b)}\star  \nu_{s}$} (F1)
(I1) edge node[left] {$\id$} (I2)
(I2) edge node[above] {$-\nu_{t} \star \id_{F(\b)}$} (F2)
(F1) edge node[right] {$(-1)^{e(\b)} \nu_{t\b, \b s}$} (F2)
(Bs) edge node[left] {$\phi$} (Bt)
(Bs) edge node[above] {$\id_{F(\b)}\star \iota_s$} (C1)
(C2) edge node[above] {$p_t\star \id_{F(\b)}$} (Bt)
(C1) edge node[above] {$=$} (I1)
(I2) edge node[above] {$=$} (C2)
(C1) edge node[right] {$\phi'$} (C2);
\path[-stealth, dashed]
(I1) edge (F2);
\end{tikzpicture}
\end{equation}
The parenthesized entries in columns three and four are cones, being an elaboration of the complexes in the second column. 

By Proposition \ref{prop:res}, the map $-\id_{F(\b)} \star \nu_s$ within the differential on $F(\b) \otimes \Cone(-\nu_s)$ has residue $-(-1)^{e(\b)}$, while $-\nu_t \star \id_{F(\b)}$ has residue $(-1)$. Thus the residues of the two paths from $F(\b)$ to $F_t \star F(\b)$ agree. By Lemma \ref{lemma:canonicity2}, these two paths must give homotopic maps, i.e.
\[ (-1)^{e(\b)} \nu_{t\b, \b s}\circ (-\id_{F(\b)}\star \nu_s)\simeq -\nu_t\star \id_{F(\b)}. \]
Choose a homotopy $h$ for their difference (the dashed arrow). The sum of the two vertical maps and the dashed arrow defines a chain map $\phi' \colon F(\b) \star \Cone(-\nu_s) \to \Cone(-\nu_t) \star F(\b)$.  This chain map $\phi'$ is a homotopy equivalence since the vertical arrows are homotopy equivalences (given a morphism of distinguished triangles where the outer maps are isomorphisms, the inner map is also an isomorphism).

Now we compute that
\begin{align} \nonumber (\e_t \star \id_{F(\b)}) \circ \phi  & = ((\e_t \circ p_t)\star \id_{F(\b)}) \circ \phi' \circ (\id_{F(\b)} \star \iota_s) \simeq (\pr_{\one(1)} \star \id_{F(\b)}) \circ (\id_{F(\b)} \star \iota_s) \\ & = \pr_{F(\b)} \circ (\id_{F(\b)} \star \iota_s) \\ \nonumber & = (\id_{F(\b)} \star \pr_{\one(1)}) \circ (\id_{F(\b)} \star \iota_s) = \id_{F(\b)} \star \e_s. \end{align}
In the second line, $\pr_{F(\b)}$ denotes the projection from $\Cone(F(\b) \star -\nu_s)$ to $F(\b)$, and since $\pr_{\one(1)}$ has degree zero the Koszul sign rule does not produce any signs when passing from the second row to the third. Consequently, \eqref{phisquare} commutes up to homotopy.
\end{proof}

\begin{prop}\label{prop:b commuting alt}
Let $s, t \in S$ (possibly equal) and suppose $\b$ is a braid such that $\b \sigma_s = \sigma_t \b$.  Then there is a homotopy equivalence $\psi \colon F(\b) B_s \to B_t F(\b)$ such that the following square commutes up to homotopy:
\begin{equation} \label{psisquare}
\begin{tikzpicture}
\node (a) at (0,0) {$F(\b) \star B_s$};
\node (b) at (0,2) {$B_t \star F(\b)$};
\node (c) at (3,0) {$F(\b)$};
\node (d) at (3,2) {$F(\b)$};
\path[-stealth,very thick]
(a) edge node[left] {$\psi$} (b)
(c) edge node[above] {$\id \star \eta_s$} (a)
(d) edge node[above] {$\eta_t\star  \id$} (b)
(c) edge node[right] {$\id$} (d);
\end{tikzpicture}
\end{equation}
\end{prop}

\begin{proof} (Sketch) This follows by symmetry, applying the duality functor to all the arguments in this section and the last. \end{proof}
	
What is essential to understand is that there is no reason to expect that $\phi$ and $\psi$ are equal, and indeed they frequently are not!

\begin{example}
The homotopy equivalence $\phi \colon F_s\star B_s\to B_s\star F_s$ constructed in Proposition \ref{prop:b commuting} is
\begin{equation}
\begin{tikzpicture}[anchorbase]
\node (a) at (0,0) {$\Big(\underline{B_sB_s}$};
\node (b) at (3,0) {$B_s(1)\Big)$};
\node (c) at (0,2) {$\Big(\underline{B_sB_s}$};
\node (d) at (3,2) {$B_s(1)\Big)$};
\path[-stealth,very thick]
(a) edge node[above] {$\finaldotred\linered$}(b)
(c) edge node[above] {$\linered\finaldotred$} (d)
(a) edge node[left] {$-\linered\linered+\splitred\finaldotred$} (c)
(b) edge node[right] {$0$} (d);
\end{tikzpicture}
\end{equation}
It is a good exercise to verify that $\phi$ commutes with $\e_s$ as in \eqref{phisquare}. However, $\phi$ anticommutes with $\eta_s$, so $\psi := - \phi$ will satisfy \eqref{psisquare}.
\end{example}

\begin{example} Let $m_{st} = 3$. After tensoring the homotopy equivalence $F_s F_t B_s \to B_t F_s F_t$ with $F_s^{-1}$ on the left and $F_t^{-1}$ on the right, one obtains the homotopy equivalences in \cite[Proof of Lemma 4.8]{MMVFlat}. Moreover, \cite[Lemma 4.9]{MMVFlat} computes directly that both \eqref{phisquare} and \eqref{psisquare} commute for this same morphism.  \end{example}

\begin{prop} \label{prop:phivspsi} Let $s, t \in S$ (possibly equal) and suppose $\b$ is a braid such that $\b \sigma_s = \sigma_t \b$. Let $w\in W$ be the image of $\b$ under the
group homomorphism $\Br \to W$. Then there is a sign $\lambda = \pm 1 \in \k^\times$ such that $w(\alpha_s) = \lambda^{-1} \alpha_t$. Let $\phi$ be the homotopy equivalence in
Proposition \ref{prop:b commuting}. Then $\psi := \lambda \phi$ satisfies \eqref{psisquare}. \end{prop}

\begin{proof} Using the homotopy equivalence $\phi$, the hom complex $\uHom(F(\b)\star B_s,B_t\star F(\b))$ is homotopy equivalent to the hom complex $\uEnd(F(\b)B_s)$. Since
$F(\b)$ is invertible, this hom complex is homotopy equivalent to $\uEnd(B_s)$, which in degree zero is isomorphic to $\k$. Thus the homotopy equivalence $\psi$ from Proposition
\ref{prop:b commuting alt} must be $\lambda \phi$ for some $\lambda \in \k^{\times}$.

Now consider the following diagram (in which all morphisms are chain maps of homological degree zero):
\begin{equation}
\begin{tikzpicture}[anchorbase]
\node (F1) at (0,0) {$F(\b)$};
\node (F2) at (7,0) {$F(\b)$};
\node (F3) at (0,2) {$F(\b)$};
\node (F4) at (7,2) {$F(\b)$};
\node (B1) at (3.5,0) {$F(\b)\star B_s$};
\node (B2) at (3.5,2) {$B_t\star F(\b)$};
\path[-stealth,very thick]
(F3) edge node[above] {$\eta_t\star\id_{F(\b)}$} (B2)
(B2) edge node[above] {$\e_t\star \id_{F(\b)}$} (F4)
(F1) edge node[above] {$\id_{F(\b)}\star \eta_s$} (B1)
(B1) edge node[above] {$\id_{F(\b)}\star \e_s$} (F2)
(F1) edge node[left] {$\l\inv\id_{F(\b)}$} (F3)
(B1) edge node[left] {$\phi$} (B2)
(F2) edge node[left] {$\id_{F(\b)}$} (F4);
\end{tikzpicture}.
\end{equation}
Since $\a_s=\e_s\circ \eta_s$, we deduce that $F(\b) \alpha_s
\simeq \lambda^{-1} \alpha_t F(\b)$. By Lemma \ref{lem:conjpoly} we already know that $F(\b) \alpha_s \simeq w(\alpha_s) F(\b)$. Since $\uEnd(F(\b)) \cong \uEnd(\one) = R$, the left
action of $R$ on $F(\b)$ is free (even up to homotopy), and we deduce that $\lambda^{-1} \alpha_t = w(\alpha_s)$.

The rest follows from classical Coxeter theory. For a finite Coxeter group, let $\k'$ be the subring of $\R$ containing $2\cos(\frac{\pi}{m_{st}})$ for all $s, t \in S$. Any balanced realization is a specialization of a realization over $\k'$, so we can do the computation assuming $\k = \k'$. By restriction from $\R$ we have the usual notion of a root system living inside $V$ (but not necessarily spanning $V$). We know that $wsw^{-1}$ is the reflection corresponding to the root $w(\alpha_s)$, and hence $w(\alpha_s)$ is a root colinear with $\alpha_t$. The only roots colinear with $\alpha_t$ are $\pm \alpha_t$. Hence $\lambda = \pm 1$.
\end{proof}

\begin{remark} \label{rmk:partoneofremark} Let us reiterate a point made in the introduction. All we really need to know is that, inside the one-dimensional (up to homotopy) space
of chain maps $F(\b) \star B_s \to B_t \star F(\b)$, some nonzero chain map is a homotopy equivalence and satisfies \eqref{phisquare} up to scalar. (We just did a little more work
in Proposition \ref{prop:b commuting} to describe this chain map precisely.) Using duality, some nonzero chain map in the same hom space satisfies \eqref{psisquare} up to scalar.
By rescaling, we may choose $\phi$ which satisfies \eqref{phisquare} exactly; which scalar multiple of $\phi$ satisfies \eqref{psisquare} is determined by Lemma \ref{lem:conjpoly}
as in Proposition \ref{prop:phivspsi}. This remark continues in Remark \ref{rmk:continuation}. \end{remark}

\begin{remark} We have assumed above that the realization is balanced. When the realization is not balanced, there is no consistent theory of positive roots, and the scalar $\lambda$ need not be $\pm 1$. For example, for any $\lambda \in \k^{\times}$ there is an unbalanced realization of type $A_2$ with $st(\alpha_s) = \lambda^{-1} \alpha_t$. The proposition holds true verbatim (and the proof works ignoring the last paragraph) when one removes the statement that $\lambda = \pm 1$. \end{remark}

\subsection{Conjugation by the longest element}
\label{ss:HT}

Throughout this section, let $W$ be a finite Coxeter group with simple reflections $S\subset W$, length function $\ell$, and longest element $w_0\in W$.  We will take certain basic facts for granted, for instance $w_0\inv=w_0$ and $\ell(w_0 x)= \ell(x w_0) = \ell(w_0)-\ell(x)$.

Let $\tau:W\rightarrow W$ be the automorphism $\tau(x) = w_0x w_0$.  By standard arguments, $\tau$ is length preserving, hence $\tau(s)$ is a simple reflection, and $\tau$ itself corresponds to an automorphism of the Coxeter data; in particular $\tau$ is completely determined by its restriction to $\tau:S\rightarrow S$. Many finite Coxeter groups admit no non-trivial diagram automorphisms, hence $\tau(s)=s$ for all $s\in S$. Even when $\tau$ is the identity map, $w_0$ will act nontrivially on $V$.

\begin{example}
In type $A_{n-1}$ we have $W=S_n$ with simple reflections $s_1,\ldots,s_{n-1}$.  The longest element is $w_0 = s_1(s_2s_1)\cdots(s_{n-1}\cdots s_1)$, and we have $\tau(s_i) = s_{n-i+1}$ and $w_0(\alpha_i) = - \alpha_{\tau(i)}$.
\end{example}

\begin{example}
For dihedral groups (type $I_2(m)$) with simple reflections $s,t$ we have $w_0 = (st)^m = (ts)^m$, and
\[
\tau(s) = \begin{cases} t & \text{ if $m$ is odd}\\ s & \text{ if $m$ is even}\end{cases}, \qquad w_0(\alpha_s) = - \alpha_{\tau(s)}.
\]
\end{example}

\begin{lemma}\label{lemma:w_0 on simple roots}
For any (balanced) realization of a finite Coxeter system $(W,S)$, we have $w_0(\a_s) = -\a_{\tau(s)}$ for all $s\in S$.
\end{lemma}

\begin{proof}
This is completely standard. One can also use the last paragraph of the proof of Proposition \ref{prop:phivspsi}, together with the observation that $w_0$ sends positive roots to negative roots.
\end{proof}

\begin{definition}\label{def:HT}
For each $w\in W$ let $\b_w$ be a positive braid lift of a reduced expression for $w$.  Set $\Delta_w=F(\b_w)$ and $\nabla_w=F(\b_{w\inv}\inv)$.  For the longest element we will often write $\HT=\Delta_{w_0}$.
\end{definition}

\begin{lemma}\label{lemma:HT on obj}
For each $s\in S$ there is a homotopy equivalence $\phi \colon \HT B_s \simeq B_{\tau(s)}\HT$  such that the following squares commute up to homotopy:
\begin{equation}
\begin{tikzpicture}
\node (a) at (0,0) {$\HT \star B_s$};
\node (b) at (0,2) {$B_t \star \HT$};
\node (c) at (3,0) {$\HT$};
\node (d) at (3,2) {$\HT$};
\path[-stealth,very thick]
(a) edge node[left] {$\phi$} (b)
(a) edge node[above] {$\id \star \e_s$} (c)
(b) edge node[above] {$\e_t\star  \id$} (d)
(c) edge node[right] {$\id$} (d);
\end{tikzpicture}, \qquad 
\begin{tikzpicture}
\node (a) at (0,0) {$\HT \star B_s$};
\node (b) at (0,2) {$B_t \star \HT$};
\node (c) at (3,0) {$\HT$};
\node (d) at (3,2) {$\HT$};
\path[-stealth,very thick]
(a) edge node[left] {$-\phi$} (b)
(c) edge node[above] {$\id \star \eta_s$} (a)
(d) edge node[above] {$\eta_t\star  \id$} (b)
(c) edge node[right] {$\id$} (d);
\end{tikzpicture}
\end{equation}
% \[
% \begin{tikzpicture}
% \node (a) at (0,0) {$\HT\star  B_s$};
% \node (b) at (3,0) {$B_{\tau(s)}\star  \HT$};
% \node (c) at (0,2) {$\HT$};
% \node (d) at (3,2) {$\HT$};
% \path[-stealth,very thick]
% (a) edge node[above] {$\simeq$} (b)
% (a) edge node[left] {$\e_{\tau_s} \id$} (c)
% (b) edge node[right] {$\id \e_s$} (d)
% (c) edge node[above] {$\id$} (d);
% \end{tikzpicture}
% \]
\end{lemma}
\begin{proof}
Fix a simple reflection $s\in S$ and write $t:=\tau(s)$, so that $w_0s = tw_0$.  then we may write $\b_{w_0} = \b_{w_0s}\b_s$ and, similarly, $\b_{w_0}=\b_t\b_{tw_0}=\b_t\b_{w_0s}$.  Thus
\[
\b_t \b_{w_0} = \b_t \b_{w_0s} \b_s  = \b_{w_0} \b_s.
\]
Now the Lemma follows from Propositions \ref{prop:b commuting} and \ref{prop:phivspsi}.
\end{proof}

\begin{remark} Again, in the unbalanced case, one should replace the sign $-1$ in the second square above with the appropriate scalar $\lambda$. \end{remark}

\begin{theorem}\label{thm:HT action}
There exists a unique (up to unique isomorphism) monoidal autoequivalence $\Phi\colon \Diag\to \Diag$ such that $\Phi(B_s)=B_{\tau(s)}$, $\Phi(\e_s)=\e_{\tau(s)}$, and $\Phi(f) = w_0(f)$ for all $f\in R$.  Moreover, $\HT$ admits the structure of an object in $\ZS(\KC^b(\Diag),{}^\Phi \KC^b(\Diag))$.
\end{theorem}
\begin{proof} Clearly tensoring by $\HT$ on the left or the right is a fully faithful functor, as it is invertible. By Lemma \ref{lemma:HT on obj}, $\HT \otimes B_{s_1} \cdots B_{s_r} \cong B_{\tau(s_1)} \cdots B_{\tau(s_r)} \otimes \HT$, so the essential images of $L_{\HT}$ and $R_{\HT}$ agree. Now the existence of $\Phi$ was proven in Theorem \ref{thm:Phi and Z}.  By choosing an inverse to $R_Z^{-1}$ compatibly with the map $\phi$ from Lemma \ref{lemma:HT on obj}, we can assert that $\Phi$ satisfies $\Phi(B_s)=B_{\tau(s)}$ and $\Phi(\e_s)=\e_{\tau(s)}$. The action of $\Phi$ polynomials is computed by Lemma \ref{lem:conjpoly}.
	
By composing $\Phi$ with the automorphism of $\Diag$ induced by $\tau$, we get an automorphism of $\Diag$ which fixes $B_s$ and fixes $\e_s$ and sends $f \mapsto \tau(w_0(f))$ for each $f \in V \subset R$. By the results of \cite[Theorem 3.2 and Corollary 3.3]{EBigraded}, an automorphism of $\Diag$ fixing $B_s$ is uniquely determined by where it sends $\e_s$ (which will be a scalar multiple of $\e_s$) and where it sends each $f \in V$. Uniqueness of $\Phi$ is an immediate consequence.

Theorem \ref{thm:lifting drinfeld} completes the proof.
\end{proof}

\begin{remark} It was already noted in Remark \ref{rmk:Phiuptoisomorphism} that it really only matters what $\Phi$ is up to isomorphism of (monoidal graded) functors, because this
isomorphism can instead be baked into the isomorphism between $L_{\HT}$ and $R_{\HT}$. Given two automorphisms of $\Diag$ which fix all objects, a natural isomorphism between them
is necessarily a dilation: a multiple of the identity on each object. Moreover, if the automorphisms preserve the monoidal and graded structures, then a natural transformation is a
monoidal, graded dilation; it acts by $\lambda_s \in \k^{\times}$ on each generator $B_s$, and acts as $\prod \lambda_{s_i}$ on $B_{s_1}\cdots B_{s_r}(k)$ for any $k \in \Z$. Such an automorphism fixes $\one$, but it can rescale $\e_s$ by $\lambda_s$.

In particular, it is immediate from the classification of monoidal (graded) autoequivalences in \cite[Theorem 3.2 and Corollary 3.3]{EBigraded} that, up to isomorphism of such autoequivalences, an autoequivalence is uniquely determined by what it does to polynomials! For conjugation by braids, this is deterined already by Lemma \ref{lem:conjpoly}. \end{remark}

\begin{remark} \label{rmk:continuation}  This continues Remark \ref{rmk:partoneofremark}. By the previous remark, we need not compute the automorphism $\Phi$ associated
to conjugation by the full twist explicitly, but can deduce its properties (up to isomorphism) from first principles. Thus, knowing only that there is some homotopy equivalence
$\phi$ that satisfies \eqref{phisquare} up to scalar for $\b = \HT$, we deduce that there is some identification of $L_{\HT}$ and $R_{\HT}$ such that Theorem \ref{thm:HT action}
holds. Proposition \ref{prop:b commuting} does go further and allow one to pin down this isomorphism concretely using canonical maps. \end{remark}

For the rest of this section, we let $\Phi$ be the automorphism from Theorem \ref{thm:HT action}. Since $\Phi^2$ is the identity functor, the following is immediate.
\begin{corollary}
Let $\FT:=\HT^{\otimes 2}$ denote the ``full twist'' Rouquier complex.  Then $\FT$ admits the structure of an object in $\ZS(\KC^b(\Diag))$.\qed
\end{corollary}

\begin{proof} Let $\Phi$ be the automorphism from Theorem \ref{thm:HT action}. The result follows from the next lemma, since $\Phi^2$ is the identity functor.\end{proof}

\begin{lemma} \label{lem:composedrinfeld} Let $\AS$ be a $\k$-linear monoidal category.  Suppose $Z_1,Z_2\in \AS$ are objects and $\Phi_1,\Phi_2\colon \AS\to \AS$ are monoidal autoequivalences such that $Z_i$ lifts to an object of the $\Phi_i$-twisted Drinfeld centralizer. Then $Z_1\star Z_2$ lifts to an object of the $\Phi_1\circ \Phi_2$-twisted Drinfeld centralizer. \end{lemma}

\begin{proof} Obvious. \end{proof}

\subsection{Relative half-twists}

 Given $I\subset S$, we let $W_I\subset W$ be the parabolic subgroup generated by $s\in I$, with longest element denoted by $w_0(I)$.  In particular the longest element of $W$ may be denoted $w_0(S)$ for extra clarity.  Let
\[
w_0(S/I):= w_0(S)w_0(I)
\]
(so that $w_0(S/I)$ is a shortest length representative of the coset $w_0(S) W_I$ in $W/W_I$).  Let $\HT_I$ and $\HT_{S/I}$ denote the Rouquier complexes associated to the positive braid lift of $w_0(I)$ and $w_0(S/I)$, respectively.  Note that
\[
\HT_S \simeq \HT_{S/I}\HT_I.
\]
\begin{example}
In case $W=S_{n+1}$ with simple reflections $S=\{s_1,\ldots,s_{n}\}$, we may consider the parabolic subgroup $S_{n}\times S_1$, generated by $I=\{1,\ldots,n-1\}$.  In this case
\[
\HT_{S/I} = F(\sigma_1\cdots \sigma_{n}).
\]
\end{example}

\begin{example}\label{ex:thick half twist}
More generally, if $W=S_{n+m}$ and and $I=\{s_1,\ldots,\widehat{s}_n,\ldots, s_{n+m-1}\}$, then $W_I=S_n\times S_m$ and
\[
\HT_{S/I} = F\Big((\sigma_m\cdots \sigma_{n+m-1})\cdots (\sigma_2\cdots \sigma_{n+1}) (\sigma_1\cdots \sigma_n)\Big)
\ \ = \ \ F\left(\,
\begin{tikzpicture}[scale=1.5,baseline=.5cm]
\draw[very thick]
(1,0) -- (0,1)
(1.3,0) -- (.3,1)
(1.6,0) -- (.6,1);
\draw[line width=5pt,white]
(0,0) -- (1,1)
(.2,0) -- (1.2,1)
(.4,0) -- (1.4,1)
(.6,0) -- (1.6,1);
\draw[very thick]
(0,0) -- (1,1)
(.2,0) -- (1.2,1)
(.4,0) -- (1.4,1)
(.6,0) -- (1.6,1);
\node at (0,-.1) {\scriptsize$1$};
\node at (.3,-.1) {\scriptsize$\cdots$};
\node at (.6,-.1) {\scriptsize$n$};
\node at (1,-.1) {\scriptsize$1$};
\node at (1.3,-.1) {\scriptsize$\cdots$};
\node at (1.6,-.1) {\scriptsize$m$};
\end{tikzpicture}\,\right)
\]
as in \eqref{eq:type A rel HT}.
\end{example}

Given $I \subset S$, we consider $\Diag(W_I,I,V)$ as a full subcategory of $\Diag(W,S,V)$.   We have the inclusion $\iota\colon  \Diag(W_I,I,V)\rightarrow \Diag(W,S,V)$. 

\begin{theorem}\label{thm:relative HT}
For each subset $I\subset S$ the relative half twist $\HT_{S/I}$ has the structure of an object in the Drinfeld centralizer $\ZS(\KC^b(\AS), \KC^b(\MS))$ where $\AS:=\Diag(W_I,I,V)$ and $\MS=\Diag(W,S,V)$ with its $(\AS,\AS)$ $\star$-bimodule structure given by
\begin{equation}
X\cdot Z\cdot X' := (\Phi_S\circ \iota\circ \Phi_I(X))\star Z\star \iota(X'),
\end{equation}
where $\Phi_S$ is the automorphism $\Phi$ from Theorem \ref{thm:HT action}, and $\Phi_I$ is defined similarly but for $\HT_I$ acting on $\Diag(W_I,I,V)$.
\end{theorem}

\begin{proof} The result follows from Theorem \ref{thm:HT action} and Lemma \ref{lem:composedrinfeld}.  \end{proof}

\begin{remark}
Let $\tau_S\circ \iota\circ \tau_I$ denote the map $I \to S$ induced by $w_0(S/I)$. If $\tau_S\circ \iota\circ \tau_I(s) = t$ then the composition $\Phi_S\circ \iota\circ \Phi_I(X)$ will send $B_s\mapsto B_t$, $\e_s\mapsto \e_t$, and $\eta_s\mapsto \eta_t$, $\alpha_s \mapsto \alpha_t$ (with no signs). The action on polynomials agrees with the action of $w_0(S/I)$. Thus $\Phi_S\circ \iota\circ \Phi_I(X)$ is the functor induced by the map $\tau_S\circ \iota\circ \tau_I$ of Coxeter graphs, extended in the natural way to polynomials.
\end{remark}

\begin{example}\label{ex:thick half twist 2}
Let us continue Example \ref{ex:thick half twist}.  Let $W=S_{n+m}$ and $W_I=S_n\times S_m$.  Then Theorem \ref{thm:relative HT} gives us natural homotopy equivalences
\begin{equation}
\HT_{S/I}\star (X\boxtimes Y) \simeq (Y\boxtimes X)\star \HT_{S/I},
\end{equation}
which may be viewed as an isomorphism of functors $\KC^b(\Diag_n\otimes \Diag_m)\to \KC^b(\Diag_{n+m})$.
\end{example}

%\section{Application: Khovanov-Rozansky homology with $\SBim$-flavored sprinkles}

\section{$A_\infty$-natural transformations}
\label{s:infinity}
In this section we show how to extend Theorem \ref{thm:lifting lemma} to a kind of \emph{homotopy coherent natural transformation}, which we refer to as an $A_\infty$-natural transformation. These are defined using two-sided bar complex of a differential graded (dg) category.  In \S \ref{ss:dg cats} we recalled the basic definitions concerning dg categories, in \S \ref{ss:bimodules} we recall the notion of bimodules over dg categories, and in \S \ref{ss:bar} we recall the notion of the bar complex over a dg category.

\subsection{Bimodules over categories}
\label{ss:bimodules}
When discussing bimodules over dg categories, it will be convenient to have a notation for hom spaces with a ``right-to-left'' convention.  To this end we may write
\begin{equation}
\Hom_\CS(X\leftarrow X') := \Hom_\CS(X',X).
\end{equation}

Now, given dg categories $\CS$, $\DS$, a $(\CS,\DS)$-bimodule $B$ is a family of complexes $B(X,Y)$ parametrized by objects $X\in \CS$, $Y\in \DS$, together with action maps
\begin{equation}
\Hom_\CS(X'\leftarrow X)\otimes_\k B(X,Y)\otimes_\k \Hom_\DS(Y\leftarrow Y')\to B(X',Y') \, , \qquad f\otimes b\otimes g\mapsto f\cdot b\cdot g
\end{equation}
satisfying the usual constraints that $\id_X \cdot b \cdot \id_Y = b$  for all $b\in B(X,Y)$ and $(f'\circ f)\cdot b\cdot (g\circ g') = f'\cdot(f\cdot b\cdot g)\cdot g'$.

The identity $(\CS,\CS)$-bimodule $\ibim{\CS}$ is defined by
\begin{equation}
\ibim{\CS}(X,Y) = \Hom_\CS(X\leftarrow Y),
\end{equation}
with action defined by usual composition of morphisms in $\CS$.

Given $(\CS,\DS)$-bimodules $B_1$, $B_2$ a degree $l$ morphism $B_1\to B_2$ is by definition a collection of degree $l$ $\k$-linear maps $\tau_{X,Y}\colon B_1(X,Y)\to B_2(X,Y)$ such that
\begin{equation} \label{itsabimodulemap}
\tau_{X',Y'}(f\cdot b\cdot g) = (-1)^{l|f|} f\cdot \tau_{X,Y}(b)\cdot g
\end{equation}
for all $f\in \Hom_\CS(X'\leftarrow X)$, $b\in B(X,Y)$, $g\in \Hom_\DS(Y\leftarrow Y')$.  The differential of a bimodule map is defined so that $d(\tau)_{X,Y} = d(\tau_{X,Y})$. This makes the collection of $(\CS,\DS)$-bimodules into a dg category, denoted $\Bim_{\CS,\DS}$.

Any $(\CS,\CS)$ bimodule morphism $\ibim{\CS} \to B$ is determined uniquely via \eqref{itsabimodulemap} by the images $\tau_X := \tau_{X,X}(\id_X)$ of the identity maps of objects $X \in \CS$. Namely, $\tau(f) = \tau(\id_X) \cdot f$ for all morphisms $f \colon X' \rightarrow X$. However, the images $\tau_X$ are constrained by the requirement that
\begin{equation} \tau_X \cdot f = (-1)^{l|f|} f \cdot \tau_{X'} \end{equation}
for all $f \colon X' \rightarrow X$, as both equal $\tau(f)$.

The following similar result is unconstrained.

\begin{lemma}\label{lemma:hom from ketbra}
Let $\CS,\DS$ be dg categories, and let $M$ be a $(\CS,\DS)$-bimodule.  Given objects $X\in \CS$, $Y\in \DS$ we can consider the $(\CS,\DS)$-bimodule $\ibim{\CS}(-,X)\otimes_\k \ibim{\DS}(Y,-)$ (homs out of $X$, tensored with homs into $Y$).  Then the complex of bimodule homs satisfies
\begin{equation}
\Hom_{\Bim_{\CS,\DS}}\Big(\ibim{\CS}(-,X)\otimes_\k \ibim{\DS}(Y,-), M\Big) \cong M(X,Y).
\end{equation}
\end{lemma}
\begin{proof}
Any such bimodule morphism is determined uniquely via \eqref{itsabimodulemap} by the image of $\id_X \otimes \id_Y \in \ibim{\CS}(X,X)\otimes_\k \ibim{\DS}(Y,Y)$. This image is an element of $M(X,Y)$. We leave the rest as an exercise to the reader.
\end{proof}

\subsection{Representing natural transformations via bimodules}
\label{ss:representing nat}

\begin{definition}\label{def:MFG}
Suppose $\CS,\DS$ are $\k$-linear categories and $\FZ,\GZ\colon \CS\to \DS$ are functors.  Then we form the $(\CS,\CS)$-bimodule $M^{\GZ}_{\FZ}$, defined for each $X,X' \in \CS$ by
\begin{equation}\label{eq:MGF}
M^G_F(X,X'):=\Hom_\DS\Big(G(X)\leftarrow F(X')\Big).
\end{equation}
The bimodule structure is given by
\begin{equation}\label{eq:MGF bimod}
a\cdot m \cdot b = G(a)\circ m \circ F(b)
\end{equation}
for all $a\in \Hom_\CS(Y\leftarrow X)$, $m\in \Hom_\DS(G(X)\leftarrow F(X'))$, $b\in \Hom_\CS(X'\leftarrow Y')$.
\end{definition}

It is a simple exercise to verify that a natural transformation $\eta \colon F\to G$ is the same data as a bimodule morphism $\tau \colon \ibim{\CS}\to M^G_F$, where we set $\eta_X = \tau_X$ as above.

We can define similar concepts for dg categories. To this end, suppose $\CS,\DS$ are dg categories, and let $F,G\colon \CS\to \DS$ be dg functors.  We may form the dg $(\CS,\CS)$-bimodule $M^G_F$ exactly as before,   Now, a dg natural transformation to $F$ from $G$ is simply a (closed, degree zero) map from the trivial bimodule $\ibim{\CS} \to M^G_F$. For completeness, we do most of this exercise.

\begin{lemma}\label{lemma:representing nat trans}
With setup as above, the notions of a degree $l$ natural transformation $\GZ\rightarrow \FZ$ and a degree $l$ bimodule map $\ibim{\CS} \to M^{\FZ}_{\GZ}$ are equivalent.
\end{lemma}

\begin{proof}
Suppose $\tau\colon \ibim{\CS} \to M^{\FZ}_{\GZ}$ is a bimodule map.  The bimodule map condition means
\begin{equation}
\FZ(f)\circ  \tau(h) \circ \GZ(g)= (-1)^{l|f|} \tau(f\circ h\circ g),
\end{equation}
for all triples of composable morphisms $(f,h,g)$ in $\CS$.  Given a morphism $f\in \Hom_\CS(X\leftarrow Y)$, we may specialize to a triple of the form $(f,\id_Y,\id_Y)$, obtaining
\[
\FZ(f)\circ \tau(\id_Y) = (-1)^{l|f|} \tau(f),
\]
or we may specialize to a triple of the form $(\id_X,\id_X,f)$, obtaining
\[
\tau(\id_X)\circ \GZ(f) = \tau(f).
\]
Thus $\FZ(f)\circ \tau(\id_Y) = (-1)^{l|f|}\tau(\id_X)\circ \GZ(f)$, which is exactly the condition for $\tau(\id_X)$ to be a natural transformation $\FZ \leftarrow \GZ$.

 % This implies that the morphisms $\phi_X:=\tau(\id_X)\in \Hom_\CS(\FZ(X)\leftarrow \GZ(X))$ are the components of a degree $l$ natural transformation $\FZ\buildrel\phi\over\leftarrow \GZ$.

Conversely, given a degree $l$ natural transformation $\FZ\buildrel \phi\over \leftarrow \GZ$ we may define $\tau\colon \ibim{\CS}\to M^{\FZ}_{\GZ}$ by the formula
\begin{equation}
\tau(f):= \phi_X\circ \GZ(f)
\end{equation}
for all $f\in \Hom_{\CS}(X\leftarrow Y)$.
%Thus, $\tau$ is completely determined by the family of morphisms $\phi_X:=\tau(\id_X)\in \Hom_\CS(\FZ(X)\leftarrow \GZ(X))$, and the bimodule map condition for $\tau$ is equivalent to the naturality condition for the maps $\tau(\id_X)$.
\end{proof}

\begin{remark}
In this language, composition of natural transformations comes from the following construction. Composition of morphisms in $\DS$ gives rise to a bimodule map $M^G_F\otimes_\CS M^F_E\to M^G_E$. Given $\phi \colon \ibim{\CS}\to M^G_F$ and $\psi \colon \ibim{\CS}\to M^F_E$, the bimodule map
\begin{equation} \ibim{\CS} \buildrel \sim \over \to \ibim{\CS} \otimes_{\CS} \ibim{\CS} \buildrel \phi \otimes \psi \over \to M^G_F\otimes_\CS M^F_E\to M^G_E\end{equation}
corresponds to the composition of the natural transformations associated with $\phi$ and $\psi$.
\end{remark}

\subsection{The bar complex}
\label{ss:bar}

One may obtain a good notion of ``up-to-homotopy'' natural transformations by replacing the trivial bimodule in Lemma \ref{lemma:representing nat trans} by a projective resolution thereof. Let $\CS$ be a dg category whose hom complexes are projective over $\k$.  The projective resolution of the trivial bimodule $\ibim{\CS}$ is provided by the two-sided bar complex of $\CS$, which we now recall.

Firstly, we will be considering sequences of composable morphisms in $\CS$. Since composition is performed from right-to-left, our sequences of composable morphisms are best visualized via leftward arrows, as in the following diagram:
\begin{equation}
\begin{tikzpicture}
\node (a) at (0,0) {$X_{-1}$};
\node (b) at (2,0) {$X_0$};
\node (c) at (4,0) {$\cdots$};
\node (d) at (6,0) {$X_{r+1}$};
\path[-stealth]
(b) edge node[above] {$f_0$} (a)
(c) edge node[above] {$f_1$} (b)
(d) edge node[above] {$f_{r+1}$} (c);
\end{tikzpicture}
\end{equation}
Such a sequence will be denoted compactly as $[f_0,f_1,\cdots, f_{r+1}]$. We will regard these expressions as being linear in each argument so that they represent elements of the tensor product $\Hom_{\CS}(X_0,X_{-1})\otimes \Hom_\CS(X_1,X_0)\otimes\cdots \otimes \Hom_\CS(X_{r+1},X_r)$. All tensor products are over $\k$ unless otherwise specified.

We fix $X:=X_{-1}$ and $X':=X_{r+1}$, and take the direct sum over all $X_1,\ldots,X_r$, obtaining
\begin{equation}
\Bar_\CS^{-r}(X,X') = \bigoplus_{X_1,\ldots,X_r\in \CS} \Hom_{\CS}(X_{-1}\leftarrow X_0) \otimes \Hom_\CS(X_0\leftarrow X_1)\otimes\cdots \otimes \Hom_\CS(X_r\leftarrow X_{r+1}).
\end{equation}
Note that each summand above is itself a complex, as is $\Bar_{\CS}^{-r}(X,X')$ (which has the untwisted, direct sum differential). We make the direct sum $\bigoplus_{r\geq 0} \Bar_\CS^{-r}(X,X')[r]$ into a complex, denoted by $\Bar_\CS(X,X')$, with differential given by the internal differential plus the usual ``bar'' differential. That is, 
\begin{equation}
\Bar_\CS(X,X') := \tw_{\bard} \left(\bigoplus_{r \geq 0} \Bar_\CS^{-r}(X,X')[r] \right)
\end{equation}
where 
\begin{equation} \label{bardformula}
\bard \colon [f_0,\ldots, f_{r+1}] \mapsto \sum_{i=0}^r (-1)^i [f_0,\ldots,(f_i\circ f_{i+1}), \ldots,f_{r+1}].
\end{equation}
The complexes $\Bar_\CS(X,X')$ assemble into a $(\CS,\CS)$-bimodule, with left and right action defined by
\begin{equation}
\label{eq:bimodule}
g\cdot [f_0,f_1,\ldots,f_{r+1}]\cdot g' = (-1)^{r|g|}[g\circ f_0, f_1,\ldots, f_r,f_{r+1}\circ g']
\end{equation}
for all sequences of composable morphisms 
\begin{equation}
\begin{tikzpicture}
\node(z) at (-2,0) {$Y$};
\node (a) at (0,0) {$X_{-1}$};
\node (b) at (2,0) {$X_0$};
\node (c) at (4,0) {$\cdots$};
\node (d) at (6,0) {$X_{r+1}$};
\node (e) at (8,0) {$Y'$};
\path[-stealth]
(a) edge node[above] {$g$} (z)
(b) edge node[above] {$f_0$} (a)
(c) edge node[above] {$f_1$} (b)
(d) edge node[above] {$f_{r+1}$} (c)
(e) edge node[above] {$g'$} (d);
\end{tikzpicture}.
\end{equation}

Note that $\Bar_{\CS}(X,X')$ is a resolution of $\Bar_{\CS}^{-1}(X,X')$, where the augmentation map $\Bar^0_{\CS} \to \Bar^{-1}_{\CS}$ (still denoted $\bard$) still obeys the formula \eqref{bardformula}. Note that
$\Bar_{\CS}^{-1}(X,X') = \Hom_{\CS}(X \leftarrow X')$, so that $\Bar_{\CS}^{-1} \cong \ibim{\CS}$ is the identity bimodule. The element $[\id_X,\id_X] \in \Bar_{\CS}^0(X,X)$ is
sent by $\bard$ to $\id_X \in \ibim{\CS}(X,X)$.

For more on the bar complex, see for example \cite{WitherspoonBook}.

%
%{MH: An alternate sign convention:
%\[
%d([f_0,\ldots,f_{r+1}])  = %\sum_{i=0}^{r+1} (-1)^{i+1} [\ldots, d(f_i),\ldots] + \sum_{i=0}^r (-1)^i [\ldots,f_i\circ f_{i+1},\ldots]
%\]}

\subsection{$A_\infty$-natural transformations and lifting}
\label{ss:inftrans}

%Replacing the trivial bimodule $\CS$ with the bar complex $\Bar_\CS$ defines the notion of an $A_\infty$-transformation between dg functors.
 
\begin{definition}\label{def:infinity transformation}
Let $\FZ,\GZ\colon \CS\to \DS$ be dg functors.  A (closed) \emph{$A_\infty$-natural transformation $\GZ\rightarrow \FZ$} of degree $l$ is a (closed) degree $l$ bimodule map $\ttau\colon \Bar(\CS)\to M^{\FZ}_{\GZ}$.
\end{definition}

We begin by observing that possessing an $A_{\infty}$-natural transformation is a stronger condition than admitting a natural transformation up to homotopy.

\begin{lemma}\label{lemma:infty and homotopy transformation}
If $\ttau\colon \Bar(\CS)\to M^{\FZ}_{\GZ}$ is a closed degree zero $A_\infty$- transformation then the collection of maps $\ttau([\id_X,\id_X])\in \Hom_\DS(\FZ(X),\GZ(X))$ induces a natural transformation relating the induced functors $H^0(\CS)\to H^0(\DS)$.
\end{lemma}
\begin{proof}
Let $f$ be a degree zero closed morphism $X\leftarrow Y$ in $\CS$.  Consider the element of the bar complex $[\id_X,f,\id_Y]$.  The differential of this element is $[f,\id_Y]-[\id_X,f]$.  Since $\ttau$ is assumed to be a chain map, we have
\begin{equation}
d_\DS(\ttau([\id_X,f,\id_Y]))=\ttau([f,\id_Y])-\ttau([\id_X,f]).
\end{equation}
Now, the bimodule condition gives
\begin{equation}
d_\DS(\ttau([\id_X,f,\id_Y]))=f\circ \ttau([\id_Y,\id_Y])-\ttau([\id_X,\id_X])\circ f,
\end{equation}
which implies naturality up to homotopy.
\end{proof}

As discussed in the introduction, our lifting theorem takes place in two stages. The first stage concerns lifting homotopy natural transformations to $A_{\infty}$-natural transformations.

% \BE{expanded paragraph} Our lifting theorem now takes place in two stages. Suppose we have two dg functors $\FZ, \GZ \colon \AS \to \Ch^b(\BS)$ and a natural transformation between them up to homotopy. First we lift this to an $A_{\infty}$-natural transformation from $\FZ$ to $\GZ$. Then we lift this $A_{\infty}$-natural transformation to an $A_{\infty}$-natural transformation from $\Fdg$ to $\Gdg$, as functors $\Ch^b(\AS) \to \Ch^b(\BS)$. We address the first step now, and the second step in the next section.

\begin{theorem}\label{thm:infty trans 1}
Let $\AS$ be a $\k$-linear category and assume that all hom spaces in $\AS$ are projective over $\k$.  Let $F,G\colon \AS\to \Ch^b(\BS)$ be an unobstructed pair of functors (see Definition \ref{def:unobstructed}), and let $H^0(F),H^0(G) \colon \AS\to \KC^b(\BS)$ denote the induced functors.  Then any natural transformation  $H^0(F)\to H^0(G)$  extends to a unique-up-to-homotopy $A_\infty$-natural transformation $\Bar(\AS)\to M^{\GZ}_{\FZ}$.
\end{theorem}

The proof is an immediate consequence of the following general statement.

\begin{lemma}\label{lemma:infty trans 2}
Let $\AS$ be a $\k$-linear category and assume that all hom spaces in $\AS$ are projective over $\k$.  Let $M$ be an $(\AS,\AS)$-bimodule such that $H^k(M)=0$ for $k<0$.  Then any map of bimodules $\AS\to H^0(M)$ admits a unique-up-to-homotopy lift to a closed degree zero map $\Bar_\AS\to M$.
\end{lemma}

\begin{proof}
Consider the complex of bimodule maps $V:=\uHom_{\Bim_{\AS,\AS}}(\Bar_\AS,M)$.  This complex has a filtration by subcomplexes $\cdots \supset \FC^k\supset \FC^{k+1}\supset \cdots$
where $\FC^k\subset V$ is the subcomplex consisting of those bimodule maps which are zero on $\Bar^{-j}_\AS$ for $j<k$.  We claim that this filtration satisfies \eqref{filtconditions}.  
Indeed, $\cap_{k\geq 0} \FC^k=\{0\}$ is clear, and $V = \lim V/\FC^{k+1}$ because of the universal property of direct sums (since $\Bar_\AS$ is a direct sum of its chain bimodules $\Bar^{-k}_\AS[k]$). More plainly, any sum $v_0 + v_1 + \ldots$ with $v_k \in \FC^k$ will converge because the sum is finite when restricted to any given $\Bar^{-j}_{\AS}$.

In similar fashion to the proof of Lemma \ref{lemma:filtered hom cx}, observe that 
\begin{equation} \label{assgrblah} \FC^k/\FC^{k+1} \cong \Hom_{\Bim_{\AS,\AS}}(\Bar_{\AS}^{-k}, M) [k].\end{equation}
Note that $\Bar_{\AS}^{-k}$ is a direct sum of bimodules over the form 
\begin{equation}
\Bar_{\AS}^{-k} = \bigoplus_{X_0,\ldots,X_r} \ibim{\AS}(-,X_0)\otimes \ibim{\AS}(X_0,X_1)\otimes \cdots \otimes \ibim{\AS}(X_{r-1},X_r)\otimes \ibim{\AS}(X_r,-).
\end{equation}
where $\ibim{\AS}$ denotes the identity bimodule. 
Thus, using Lemma \ref{lemma:hom from ketbra} we have
\begin{equation} \label{homfrombarminusk}
\uHom_{\Bim_{\AS,\AS}}(\Bar_{\AS}^{-k}, M) \cong \prod_{X_0,\ldots,X_r} \uHom_\k\Big(\ibim{\AS}(X_0,X_1)\otimes \cdots \otimes \ibim{\AS}(X_{r-1},X_r) \, , \, M(X_0,X_r)\Big)
\end{equation}

By hypothesis, each $\ibim{\AS}(X_0,X_1)\otimes \cdots \otimes \ibim{\AS}(X_{r-1},X_r)$ is projective over $\k$ and $M(X_0,X_r)$ has zero homology in degrees $<0$. Hence 
\begin{equation}
H^l\Big(\uHom_{\Bim_{\AS,\AS}}(\Bar_{\AS}^{-k}, M)\Big) =0
\end{equation}
for $l<0$.  It follows that $\FC^k/\FC^{k+1}$ has zero homology in degrees $<k$. Consider the sequence of maps
\begin{equation}\label{eq:H0 sequence}
H^0(V)\to H^0(V/\FC^2) \to H^0(V/\FC^1).
\end{equation}
By Lemma \ref{lemma:filtered cx}, the first arrow is surjective and the composition is injective.

We now identify the terms in \eqref{eq:H0 sequence}. By \eqref{assgrblah} and \eqref{homfrombarminusk} we have
\begin{equation}
V/\FC^1 = \FC^0/\FC^1 \cong \prod_{X} M(X,X), \qquad \FC^1/\FC^2 \cong \prod_{X,Y}\uHom_\k(\left(\ibim{\AS}(X,Y),M(X,Y))\right)[1]
\end{equation}
Meanwhile we have
\begin{equation}
V / \FC^2 \ \cong \ \tw_{\Delta} \left( \underbrace{\prod_{X} M(X,X)}_{V/\FC^1}
\ \ \oplus \ \  \underbrace{\prod_{X,Y}\uHom_\k(\left(\ibim{\AS}(X,Y),M(X,Y))\right)[1]}_{\FC^1/\FC^2} \right)
\end{equation}
where $\Delta$ is a map from the first summand to the second induced by the bar differential. More precisely, $\Delta$ sends $\mathbf{m}=(m_X)_{X\in \AS} \in \prod_X M(X,X)$ to the mapping which (for all $X, Y \in \AS$) takes $f_{X,Y} \in \ibim{\AS}(X,Y)$ to $-f_{X,Y}\cdot m_Y + m_X\cdot f_{X,Y} \in M(X,Y)$.

Now, suppose we are given a bimodule map $\phi\colon \AS\to H^0(\MS)$.   For each $X\in \AS$ let $m_X\in M(X,X)$ be a degree zero cycle which represents $\phi(\id_X)\in H^0(M(X,X))$.  Let $\mathbf{m}=(m_X)_X\in V/\FC^1$. The bimodule map condition for $\phi$ implies that $\Delta(\mathbf{m})$ is homotopic to zero in $\FC^1/\FC^2$. Thus we can choose $h \in \FC^1/\FC^2$ whose differential is $-\Delta(\mathbf{m})$. Then $\mathbf{m}+h$ is a degree zero cycle in $V/\FC^2$. This has a lift to a degree zero cycle in $V=\uHom_{\Bim_{\AS,\AS}}(\Bar_\AS,M)$ surjectivity of the first map in \eqref{eq:H0 sequence}.  This lift is uniquely determined by $\mathbf{m}$ by injectivity of the composition in \eqref{eq:H0 sequence}.  This completes the proof.
\end{proof}

\subsection{Extending $A_\infty$-transformations to complexes}
\label{ss:extending infty transformations}

In this section we show that an $A_\infty$-transformation between functors $\FZ,\GZ\colon \AS\to \Ch^b(\BS)$ can be extended to an $A_\infty$-transformation between their lifts $\Fdg,\Gdg\colon \Ch^b(\AS)\to \Ch^b(\BS)$. The difficulty is that the bar complex of $\Ch^b(\AS)$ is not obviously related to that of $\AS$. We will need an intermediary between the bar complex of $\AS$ and the bar complex of $\Ch^b(\AS)$, that we refer to as the \emph{relative bar complex} of $\Ch^b(\AS)$ relative to $\AS$, denoted $\Bar_{\Ch^b(\AS),\AS}$.

The complex $\Bar_{\Ch^b(\AS),\AS}(\XB,\XB')$ is formally spanned by sequences of composable morphisms $[g_0 , \ldots, g_{r+1}]$, as in
\[
\begin{tikzpicture}
\node (a) at (0,0) {$\mathbf{X}$};
\node (b) at (2,0) {$Y_0$};
\node (c) at (4,0) {$\cdots$};
\node (d) at (6,0) {$Y_r$};
\node (e) at (8,0) {$\mathbf{X}'$};
\path[-stealth]
(b) edge node[above] {$g_0$} (a)
(c) edge node[above] {$g_1$} (b)
(d) edge node[above] {$g_r$} (c)
(e) edge node[above] {$g_{r+1}$} (d);
\end{tikzpicture}
\]
where $Y_i$ are objects of $\AS$ (which may be regarded as complexes sitting in degree 0), and $\XB,\XB'\in \Ch^b(\AS)$.  The $(\Ch^b(\AS),\Ch^b(\AS))$-bimodule structure on $\Bar_{\Ch^b(\AS),\AS}$ is given by
\begin{equation}
f\cdot [g_0, \ldots, g_{r+1}]\cdot f' = (-1)^{r|f|}[f\circ g_0, g_1, \ldots, g_r, g_{r+1}\circ f'].
\end{equation}

The paper \cite{MattsWorkInProgress} constructs a closed degree zero bimodule map $\Matt \colon \Bar_{\Ch^b(\AS)}\to  \Bar_{\Ch^b(\AS),\AS}$.  The construction of this map was outlined in \cite{GHW}, with some key details left as exercises. We recall the relevant formulas.

Before defining $\Matt$ we will need some setup.  First, let $\XB_0,\ldots,\XB_r\in \Ch^b(\AS)$ be complexes, with differentials denoted by $\boldsymbol{\d}_u$, and suppose we have a sequence of composable morphisms
\[
\begin{tikzpicture}
\node (a) at (0,0) {$\mathbf{X}$};
\node (b) at (2,0) {$\mathbf{X}_0$};
\node (c) at (4,0) {$\cdots$};
\node (d) at (6,0) {$\mathbf{X}_r$};
\node (e) at (8,0) {$\mathbf{X}'$};
\path[-stealth]
(b) edge node[above] {$f_0$} (a)
(c) edge node[above] {${f}_1$} (b)
(d) edge node[above] {${f}_r$} (c)
(e) edge node[above] {$f_{r+1}$} (d);
\end{tikzpicture}
\]
with $f_u\in \uHom_\AS(\XB_{u-1}\leftarrow \XB_u)$. This is an element of the bar complex of $\Ch^b(\AS)$, and we will replace it with a linear combination of morphisms factoring through the chain objects of each intermediate complex $\XB^u$. For each $0\leq u\leq r$ and each $a\in \Z$, let
\[
\begin{tikzpicture}
\node (a) at (0,0) {$X_u^a$};
\node (b) at (3,0) {$\XB_u$};
\path[-stealth, thick]
([yshift=2pt]a.east) edge node[above]  {$\sigma_u^a$} ([yshift=2pt]b.west)
([yshift=-2pt]b.west) edge node[below]  {$\pi_u^a$} ([yshift=-2pt]a.east);
\end{tikzpicture}
\]
be the inclusion and projection of the $a$-th chain object. Note that $\sigma_a$ has degree $a$, while $\pi_a$ has degree $-a$.  Given a sequence $\vec{a}=(a_0,\ldots,a_r)$, let
\begin{equation}
[f_0 \, |\, \ldots \, |\, f_{r+1}]^{\vec{a}} := 
 [f_0\circ \sigma_0^{a_0} \,   | \, \pi_0^{a_0}\circ f_1\circ \sigma_1^{a_1} \, | \, \ldots \, |  \, \pi_{r-1}^{a_{r-1}}\circ f_r\circ \sigma_r^{a_r} \, | \, \pi_r^{a_r} \circ f_{r+1}].
\end{equation}
%\sigma_a\colon \XB_u^a\to \XB_u$ denote the inclusion of the $a$-th chain group  and denote the inclusion of $\XB  Denote the differential of $\XB_u$ by $\boldsymbol{d}_u$.  For each

% Given such a sequence $\vec{a}$ we let $s=r+a_0+\cdots+a_r$.  \BE{there used to be $n$ in the definition of $s$, but I replaced with $a$, I think that's right. Also, there used to be a few $s$ in the equation above, but I tihnk it is $r$.}

\begin{definition}
Define $\Matt\colon \Bar_{\Ch^b(\AS)}\to \Bar_{\Ch^b(\AS),\AS}$ by 
\end{definition}\begin{equation}\label{eq:PhiPretr gfg}
\Matt([f_0 \sep \ldots \sep f_{r+1}]):=\sum_{\vec{n},\vec{a}}(-1)^{sgn} %{k_0+m_0+\cdots+m_s}
[f_0 \sep \d_0^{\ast n_0} \sep f_1 \sep \ldots  \sep \d_r^{\ast n_r} \sep f_{r+1}]^{\vec{a}}.
\end{equation}
Here $\d_u$ is the differential on $\XB_u$, and $\d_u^{\ast n_u}$ denotes the sequence $\d_u,\ldots,\d_u$ (of length $n_u$). The sum is over sequences $\vec{n}=(n_0,\ldots,n_r)\in \Z_{\geq 0}^{r+1}$ and $\vec{a}=(a_0,\ldots,a_s)\in \Z^{s+1}$ where $s=r+n_0+\cdots+n_r$. Finally, letting $k_u$ be the degree of $f_u$, we have
\begin{equation}
sgn = \binom{s-r+1}{2}+\sum_{0\leq u\leq r} (r-u)n_u+\sum_{0\leq u\leq u'\leq r}k_u n_{u'}.
\end{equation}

%This map $\Phi_{\pretr}$ is a homotopy equivalence in $\Bim_{\pretr(\CS),\pretr(\CS)}$.

\begin{example}
In case $r=0$ this specializes to 
\begin{equation}
\Matt([\id_{\XB_0},\id_{\XB_0}]):=
\sum_{n_0,a}(-1)^{\binom{n_0+1}{2}} 
[\id_{\XB_0},\d_0^{\ast n_0},\id_{\XB_0}]^{\vec{a}}
\end{equation}
where the sum is over $n_0\in \Z_{\geq 0}$ and vectors $\vec{a}$ of the form
\[
\vec{a} = (a_0,\ldots,a_{n_0}).
\]
Now, since $\d_0$ has degree $1$, the component $\pi_0^{a_{v-1}}\circ \d_0\circ \sigma_0^{a_v}$ is zero unless $a_{v-1}=a_v+1$ for all $1\leq v\leq n_0$.  Thus, the sum over $\vec{a}\in \Z^{n_0+1}$ may be reduced to a sum over $a_0\in \Z$.
\end{example}

\begin{example}
In case $r=1$ this specializes to
\begin{equation}
\Matt([\id_{\XB_0}\sep f_1\sep\id_{\XB_1}]) = \sum_{n_0,n_1,\vec{a}} (-1)^{\binom{n_0+n_1+1}{2}+n_0+k_1n_1} [\id_{\XB_0}\sep\d_0^{\ast n_0}\sep f_1 \sep \d_1^{\ast n_1}\sep\id_{\XB_1}]^{\vec{a}}.
\end{equation}
\end{example}

\begin{thm}[\cite{MattsWorkInProgress}] The map $\Matt$ is a degree zero chain map, and a homotopy equivalence. \end{thm}

\begin{thm}
Let $\FZ,\GZ\colon \AS\to \Ch^b(\BS)$ be dg functors.  Let $\Fdg,\Gdg\colon \Ch^b(\AS)\to \Ch^b(\BS)$ denote the lifts of $\FZ$, $\GZ$.  Let $\tau\colon \Bar_\AS\to M^{\FZ}_{\GZ}$ be an $A_\infty$-transformation.  Then:
\begin{enumerate}
\item We may lift $\tau$ to a closed degree zero bimodule map $\tau'\colon \Bar_{\Ch^b(\AS),\AS}\to M^{\Fdg}_{\Gdg}$ via
\begin{equation}
\tau'([f_0,\ldots,f_{r+1}])
= (-1)^{r|f_0|} \Fdg(f_0)\circ \tau\Big([\id_{X_0}, f_1, \ldots, f_r, \id_{X_r}]\Big) \circ \Gdg(f_{r+1})
\end{equation}
\item We have an induced closed degree zero bimodule map $\ttau\colon \Bar_{\Ch^b(\AS)}\to M^{\Fdg}_{\Gdg}$ defined by
\begin{equation} \ttau := \tau' \circ \Matt. \end{equation}
% \[
% \ttau([f_0\sep \ldots \sep f_{r+1}]) = \sum_{\vec{n},\vec{a}}  (-1)^{sgn} \tau'([f_0 \sep \cdots \sep f_{r+1}]^{\vec{a}})
% \]
\end{enumerate} 
\end{thm}

\begin{proof} The proof of part (1) is completely straightforward; the map $\tau'$ is defined on the elements $[\id_{X_0},f_1, \ldots, f_r, \id_{X_r}]$ in the obvious way, and extended from there using the bimodule axioms. Part (2) has no content, but merely states the desired conclusion of this process. \end{proof}

\section{$A_\infty$-Drinfeld centralizers}
\label{s:infty drinfeld}

In this chapter we read diagrams as morphisms from right to left (in contrast to the introduction, where we read from bottom to top).

In this section we introduce the notion of the $A_\infty$-Drinfeld center of a dg category.  A more restricted kind of dg Drinfeld center was introduced in earlier work \cite{GHW}, but with a few flaws.  We correct those flaws by bringing $A_\infty$ structures into the mix.  The central ideas are summarized below.

The most convenient language to express the $A_\infty$-Drinfeld center is the language of categories which are enriched in a (nonsymmetric) monoidal dg category $\BB$. If $A\in \BB$ is an ($A_\infty$) algebra and $\CB$ is a category enriched in $\BB$, then we may discuss ($A_\infty$) $A$-modules in $\CB$.  Since the algebra and its modules live in different categories, our treatment of such modules must deviate from the tradition.  For instance, in the usual story a morphism between $A_\infty$ $A$-modules is a family of morphisms $A^{\otimes n}\otimes M\to N$, but this becomes nonsensical as soon as $A$ lives in a different category from $M,N$.

\begin{remark} We \textbf{warn} the reader that we will use two distinct notational schemes for discussing dg monoidal categories, $(\CS,\star,\one)$ and $(\BB,\diamond,\oone)$, in order to emphasize the two very different roles played by them.  One should think of $(\CS,\star,\one)$ as being something akin to the diagrammatic Hecke category $\Diag$.  In contrast, one should think of $\BB$ is something like the dg category of dg $(\Diag,\Diag)$-bimodules but with monoidal structure $\diamond$ as defined in \S \ref{ss:diamond}.
\end{remark}

\subsection{Diagrammatics for bimodules}
\label{ss:bim diagrams}
We will adopt a graphical notation for bimodules over (dg) categories, as well as elements of those bimodules.  We will draw the strands in our diagrams for bimodules thicker than the strands in our diagrams for bimodule elements, to disambiguate.

We will denote the identity bimodule $\ibim{\AS}$ as a horizontal black line, and an element $f\in \ibim{\AS}(X,X')$ of the identity bimodule will be denoted by the usual diagrammatics for monoidal categories (with a right-to-left orientation), as in
\begin{equation}
\ibim{\AS} \ \ = \ \ 
\begin{tikzpicture}[baseline=-.1cm]
\draw[line width=2.5pt] (-.7,0) --(.7,0);
\node at (0,.3) {\scriptsize$\AS$};
\end{tikzpicture}
\ , \qquad\quad
\begin{tikzpicture}[anchorbase]
\draw[ultra thick] (.7,0) --(-.8,0);
\filldraw[fill=white]
(0,0) circle (.3cm);
\node at (0,0) {\scriptsize$f$};
\node at (-1,0) {\scriptsize$X$};
\node at (1,0) {\scriptsize$X'$};
\end{tikzpicture}
\ \ \ \in \ \ibim{\AS}(X,X').
\end{equation}

Given bimodules $M\in \Bim_{\AS,\BS}$ and $N\in \Bim_{\BS,\CS}$, their tensor product $M\otimes_\BS N$ is depicted diagrammatically as in
\begin{equation}\label{eq:MN diagram}
\begin{tikzpicture}[baseline=-.1cm]
\draw[line width = 2.5pt]
(-.7,0) -- (.7,0);
\filldraw[fill=white]
(-.3,-.25) rectangle (.3,.25);
\node at (0,0) {\scriptsize$M$};
\node at (.6,.3) {\scriptsize$\BS$};
\node at (-.6,.3) {\scriptsize$\AS$};
\end{tikzpicture}
\otimes_\BS
\begin{tikzpicture}[baseline=-.1cm]
\draw[line width = 2.5pt]
(-.7,0) -- (.7,0);
\filldraw[fill=white]
(-.3,-.25) rectangle (.3,.25);
\node at (0,0) {\scriptsize$N$};
\node at (.6,.3) {\scriptsize$\CS$};
\node at (-.6,.3) {\scriptsize$\BS$};
\end{tikzpicture}
\ \ \ := \ \ \ 
\begin{tikzpicture}[baseline=-.1cm]
\draw[line width = 2.5pt]
(-1.2,0) -- (1.2,0);
\filldraw[fill=white]
(-.9,-.25) rectangle (-.3,.25)
(.9,-.25) rectangle (.3,.25);
\node at (-.6,0) {\scriptsize$M$};
\node at (.6,0) {\scriptsize$N$};
\node at (0,.3) {\scriptsize$\BS$};
\node at (-1.1,.3) {\scriptsize$\AS$};
\node at (1.1,.3) {\scriptsize$\CS$};
\end{tikzpicture}.
\end{equation}
A simple tensor $m\otimes n\in M\otimes_\BS N$ (with $m\in M(X,Y)$ and $n\in N(Y,Z)$ will be depicted diagrammatically as in
\[
\begin{tikzpicture}[anchorbase]
\draw[very thick]
(-1.2,0) -- (1.2,0);
\filldraw[fill=white]
(-.6,0) circle (.3)
(.6,0) circle (.3);
\node at (-.6,0) {\scriptsize$m$};
\node at (.6,0) {\scriptsize$n$};
\node at (0,.3) {\scriptsize$Y$};
\node at (-1.1,.3) {\scriptsize$X$};
\node at (1.1,.3) {\scriptsize$Z$};
\end{tikzpicture}.
\]
The fact that we are tensoring over $\BS$ means that morphisms in $\BS$ may be transported from $m$ to $n$, as in
\begin{equation}\label{eq:morphism slide}
\begin{tikzpicture}[baseline=-.1cm]
\draw[very thick]
(-1.2,0) -- (1.2,0);
\filldraw[fill=white,rounded corners]
(-1,-.25) rectangle (-.2,.25)
(.9,-.25) rectangle (.3,.25);
\node at (-.6,0) {\scriptsize$m\cdot f$};
\node at (.6,0) {\scriptsize$n$};
\node at (.1,.3) {\scriptsize$Y'$};
\node at (-1.1,.3) {\scriptsize$X$};
\node at (1.1,.3) {\scriptsize$Z$};
\end{tikzpicture}
\ \ \ = \ \ 
\begin{tikzpicture}[baseline=-.1cm]
\draw[very thick]
(-1.2,0) -- (1.2,0);
\filldraw[fill=white,rounded corners]
(1,-.25) rectangle (.2,.25)
(-.9,-.25) rectangle (-.3,.25);
\node at (-.6,0) {\scriptsize$m$};
\node at (.6,0) {\scriptsize$f\cdot n$};
\node at (0,.3) {\scriptsize$Y$};
\node at (-1.1,.3) {\scriptsize$X$};
\node at (1.1,.3) {\scriptsize$Z$};
\end{tikzpicture}
\end{equation}

\subsection{The $\diamond$-product on bimodules}
\label{ss:diamond}

\begin{defn}
If $\AS$ is a dg monoidal category we will let $\mergebim{\AS}\in \Bim_{\AS,\AS\otimes \AS}$ and $\splitbim{\AS}\in \Bim_{\AS\otimes \AS,\AS}$ be the bimodules given by
\[
\mergebim{\AS}(X, X_1\otimes X_2) = \Hom_\AS(X\leftarrow X_1\star X_2) \, , \qquad \splitbim{\AS}(Y_1\otimes Y_2,Y) = \Hom_\AS(Y_1\star Y_2\leftarrow Y).
\]
If $\AS$ and $\BS$ are dg monoidal categories, then given a pair of bimodules $M,N\in \Bim_{\AS,\BS}$ we let $M\diamond N \in \Bim_{\AS,\BS}$ be the bimodule defined by
\begin{equation}
\mergebim{\AS}\otimes_{\AS\otimes \AS} (M\otimes N)\otimes_{\BS\otimes \BS}\splitbim{\BS}.
\end{equation}
\end{defn}

Diagrammatically, we depict the bimodule $\mergebim{\AS}$ (and an element $a\in \mergebim{\AS}(X,X_1\star X_2)$ therein) as in
\begin{equation}\label{eq:mergebim diagram}
\mergebim{\AS} \ \ = \ \ 
\begin{tikzpicture}[anchorbase]
\draw[line width=2.5pt]
(-.7,0)--(0,0)..controls++(.2,.2) and ++(-.2,0)..(.8,.5)
(0,0)..controls++(.2,-.2) and ++(-.2,0)..(.8,-.5);
\node at (1.1,.5) {\scriptsize $\AS$};
\node at (1.1,-.5) {\scriptsize $\AS$};
\node at (-1.1,0) {\scriptsize $\AS$};
\end{tikzpicture}
\ \qquad \left(\text{respectively } \ \ 
\begin{tikzpicture}[baseline=0cm]
\draw[ultra thick,stealth-]
(-.9,0)--(0,0)..controls++(.2,.2) and ++(-.2,0)..(.8,.5);
\draw[ultra thick,stealth-]
(0,0)..controls++(.2,-.2) and ++(-.2,0)..(.8,-.5);
\filldraw[fill=white]
(0,0) circle (.3cm);
\node at (0,0) {\scriptsize $a$};
\node at (1.1,.5) {\scriptsize $X_2$};
\node at (1.1,-.5) {\scriptsize $X_1$};
\node at (-1.1,0) {\scriptsize $X$};
\end{tikzpicture}\right).
\end{equation}
Similarly we depict the bimodule $\splitbim{\AS}$ (and an element $b\in \splitbim{\AS}(Y_1\star Y_2,Y)$ therein) diagrammatically as in
\begin{equation}\label{eq:splitbim diagram}
\splitbim{\AS} \ \ = \ \ 
\begin{tikzpicture}[anchorbase,xscale=-1]
\draw[line width=2.5pt]
(-.7,0)--(0,0)..controls++(.2,.2) and ++(-.2,0)..(.8,.5)
(0,0)..controls++(.2,-.2) and ++(-.2,0)..(.8,-.5);
\node at (1.1,.5) {\scriptsize $\AS$};
\node at (1.1,-.5) {\scriptsize $\AS$};
\node at (-1.1,0) {\scriptsize $\AS$};
\end{tikzpicture}
\qquad \left(\text{respectively } \ \ 
\begin{tikzpicture}[baseline=0cm,xscale=-1]
\draw[ultra thick,-stealth]
(-.9,0)--(0,0)..controls++(.2,.2) and ++(-.2,0)..(.8,.5);
\draw[ultra thick,-stealth]
(0,0)..controls++(.2,-.2) and ++(-.2,0)..(.8,-.5);
\filldraw[fill=white]
(0,0) circle (.3cm);
\node at (0,0) {\scriptsize $b$};
\node at (1.1,.5) {\scriptsize $Y_2$};
\node at (1.1,-.5) {\scriptsize $Y_1$};
\node at (-1.1,0) {\scriptsize $Y$};
\end{tikzpicture}\right).
\end{equation}

Accordingly, we may visualize the bimodule $M\diamond N$ (and an element $a\otimes (m\otimes n)\otimes b\in (M\diamond N)(X,Y)$ therein) diagrammatically as
\begin{equation} \label{MdiamondN}
M\diamond N \ \ = \ \ 
\begin{tikzpicture}[anchorbase]
\draw[line width = 2.5pt]
(-1.7,0)--(-.95,0)..controls++(.2,.2) and ++(-.2,0)..(-.2,.5)
(-.95,0)..controls++(.2,-.2) and ++(-.2,0)..(-.2,-.5)
(1.7,0)--(.95,0)..controls++(-.2,.2) and ++(.2,0)..(.2,.5)
(.95,0)..controls++(-.2,-.2) and ++(.2,0)..(.2,-.5);
\filldraw[fill=white]
(-.3,.25) rectangle (.3,.75)
(-.3,-.25) rectangle (.3,-.75);
\node at (0,.5) {\scriptsize $M$};
\node at (0,-.5) {\scriptsize $N$};
\node at (-.7,.6) {\scriptsize $\AS$};
\node at (-.7,-.6) {\scriptsize $\AS$};
\node at (.7,.6) {\scriptsize $\BS$};
\node at (.7,-.6) {\scriptsize $\BS$};
\node at (-2,0) {\scriptsize $\AS$};
\node at (2,0) {\scriptsize $\BS$};
\end{tikzpicture}
\qquad \left(\text{respectively } \ \ 
\begin{tikzpicture}[anchorbase]
\draw[ultra thick]
(-1.7,0)--(-.95,0)..controls++(.2,.2) and ++(-.2,0)..(-.2,.5)
(-.95,0)..controls++(.2,-.2) and ++(-.2,0)..(-.2,-.5)
(1.7,0)--(.95,0)..controls++(-.2,.2) and ++(.2,0)..(.2,.5)
(.95,0)..controls++(-.2,-.2) and ++(.2,0)..(.2,-.5);
\filldraw[fill=white]
(0,.5) circle (.3cm)
(0,-.5) circle (.3cm)
(-.95,0) circle (.3cm)
(.95,0) circle (.3cm);
\node at (0,.5) {\scriptsize $m$};
\node at (0,-.5) {\scriptsize $n$};
\node at (-.95,0) {\scriptsize $a$};
\node at (.95,0) {\scriptsize $b$};
\node at (-.7,.6) {\scriptsize $X_2$};
\node at (-.7,-.6) {\scriptsize $X_1$};
\node at (.7,.6) {\scriptsize $Y_2$};
\node at (.7,-.6) {\scriptsize $Y_1$};
\node at (-2,0) {\scriptsize $X$};
\node at (2,0) {\scriptsize $Y$};
\end{tikzpicture}\right)
\end{equation}
where
\[
\begin{tikzpicture}[baseline=-.1cm]
\node (a) at (0,0) {$X$};
\node (b) at (2,0) {$X_1\star X_2$};
\path[-stealth,very thick]
(b) edge node[above] {$a$} (a);
\end{tikzpicture}
\qquad
m\in M(X_1,Y_1)
\qquad
n\in N(X_2,Y_2)
\qquad
\begin{tikzpicture}[baseline=-.1cm]
\node (a) at (0,0) {$Y_1\star Y_2$};
\node (b) at (2,0) {$Y$};
\path[-stealth,very thick]
(b) edge node[above] {$b$} (a);
\end{tikzpicture}.
\]
General elements of $M\diamond N$ are linear combinations of the pure tensors depicted above. Note that the diagram for $a\otimes(m\otimes n)\otimes b$ should be considered as a formal picture (it is not a composition of morphisms), but the fact that we tensor over $\AS \otimes \AS$ and $\BS \otimes \BS$ implies that all relevant local relations of the form \eqref{eq:morphism slide} are satisfied.  

Corollary \ref{cor:whatdiamondrepresents} below gives another way to view the diamond product, in terms of how it affects morphism spaces.

\begin{lemma}
The bimodule $\splitbim{\AS}$ is right dual to the bimodule $\mergebim{\AS}$. That is, we have evaluation and coevaluation maps
\begin{equation}
\mergebim{\AS}\otimes_{\AS\otimes \AS}\splitbim{\AS}\to \ibim{\AS} \ , \qquad
\ibim{\AS}\otimes\ibim{\AS}\to \splitbim{\AS}\otimes_\AS \mergebim{\AS}
\end{equation}
giving rise to adjunctions 
\begin{equation} \mergebim{\AS} \otimes_{\AS \otimes \AS} (-) \vdash \splitbim{\AS} \otimes_{\AS} (-), \qquad (-) \otimes_{\AS \otimes \AS} \splitbim{\AS} \vdash (-) \otimes_{\AS} \mergebim{\AS}. \end{equation}
\end{lemma}
\begin{proof}
The evaluation map sends $a\otimes b\mapsto a\circ b$ for $a\in \Hom_\AS(X\leftarrow X_1\star X_2)$ and $b\in \Hom_\AS(X_1\star X_2\leftarrow Y)$.  The coevaluation map sends $f\otimes g\mapsto (f\star g)\otimes \id_{X'\star Y'} = \id_{X \star Y} \otimes (f \star g)$ for $f\in \ibim{\AS}(X,X')$, $g\in \ibim{\AS}(Y,Y')$.  We leave the details to the reader.
\end{proof}

\begin{corollary} \label{cor:whatdiamondrepresents}
Given bimodules $L,M,N\in \Bim_{\AS,\BS}$ we have a natural isomorphism
\begin{equation} \label{diamondadjunction}
\Hom_{\Bim_{\AS,\BS}}(M\diamond N , L) \cong \Hom_{\Bim_{\AS\otimes \AS,\BS\otimes \BS}}(M\otimes N, \splitbim{\AS}\otimes_\AS L \otimes_\BS \mergebim{\BS}).
\end{equation}
\end{corollary}

\begin{proof} Consider
\begin{align*}
\Hom_{\Bim_{\AS,\BS}}(M\diamond N , L) &= \Hom_{\Bim_{\AS,\BS}}(\mergebim{\AS}\otimes_{\AS\otimes\AS}(M\otimes N)\otimes_{\BS\otimes \BS}\splitbim{\BS} , L)\\
&\cong \Hom_{\Bim_{\AS\otimes \AS,\BS\otimes \BS}}(M\otimes N, \splitbim{\AS}\otimes_\AS L \otimes_\BS \mergebim{\BS}).
\end{align*}
The first line is tautological and the second line is an application of the adjunctions from the previous lemma. \end{proof}

\subsection{The $\diamond$-unit}
\label{ss:diamond unit}
Next we would like to discuss the unit bimodule $\oone_\diamond$ with respect to the $\diamond$ operation.  Just like the $\diamond$-product itself is built from the merge and split bimodules for $\AS$ and $\BS$, the unit for the $\diamond$-product will be built from two pieces, defined next.

\begin{definition}\label{def:unit and counit bim}
If $\AS$ is a monoidal dg category, define bimodules $\ubim{\AS}\in \Bim_{\AS,\k}$ and $\coubim{\AS}\in \Bim_{\k,\AS}$ by
\[
\ubim{\AS}(X,\cdot) = \Hom_\AS(X\leftarrow \one_\AS) \ , \qquad \coubim{\AS}(\cdot,Y) = \Hom_\AS(\one\leftarrow Y).
\]
Here, $\k$ is regarded as a dg category with one object (denoted $\cdot$), and endomorphism algebra $\End_\k(\cdot)=\k$.  If $\AS$ and $\BS$ are dg monoidal categories then we let $\oone_\diamond\in \Bim_{\AS,\BS}$ be the bimodule defined by
\begin{equation}
\oone_{\diamond} = \ubim{\AS}\otimes_\k \coubim{\BS}
\end{equation}
i.e.~ $\oone_\diamond(X,Y) = \Hom_\AS(X\leftarrow\one_\AS)\otimes_\k \Hom_\BS(\one_\BS\leftarrow Y)$.
\end{definition}

We draw $\ubim{\AS}$ (and an element $a\in \ubim{\AS}(X,\cdot)$ therein) as
\begin{equation}
\ubim{\AS} \ \ = \ \ 
\begin{tikzpicture}[anchorbase]
\draw[line width =2.5pt]
(0,0)--(-.7,0);
\filldraw[fill=black]
(0,0) circle (.1);
\node at (-1,0) {\scriptsize$\AS$};
\end{tikzpicture}
\qquad \left(\text{respectively } \ \ 
\begin{tikzpicture}[anchorbase]
\draw[very thick, -stealth]
(0,0)--(-.8,0);
\filldraw[fill=white]
(0,0) circle (.3);
\node at (-1,0) {\scriptsize$X$};
\node at (0,0) {\scriptsize$a$};
\end{tikzpicture}
\right).
\end{equation}
Similarly, we draw $\coubim{\AS}$ (and an element $b\in \coubim{\AS}(\cdot,Y)$ therein) as
\begin{equation}
\coubim{\AS} \ \ = \ \ 
\begin{tikzpicture}[anchorbase,xscale=-1]
\draw[line width =2.5pt]
(0,0)--(-.7,0);
\filldraw[fill=black]
(0,0) circle (.1);
\node at (-1,0) {\scriptsize$\AS$};
\end{tikzpicture}
\qquad \left(\text{respectively } \ \ 
\begin{tikzpicture}[anchorbase,xscale=-1]
\draw[very thick, stealth-]
(-.3,0)--(-.8,0);
\filldraw[fill=white]
(0,0) circle (.3);
\node at (-1,0) {\scriptsize$Y$};
\node at (0,0) {\scriptsize$b$};
\end{tikzpicture}
\right).
\end{equation}
Accordingly, we visualize the bimodule $\oone_{\diamond}$ (and an element $a \otimes b \in \Hom_\AS(X\leftarrow \one_\AS)\otimes_\k \Hom_\BS(\one_\BS\leftarrow Y)$ therein) diagrammatically as 
\begin{equation}
\oone_\diamond \ \ := \ \ 
\begin{tikzpicture}[anchorbase]
\draw[line width=2.5]
(-1,0)--(-.3,0)
(1,0)--(.3,0);
\filldraw[fill=black]
(-.3,0) circle (.1)
(.3,0) circle (.1);
\node at (-1.3,0) {\scriptsize$\AS$};
\node at (1.3,0) {\scriptsize$\BS$};
\end{tikzpicture}
, \qquad \left(\text{respectively } \ \  \begin{tikzpicture}[anchorbase]
\draw[dashed]
(0,-.5) -- (0,.5);
\begin{scope}[xshift=-.5cm,xscale=-1]
\draw[ultra thick,-stealth]
(0,0)--(.7,0);
\filldraw[fill=white]
(0,0) circle (.3);
\node at (0,0) {\scriptsize$a$};
\node at (1,0) {\scriptsize$X$};
\end{scope}
\begin{scope}[xshift=.5cm,xscale=1]
\draw[ultra thick,stealth-]
(.3,0)--(.7,0);
\filldraw[fill=white]
(0,0) circle (.3);
\node at (0,0) {\scriptsize$b$};
\node at (1,0) {\scriptsize$Y$};
\end{scope}
\end{tikzpicture} \right). \end{equation}
When $\AS = \BS$, the dashed line helps to distinguish $a\otimes b$ from the composition $a\circ b$; when $\AS \ne \BS$, the composition $a \circ b$ does not make sense.

\begin{prop} The operation $\diamond$ makes $\Bim_{\AS,\BS}$ into a dg monoidal category with monoidal identity $\oone_\diamond$. \end{prop}

\begin{proof}
The associator isomorphism is induced from isomorphisms of bimodules $\mergebim{\AS}\otimes_{\AS\otimes \AS}(\mergebim{\AS}\otimes \ibim{\AS})\cong \mergebim{\AS}\otimes_{\AS\otimes \AS}(\ibim{\AS}\otimes \mergebim{\AS})$, i.e.
\begin{equation}
\begin{tikzpicture}[anchorbase]
\draw[line width=2.5pt]
(-.7,0)--(0,0)..controls++(.2,.2) and ++(-.2,0)..(.8,.5)
(0,0)..controls++(.2,-.2) and ++(-.4,0)..(1.6,-1);
\node at (-1.1,0) {\scriptsize $\AS$};
\begin{scope}[shift={(.8,.5)}]
\draw[line width=2.5pt]
(0,0)..controls++(.2,.2) and ++(-.2,0)..(.8,.5)
(0,0)..controls++(.2,-.2) and ++(-.2,0)..(.8,-.5);
\node at (1.1,.5) {\scriptsize $\AS$};
\node at (1.1,-.5) {\scriptsize $\AS$};
\end{scope}
\node at (1.9,-1) {\scriptsize $\AS$};
\end{tikzpicture}
\ \ \ \cong \ \ \ 
\begin{tikzpicture}[anchorbase,yscale=-1]
\draw[line width=2.5pt]
(-.7,0)--(0,0)..controls++(.2,.2) and ++(-.2,0)..(.8,.5)
(0,0)..controls++(.2,-.2) and ++(-.4,0)..(1.6,-1);
\node at (-1.1,0) {\scriptsize $\AS$};
\begin{scope}[shift={(.8,.5)}]
\draw[line width=2.5pt]
(0,0)..controls++(.2,.2) and ++(-.2,0)..(.8,.5)
(0,0)..controls++(.2,-.2) and ++(-.2,0)..(.8,-.5);
\node at (1.1,.5) {\scriptsize $\AS$};
\node at (1.1,-.5) {\scriptsize $\AS$};
\end{scope}
\node at (1.9,-1) {\scriptsize $\AS$};
\end{tikzpicture}
\end{equation}
with a similar isomorphism for $\splitbim{\BS}$.  The unitor isomorphisms are inherited from isomorphisms of bimodules
\begin{equation} \label{unitorisom}
\begin{tikzpicture}[anchorbase,yscale=-1]
\draw[line width=2.5pt]
(-.7,0)--(0,0)..controls++(.2,.2) and ++(-.2,0)..(.8,.5)
(0,0)..controls++(.2,-.2) and ++(-.2,0)..(.6,-.45);
\filldraw[fill=black]
(.6,-.45) circle (.1);
\node at (1.1,.5) {\scriptsize $\AS$};
\node at (-1.1,0) {\scriptsize $\AS$};
\end{tikzpicture}
\ \ \cong  \ \ 
\begin{tikzpicture}[baseline=-.1cm]
\draw[line width=2.5pt] (-.7,0) --(.7,0);
\node at (0,.3) {\scriptsize$\AS$};
\end{tikzpicture}
\ \ \cong  \ \ 
\begin{tikzpicture}[anchorbase]
\draw[line width=2.5pt]
(-.7,0)--(0,0)..controls++(.2,.2) and ++(-.2,0)..(.8,.5)
(0,0)..controls++(.2,-.2) and ++(-.2,0)..(.6,-.45);
\filldraw[fill=black]
(.6,-.45) circle (.1);
\node at (1.1,.5) {\scriptsize $\AS$};
\node at (-1.1,0) {\scriptsize $\AS$};
\end{tikzpicture}
\end{equation}
with similar isomorphism relating $\splitbim{\BS}$ and $\coubim{\BS}$.

Let us give some details on the unitor isomorphism \eqref{unitorisom}, by describing the map from the middle to the left-hand side. Let $f \colon X' \to X$ be viewed as an element of $\ibim{\AS}(X,X')$. Under the left isomorphism in \eqref{unitorisom}, $f$ would be sent to $f'\otimes\id_{\one_{\AS}}$, where $\id_{\one_{\AS}} \in \ubim{\AS}(\one_{\AS},\cdot)$, and $f' \in \mergebim{\AS}(X,\one_{\AS} \star X')$ is the composition of $X \leftarrow \one_{\AS} \star X$ (the unitor in $\AS$) with $\id_{\one_{\AS}} \star f$. Conversely, the map from left-hand side to the middle sends any pure tensor $g \otimes a$ to its composition $g \circ (a \otimes \id_{X'})$. 

We leave the remaining details as exercises.
\end{proof}

\subsection{The bimodule enrichment of $\MS$}
\label{ss:alg bimod bimod}
Now, let us fix a dg monoidal category $\AS$.  Then $\Bim_{\AS,\AS}$ has two monoidal structures, given by $\otimes_\AS$ and $\diamond$, with additional compatibilities making $\Bim_{\AS,\AS}$ into a \emph{duoidal category} (see \cite{BookStreet}).  

\begin{definition}\label{def:hom bimodule}
Let $\MS$ be a dg category which is an $\AS$ $\star$-bimodule category. Given a pair of objects $Z,Z'\in \MS$, define the \emph{hom bimodule from $Z'$ to $Z$} to be the bimodule $\bimhom{Z}{Z'} \in \Bim_{\AS,\AS}$ defined by $\bimhom{Z}{Z'}(X,X'):=\Hom_{\MS}(X\star Z\leftarrow Z'\star X')$, with bimodule structure illustrated in \eqref{eq:hom bimodule structure}.

We refer to $\bimhom{Z}{Z}$ as the endomorphism bimodule of $Z$, and denote it by $E_Z$.
\end{definition}

We may visualize elements $f\in \bimhom{Z}{Z'}(X,X')$ diagrammatically as
\[
\begin{tikzpicture}[baseline=0cm]
\path[-stealth,ultra thick, electricindigo]
(0,-.7) edge (0,.7);
\path[-stealth,ultra thick]
(1,0) edge (-1,0);
\filldraw[fill=white] (0,0) circle (.3cm);
\node at (0,0) {\scriptsize $f$};
\node at(0,-1) {\scriptsize $Z'$};
\node at (0,1) {\scriptsize $Z$};
\node at(-1.2,0) {\scriptsize $X$};
\node at (1.2,0) {\scriptsize $X'$};
\end{tikzpicture}
\qquad \qquad \left( \ \text{ or } \quad
\begin{tikzpicture}[baseline=0cm]
\draw[-stealth,ultra thick, electricindigo]
(-.5,-.7) ..controls ++(0,.5) and ++(0,-.5).. (.5,.7);
\draw[-stealth,ultra thick]
(.5,-.7) ..controls ++(0,.5) and ++(0,-.5).. (-.5,.7);
\filldraw[fill=white] (0,0) circle (.3cm);
\node at (0,0) {\scriptsize $f$};
\node at(-.5,-1) {\scriptsize $Z'$};
\node at (-.5,1) {\scriptsize $X$};
\node at(.5,1) {\scriptsize $Z$};
\node at (.5,-1) {\scriptsize $X'$};
\end{tikzpicture}
\right)
\]
with the bimodule structure being given by composing with morphisms on the left and right, i.e. 
\begin{equation}\label{eq:hom bimodule structure}
a\cdot f\cdot b \ \ = \ \  (a\star \id_Z)\circ f \circ (\id_{Z'}\star b) \ \ = \ \ 
\begin{tikzpicture}[baseline=0cm]
\path[-stealth,ultra thick, electricindigo]
(0,-.7) edge (0,.7);
\path[-stealth,ultra thick]
(1.7,0) edge (-1.7,0);
\filldraw[fill=white]
(0,0) circle (.3cm)
(-1,0) circle (.3cm)
(1,0) circle (.3cm);
\node at (0,0) {\scriptsize $f$};
\node at(0,-1) {\scriptsize $Z'$};
\node at (0,1) {\scriptsize $Z$};
\node at(-1,0) {\scriptsize $a$};
\node at (1,0) {\scriptsize $b$};
\end{tikzpicture}
\end{equation}
(This diagram is a genuine composition.) Note that $\Hom_\MS(Z',Z) = \bimhom{Z}{Z'}(\one,\one)$.  The hom bimodules $\bimhom{Z}{Z'}$ come with a composition law, which extends the usual composition of morphisms in $\MS$, defined as follows.

\begin{definition}\label{def:composition of B}
Given $Z_0,Z_1,Z_2\in \MS$ we let
\begin{equation}\label{eq:composition of B}
\bimhom{Z_2}{Z_1}\diamond \bimhom{Z_1}{Z_0}\to \bimhom{Z_2}{Z_0}
\end{equation}
be the map sending $a\otimes (f_2\otimes f_1)\otimes b$ to the composition
\begin{equation}\label{eq:affb composition}
(a\star \id_Z)\circ (\id_{X_1}\star f_2)\circ (f_1\star \id_{Y_2})\circ (\id_Z\star b)
\end{equation}
where
\[
X\buildrel a\over\leftarrow X_1\star X_2 \, \qquad X_i\star Z_{i-1}\buildrel f_i\over\leftarrow Z_i\star Y_i \, , \qquad Y_1\star Y_2\buildrel b \over\leftarrow Y
\]
\end{definition}

Diagrammatically, we draw $a,f_2,f_1,b$ as in
\[
\begin{tikzpicture}[baseline=0cm]
\draw[ultra thick]
(-.7,0)--(0,0)..controls++(.2,.2) and ++(-.2,0)..(.8,.5)
(0,0)..controls++(.2,-.2) and ++(-.2,0)..(.8,-.5);
\filldraw[fill=white]
(0,0) circle (.3cm);
\node at (0,0) {\scriptsize $a$};
\node at (1.1,.5) {\scriptsize $X_2$};
\node at (1.1,-.5) {\scriptsize $X_1$};
\node at (-1,0) {\scriptsize $X$};
\end{tikzpicture}
\qquad\qquad 
\begin{tikzpicture}[baseline=0cm]
\path[-stealth,ultra thick, electricindigo]
(0,-.7) edge (0,.7);
\path[-stealth,ultra thick]
(1,0) edge (-1,0);
\filldraw[fill=white] (0,0) circle (.3cm);
\node at (0,0) {\scriptsize $f_2$};
\node at(0,-1) {\scriptsize $Z_1$};
\node at (0,1) {\scriptsize $Z_2$};
\node at(-1.2,0) {\scriptsize $X_2$};
\node at (1.2,0) {\scriptsize $Y_2$};
\end{tikzpicture}
\qquad \qquad \begin{tikzpicture}[baseline=0cm]
\path[-stealth,ultra thick, electricindigo]
(0,-.7) edge (0,.7);
\path[-stealth,ultra thick]
(1,0) edge (-1,0);
\filldraw[fill=white] (0,0) circle (.3cm);
\node at (0,0) {\scriptsize $f_1$};
\node at(0,-1) {\scriptsize $Z_0$};
\node at (0,1) {\scriptsize $Z_1$};
\node at(-1.2,0) {\scriptsize $X_1$};
\node at (1.2,0) {\scriptsize $Y_1$};
\end{tikzpicture}
\qquad\qquad
\begin{tikzpicture}[baseline=0cm,xscale=-1]
\draw[ultra thick,-stealth]
(-.7,0)--(0,0)..controls++(.2,.2) and ++(-.4,0)..(.8,.5);
\draw[ultra thick,-stealth]
(0,0)..controls++(.2,-.2) and ++(-.4,0)..(.8,-.5);
\filldraw[fill=white]
(0,0) circle (.3cm);
\node at (0,0) {\scriptsize $b$};
\node at (1.1,.5) {\scriptsize $Y_2$};
\node at (1.1,-.5) {\scriptsize $Y_1$};
\node at (-1,0) {\scriptsize $Y$};
\end{tikzpicture}
\]
and the composition law sends $a\otimes (f_1\otimes f_2)\otimes b$ to the element depicted diagrammatically by
\begin{equation} \label{compositionlawpicture}
\begin{tikzpicture}[anchorbase]
\path[-stealth,ultra thick, electricindigo]
(0,-1.3) edge (0,.1)
(0,0) edge (0,1.3);
\begin{scope}[xscale=-1]
\draw[ultra thick,-stealth]
(-1.9,0)--(-1.2,0)..controls++(.2,.2) and ++(-.4,0)..(-.2,.5);
\draw[ultra thick,-stealth]
(-1.2,0)..controls++(.2,-.2) and ++(-.4,0)..(-.2,-.5);
\end{scope}
\draw[ultra thick,stealth-]
(-1.9,0)--(-1.2,0)..controls++(.2,.2) and ++(-.4,0)..(-.2,.5);
\draw[ultra thick,stealth-]
(-1.2,0)..controls++(.2,-.2) and ++(-.4,0)..(-.2,-.5);
\filldraw[fill=white]
(0,.5) circle (.3cm)
(0,-.5) circle (.3cm)
(-1.2,0) circle (.3cm)
(1.2,0) circle (.3cm);
\node at (0,.5) {\scriptsize $f_2$};
\node at (0,-.5) {\scriptsize $f_1$};
\node at (-1.2,0) {\scriptsize $a$};
\node at (1.2,0) {\scriptsize $b$};
\node at (-.8,.7) {\scriptsize $X_2$};
\node at (-.8,-.7) {\scriptsize $X_1$};
\node at (.8,.7) {\scriptsize $Y_2$};
\node at (.8,-.7) {\scriptsize $Y_1$};
\node at (-2.1,0) {\scriptsize $X$};
\node at (2.1,0) {\scriptsize $Y$};
\end{tikzpicture}.
\end{equation}

\begin{remark} The diagram in \eqref{compositionlawpicture} represents the genuine composition of morphisms \eqref{eq:affb composition}.   Before applying the composition law, the element $a\otimes (f_1\otimes f_2)\otimes b$ in $\bimhom{Z_2}{Z_1}\diamond \bimhom{Z_1}{Z_0}$ could be drawn using a similar formal picture, with the purple line between $f_1$ and $f_2$ broken, reflecting the fact that
\begin{equation}
a\otimes\Big(\, (f_2\circ(g\star \id)) \, \otimes \, f_1\Big) \otimes b \qquad \text{and}\qquad a\otimes\Big(\, f_2 \, \otimes \, ((\id\star g)\circ f_1)\Big) \otimes b
\end{equation}
are different elements of $\bimhom{Z_2}{Z_1}\diamond \bimhom{Z_1}{Z_0}$, but have the same image under the composition law.  To reiterate, one can transport morphisms across the purple strand in \eqref{compositionlawpicture}, analogously to \eqref{eq:morphism slide}, though one could not transport them before applying the composition law. \end{remark}

The following is a straightforward exercise.

\begin{lemma}\label{lemma:alg bimod}
The composition law \eqref{eq:composition of B} is a closed degree zero bimodule map for all $Z_0,Z_1,Z_2\in \MS$, and satisfies the associativity axiom.
\end{lemma}

We claim that the composition law \eqref{eq:composition of B} is also satisfies the unit axiom, with unit maps $u_Z \colon \oone_\diamond \to E_Z$ defined as follows.  Fix $Z\in \MS$.  For each $X, Y \in \AS$, recall that 
\[ \oone_{\diamond}(X,Y) \cong \Hom(X \leftarrow \one) \ot_{\k} \Hom(\one \leftarrow Y)\]
by definition, where $\one$ is the monoidal identity of $\AS$.  Let $\phi$ temporarily denote the canonical isomorphism $\one\star Z \leftarrow Z \star \one$. Then the unit map is 
\begin{equation}\label{eq:unit of B}
u_Z \colon \oone_\diamond \to E_Z \quad , \qquad a\otimes b\mapsto (a\star \id_Z)\circ \phi\circ (\id_Z\star b).
\end{equation}

% where
% \[
% X\buildrel a\over \leftarrow \one \, , \qquad \one\star Z\buildrel \phi\over \cong Z\star \one \, ,\qquad \one\buildrel b\over\leftarrow Y
% \]
% (here $\phi$ is the canonical isomorphism coming from the monoidal structure in $\AS)$.\qed

Diagrammatically, we view the unit map as sending
\begin{equation}
\begin{tikzpicture}[anchorbase]
\draw[dashed]
(0,-.5) -- (0,.5);
\begin{scope}[xshift=-.5cm,xscale=-1]
\draw[ultra thick,-stealth]
(0,0)--(.7,0);
\filldraw[fill=white]
(0,0) circle (.3);
\node at (0,0) {\scriptsize$a$};
\node at (1,0) {\scriptsize$X$};
\end{scope}
\begin{scope}[xshift=.5cm,xscale=1]
\draw[ultra thick,stealth-]
(.3,0)--(.7,0);
\filldraw[fill=white]
(0,0) circle (.3);
\node at (0,0) {\scriptsize$b$};
\node at (1,0) {\scriptsize$X'$};
\end{scope}
\end{tikzpicture}
 \ \ \ \mapsto \ \ \ 
\begin{tikzpicture}[anchorbase]
\draw[electricindigo,ultra thick,-stealth]
(0,-.5) -- (0,.5);
\node at (.2,-.35) {\scriptsize$Z$};
\begin{scope}[xshift=-.6cm,xscale=-1]
\draw[ultra thick,-stealth]
(0,0)--(.7,0);
\filldraw[fill=white]
(0,0) circle (.3);
\node at (0,0) {\scriptsize$a$};
\node at (1,0) {\scriptsize$X$};
\end{scope}
\begin{scope}[xshift=.6cm,xscale=1]
\draw[ultra thick,stealth-]
(.3,0)--(.7,0);
\filldraw[fill=white]
(0,0) circle (.3);
\node at (0,0) {\scriptsize$b$};
\node at (1,0) {\scriptsize$X'$};
\end{scope}
\end{tikzpicture}.
\end{equation}
% $\oone_\diamond(X,X')$ as spanned by diagrams of the form

\begin{lemma}
The map \eqref{eq:unit of B} is a closed degree zero bimodule map $\oone_\diamond\to E_Z$ for all $Z$, and satisfies the unit axiom for the given composition law \eqref{eq:composition of B}.  In particular $E_Z$ is an algebra object in $\Bim_{\AS,\AS}$ for all $Z$, and $\bimhom{Z}{Z'}$ is an $(E_Z,E_{Z'})$-bimodule object for all $Z,Z'\in \AS$.
\end{lemma}
\begin{proof}
For the unit axiom we must show that for each pair $Z,Z'\in \MS$ the following squares commute:
\begin{equation}
\begin{tikzpicture}[anchorbase]
\node (a) at (0,0) {$\bimhom{Z}{Z'}$};
\node (b) at (0,2) {$\oone_\diamond\diamond \bimhom{Z}{Z'}$};
\node (c) at (3,2) {$E_Z\diamond \bimhom{Z}{Z'}$};
\node (d) at (3,0) {$\bimhom{Z}{Z'}$};
\path[very thick,-stealth]
(a) edge node[left] {$\cong$} (b)
(b) edge node[above] {$u_Z\diamond \id$} (c)
(c) edge node[right] {\eqref{eq:composition of B}} (d)
(a) edge node[above] {$=$} (d);
\end{tikzpicture}
\qquad \quad
\begin{tikzpicture}[anchorbase]
\node (a) at (0,0) {$\bimhom{Z}{Z'}$};
\node (b) at (0,2) {$\bimhom{Z}{Z'}\diamond \oone_\diamond$};
\node (c) at (3,2) {$\bimhom{Z}{Z'}\diamond E_{Z'}$};
\node (d) at (3,0) {$\bimhom{Z}{Z'}$};
\path[very thick,-stealth]
(a) edge node[left] {$\cong$} (b)
(b) edge node[above] {$\id\diamond u_{Z'}$} (c)
(c) edge node[right] {\eqref{eq:composition of B}} (d)
(a) edge node[above] {$=$} (d);
\end{tikzpicture}
\end{equation}
We show the left square commutes; the right is similar.  The composition (up, over, down around the square) sends
\begin{equation}
\begin{tikzpicture}[baseline=0cm]
\path[-stealth,ultra thick, electricindigo]
(0,-.7) edge (0,.7);
\path[-stealth,ultra thick]
(1,0) edge (-1,0);
\filldraw[fill=white] (0,0) circle (.3cm);
\node at (0,0) {\scriptsize $f$};
\node at(0,-1) {\scriptsize $Z'$};
\node at (0,1) {\scriptsize $Z$};
\node at(-1.2,0) {\scriptsize $X$};
\node at (1.2,0) {\scriptsize $X'$};
\end{tikzpicture}
\quad \mapsto \quad 
\begin{tikzpicture}[anchorbase]
\path[-stealth,ultra thick, electricindigo]
(0,-1.3) edge (0,.1)
(0,0) edge (0,1.3);
\begin{scope}[xscale=-1]
\draw[ultra thick,-stealth]
(-1.9,0)--(-1.2,0)..controls++(.2,.2) and ++(-.4,0)..(-.2,.5);
\draw[ultra thick,-stealth]
(-1.2,0)..controls++(.2,-.2) and ++(-.4,0)..(-.2,-.5);
\end{scope}
\draw[ultra thick,stealth-]
(-1.9,0)--(-1.2,0)..controls++(.2,.2) and ++(-.4,0)..(-.2,.5);
\draw[ultra thick,stealth-]
(-1.2,0)..controls++(.2,-.2) and ++(-.4,0)..(-.2,-.5);
\filldraw[fill=white]
(0,.5) circle (.3cm)
(0,-.5) circle (.3cm)
(-1.2,0) circle (.3cm)
(1.2,0) circle (.3cm);
\node at (0,.5) {\scriptsize $\phi_Z$};
\node at (0,-.5) {\scriptsize $f$};
\node at (-1.2,0) {\scriptsize $\rho_X$};
\node at (1.2,0) {\scriptsize $\rho_{X'}\inv$};
\node at (-.8,.7) {\scriptsize $\one$};
\node at (-.8,-.7) {\scriptsize $X$};
\node at (.8,.7) {\scriptsize $\one$};
\node at (.8,-.7) {\scriptsize $X'$};
\node at (-2.1,0) {\scriptsize $X$};
\node at (2.1,0) {\scriptsize $X'$};
\end{tikzpicture}.
\end{equation}
where $\rho$ is the right unitor isomorphism in $\AS$, and $\phi$ is the natural isomorphism $-\star \one\to \one\star -$ in $\MS$.  Note that usually we do not draw occurences of the monoidal identity in our diagrams.  We have included them here, because otherwise the content of this statement is difficult to appreciate. The composition on the right above equals $f$ by the usual coherence relations for monoidal categories and their $\star$-bimodule categories.
\end{proof}

\begin{remark}
In \cite{GHW} the spaces $B_{-,-}(-,-)$ were referred to as \emph{quadmodules} over $\AS$.  In this paper we prefer slightly different point of view, which breaks the symmetry between the $X$'s and $Z$'s: for fixed $Z,Z'$ we have that $\bimhom{Z}{Z'}(-,-)$ is a bimodule over $\AS$ as well as an $(E_Z,E_{Z'})$-bimodule object in $\Bim_{\AS,\AS}$.

One might refer to $E_Z$ as \emph{algebra bimodules} since they are algebra objects in $\Bim_{\AS,\AS}$, and the $\bimhom{Z}{Z'}$ as \emph{bimodule bimodules} (or \emph{quadmodules}, keeping mind that this clashes slightly with the usage of this word in \cite{GHW}). 
\end{remark}

\begin{definition}\label{def:bimodbimodcat}
Let $\Qmod{\MS}$ be the the category enriched in $(\Bim_{\AS,\AS},\diamond,\oone_\diamond)$ whose set of objects is $\Obj(\MS)$, and whose morphisms are given by the hom bimodules
\[
\Hom_{\Qmod{\MS}}(Z',Z) := \bimhom{Z}{Z'}
\]
with composition law \eqref{eq:composition of B} and identity maps \eqref{eq:unit of B}.  We refer to $\Qmod{\MS}$ as the \emph{bimodule enrichement of $\MS$}.
%Roughly speaking, we have just defined a category internal to $(\Bim_{\AS,\AS},\diamond,\oone_\diamond)$ with objects same as $\Obj(\AS)$, and morphisms $Z\leftarrow Z'$ given by $\bimhom{Z}{Z'}$.  Then \eqref{eq:composition of B} defines the composition law and \eqref{eq:unit of B} defines the identity ``endomorphisms''.
\end{definition}

%
%\begin{definition}\label{def:tensor of B}
%Define a map
%\begin{equation}\label{eq:tensor of B}
%\bimhom{Z_1}{Z_1'}\otimes_\AS \bimhom{Z_2}{Z_2'} \to \bimhom{Z_1\star Z_2}{Z_1'\star Z_2'}
%\end{equation}
%by
%\[
%f_1\otimes f_2  \mapsto (f_1\star \id_{Z_2})\circ (\id_{Z_1'}\star f_2)
%\]
%where
%\[
%X_0\star Z_1\buildrel f_1\over\leftarrow Z_1'\star X_1 \, , \qquad X_1\star Z_2\buildrel f_2\over\leftarrow Z_2'\star X_2
%\]
%\end{definition}
%
%\begin{lemma}\label{lemma:tensor of B}
%The map \eqref{eq:tensor of B} is closed degree zero and satisfies associativity and unit constraints. \qed
%\end{lemma}
%
%\begin{remark}
%Thus, we have just defined a monoidal category internal to the category $\Bim_{\AS,\AS}$ with respect to its monoidal structures $\otimes_\AS$ and $\diamond$.
%\end{remark}

\subsection{The ordinary Drinfeld center}
\label{ss:ordinary drinfeld}

Let us express the Drinfeld centralizer $\ZS(\AS,\MS)$ using the enrichment of $\MS$ from Definition \ref{def:bimodbimodcat}.

\begin{lemma}\label{lemma:drinfeld center and bimode}
The trivial bimodule $\ibim{\AS}\in \Bim_{\AS,\AS}$ has the structure of an algebra with respect to the $\diamond$-product on $\Bim_{\AS,\AS}$, and the Drinfeld centralizer $\ZS(\AS,\MS)$ coincides with the category of $\AS$-modules in $\Qmod{\MS}$.
\end{lemma}

\begin{proof}
The unit is the map $\oone_\diamond \to \ibim{\AS}$ sending $a\otimes b\mapsto a\circ b$, as in
\begin{equation}
\begin{tikzpicture}[anchorbase]
\draw[dashed]
(0,-.5) -- (0,.5);
\begin{scope}[xshift=-.5cm,xscale=-1]
\draw[ultra thick,-stealth]
(0,0)--(.7,0);
\filldraw[fill=white]
(0,0) circle (.3);
\node at (0,0) {\scriptsize$a$};
\node at (1,0) {\scriptsize$X$};
\end{scope}
\begin{scope}[xshift=.5cm,xscale=1]
\draw[ultra thick,stealth-]
(.3,0)--(.7,0);
\filldraw[fill=white]
(0,0) circle (.3);
\node at (0,0) {\scriptsize$b$};
\node at (1,0) {\scriptsize$X'$};
\end{scope}
\end{tikzpicture}
\quad \mapsto \quad 
\begin{tikzpicture}[anchorbase]
\draw[ultra thick,-stealth] (.7,0) --(-.8,0);
\filldraw[fill=white, rounded corners]
(-.4,-.3) rectangle (.4,.3);
\node at (0,0) {\scriptsize$a \circ b$};
\node at (-1,0) {\scriptsize$X$};
\node at (1,0) {\scriptsize$X'$};
\end{tikzpicture}
\end{equation}
and the multiplication is the map $\ibim{\AS}\diamond \ibim{\AS}\to \ibim{\AS}$ sending
\begin{equation}
a\otimes (f\otimes g)\otimes b \mapsto a\circ (f\star g)\circ b.
\end{equation}
That is, multiplication maps the formal diagram in \eqref{MdiamondN} to its interpretation as a composition, which makes sense in this context.

We leave it as an exercise to verify the associativity and unit axioms.

Now, an $\ibim{\AS}$-module in $\Qmod{\MS}$ is an object $Z\in \Obj(\Qmod{\MS})=\Obj(\MS)$ equipped with map $\nu\colon \ibim{\AS}\to E_Z$ in $\Bim_{\AS,\AS}$ which is compatible with the algebra structures.  Since the trivial bimodule $\ibim{\AS}$ is generated by the identity endomorphisms $\id_X\in \AS(X,X)$, we see that $\nu\colon \ibim{\AS}\to E_Z$ is determined by $\tau_X:=\nu(\id_X)$ where $X$ ranges over the objects of $\AS$.

The requirement that $\nu$ be a map of bimodules tells us that $f \cdot \tau_{X'} = \tau_X\cdot f$ for all $X\buildrel f\over \leftarrow X'$.  Diagrammatically this gives the familiar compatibility relation
\begin{equation}
\begin{tikzpicture}[baseline=0cm]
\path[-stealth,ultra thick, electricindigo]
(0,-.7) edge (0,.7);
\path[-stealth,ultra thick]
(1.7,0) edge (-1.7,0);
\filldraw[fill=white]
(0,0) circle (.3cm)
(-1,0) circle (.3cm);
\node at (0,0) {\scriptsize $\tau_{X'}$};
\node at(0,-1) {\scriptsize $Z$};
\node at (0,1) {\scriptsize $Z$};
\node at(-1,0) {\scriptsize $f$};
\node at (1.9,0) {\scriptsize $X'$};
\node at (-1.9,0) {\scriptsize $X$};
\end{tikzpicture}
\quad = \quad
\begin{tikzpicture}[baseline=0cm]
\path[-stealth,ultra thick, electricindigo]
(0,-.7) edge (0,.7);
\path[-stealth,ultra thick]
(1.7,0) edge (-1.7,0);
\filldraw[fill=white]
(0,0) circle (.3cm)
(1,0) circle (.3cm);
\node at (0,0) {\scriptsize $\tau_{X}$};
\node at(0,-1) {\scriptsize $Z$};
\node at (0,1) {\scriptsize $Z$};
\node at (1,0) {\scriptsize $f$};
\node at (1.9,0) {\scriptsize $X'$};
\node at (-1.9,0) {\scriptsize $X$};
\end{tikzpicture}
\end{equation}
This is just \eqref{centralsquare}.

The requirement that $\nu$ be an algebra map translates to the following diagrammatic relation
\begin{equation} \label{nuisanalgebramap}
\begin{tikzpicture}[baseline=0cm]
\path[-stealth,ultra thick, electricindigo]
(0,-1.4) edge (0,1.4);
\draw[ultra thick,stealth-]
(-1.7,0)--(-.95,0)..controls++(.2,.2) and ++(-.2,0)..(-.2,.5);
\draw[ultra thick]
(-.95,0)..controls++(.2,-.2) and ++(-.2,0)..(-.2,-.5)
(1.7,0)--(.95,0)..controls++(-.2,.2) and ++(.2,0)..(.2,.5)
(.95,0)..controls++(-.2,-.2) and ++(.2,0)..(.2,-.5);
\filldraw[fill=white]
(0,.5) circle (.3cm)
(0,-.5) circle (.3cm);
\node at (0,.5) {\scriptsize $\tau_{X_2}$};
\node at (0,-.5) {\scriptsize $\tau_{X_1}$};
\node at (-.7,.6) {\scriptsize $X_2$};
\node at (-.7,-.6) {\scriptsize $X_1$};
\node at (.7,.6) {\scriptsize $X_2$};
\node at (.7,-.6) {\scriptsize $X_1$};
\node at (-2.3,0) {\scriptsize $X_1\star X_2$};
\node at (2.3,0) {\scriptsize $X_1\star X_2$};
\end{tikzpicture}
\quad = \quad
\begin{tikzpicture}[baseline=0cm]
\path[-stealth,ultra thick, electricindigo]
(0,-.7) edge (0,.7);
\path[-stealth,ultra thick]
(1.3,0) edge (-1.3,0);
\filldraw[fill=white,rounded corners] (-.6,-.3) rectangle (.6,.3);
\node at (0,0) {\scriptsize $\tau_{X_1\star X_2}$};
\node at(0,-1) {\scriptsize $Z$};
\node at (0,1) {\scriptsize $Z$};
\node at(-2,0) {\scriptsize $X_1\star X_2$};
\node at (2,0) {\scriptsize $X_1\star X_2$};
\end{tikzpicture}
\end{equation}
where the forks on the left and right should be interpreted as the identity map $\id_{X_1\star X_2}$, regarded as an elements of $\mergebim{\AS}$ and $\splitbim{\AS}$ respectively. This is equivalent to the usual multiplicativity constraint on Drinfeld central structures.

Now, suppose that $Z_1,Z_2$ are $\AS$-modules in $\Qmod{\MS}$, with structure maps $\nu_i\colon \ibim{\AS}\to E_{Z_i}$.  A map of $\ibim{\AS}$-modules $Z_1\to Z_2$ is the same thing as a map $\phi\colon \oone_\diamond\to \bimhom{Z_2}{Z_1}$ (in the enriching category $\Bim_{\AS,\AS}$) making the following diagram commute:
\begin{equation}
\begin{tikzpicture}[anchorbase]
\node (z) at (-2,1) {$\AS$};
\node (a) at (0,0) {$\AS\diamond \oone_\diamond$};
\node (b) at (3,0) {$\bimhom{Z_2}{Z_1}\diamond E_{Z_1}$};
\node (e) at (5,1) {$\bimhom{Z_2}{Z_1}$};
\node (c) at (0,2) {$\oone_\diamond \diamond \AS$};
\node (d) at (3,2) {$E_{Z_2}\diamond \bimhom{Z_2}{Z_1}$};
\path[-stealth,very thick]
(z) edge node[above] {$\cong$} (a)
(z) edge node[above] {$\cong$} (c)
(a) edge node[above] {$u_{Z_2}\diamond \phi$} (b)
(c) edge node[above] {$\phi\diamond u_{Z_1}$} (d)
(b) edge node {} (e)
(d) edge node {} (e);
\end{tikzpicture},
\end{equation}
where $u_Z$ denotes the ``enriched identity'' map $\oone_\diamond\to E_Z$ \eqref{eq:unit of B} and the rightmost arrows are given by the composition law \eqref{eq:composition of B}.  If we let $f:=\phi(\id_\one\otimes \id_\one)\in \bimhom{Z}{Z'}(\one,\one)$, then commutativity of the above diagram is easily seen to be equivalent to the following diagrammatic relation.
\begin{equation}
\begin{tikzpicture}[baseline=.5cm]
\path[-stealth,ultra thick, electricindigo]
(0,-.7) edge (0,1.7);
\path[-stealth,ultra thick]
(1,0) edge (-1,0);
\filldraw[fill=white] (0,0) circle (.3cm)
(0,1) circle (.3cm);
\node at (0,0) {\scriptsize $\tau'_X$};
\node at(0,-1) {\scriptsize $Z_2$};
\node at (0,2) {\scriptsize $Z_1$};
\node at(-1.2,0) {\scriptsize $X$};
\node at (1.2,0) {\scriptsize $X'$};
\node at (0,1) {\scriptsize $f$};
\end{tikzpicture}
\quad =\quad
\begin{tikzpicture}[baseline=-.5cm]
\path[-stealth,ultra thick, electricindigo]
(0,-1.7) edge (0,.7);
\path[-stealth,ultra thick]
(1,0) edge (-1,0);
\filldraw[fill=white] (0,0) circle (.3cm)
(0,-1) circle (.3cm);
\node at (0,0) {\scriptsize $\tau_X$};
\node at(0,-2) {\scriptsize $Z_1$};
\node at (0,1) {\scriptsize $Z_2$};
\node at(-1.2,0) {\scriptsize $X$};
\node at (1.2,0) {\scriptsize $X$};
\node at (0,-1) {\scriptsize $f$};
\end{tikzpicture}.
\end{equation}
This is just a diagrammatic reformulation of what it means to be a morphism in the Drinfeld centralizer, see \eqref{tobeamorphism}.
\end{proof}

%\begin{definition}\label{def:ordinary drinfeld}
%The Drinfeld center of $\AS$ is the category of $\AS$-modules in $\Qmod{\MS}$.  In other words, an object of $\ZS(\AS)$ is a pair $(Z,\nu)$ where $Z\in \Obj(\AS)$ and $\nu$ is an algebra map $\AS\to E_Z$ (with respect to the $\diamond$-monoidal structure on $\Bim_{\AS,\AS}$, and a morphism $(Z,\nu)\leftarrow (Z',\nu')$ in $\ZS(\BS)$ is a map $f\colon \AS\to \bimhom{Z}{Z'}$ such that the following diagram commutes
%\begin{equation}\label{eq:map in ordinary drinfeld}
%\begin{tikzpicture}
%\node (a) at (0,0) {$\AS\diamond \AS\diamond \AS$};
%\node (b) at (5,0) {$E_Z\diamond \bimhom{Z}{Z'}\diamond E_{Z'}$};
%\node (c) at (0,-2) {$\AS$};
%\node (d) at (5,-2) {$\bimhom{Z}{Z'}$};
%\path[-stealth,very thick]
%(a) edge node[above] {$\nu\diamond f\diamond \nu'$} (b)
%(a) edge node {} (c)
%(b) edge node {} (d)
%(c) edge node[above] {$f$} (d);
%\end{tikzpicture}
%\end{equation}
%where the vertical arrows are the come from the $\diamond$-multiplication on $\AS$ and the composition law \eqref{eq:composition of B}, respectively.
%\end{definition}
%

% An A-module is a map $A\to \End_\k(M)$.  An $A$-module map is the same as a map of bimodules from the trivial bimodule $A$ to the hom bimodule $\Hom_\k(M,N)$.

\subsection{Eilenberg-Zilber product}
\label{ss:EZ}
Let $\AS$ be a dg monoidal category and let $\MS$ be an $(\AS,\AS)$ $\star$-bimodule category.  The $A_\infty$-Drinfeld centralizer of $\AS$ in $\MS$ is defined in a similar spirit to the ordinary Drinfeld centralizer, but replacing the trivial bimodule $\ibim{\AS}\in \Bim_{\AS,\AS}$ by its projective resolution $\Bar_\AS$, and using $A_\infty$-modules in place of ordinary modules.  For the definition, we will need to lift the algebra structure on $\ibim{\AS}$ (see Lemma \ref{lemma:drinfeld center and bimode}) to $\Bar_\AS$.

\begin{proposition}\label{prop:bar is alg}
If $\AS$ is a dg monoidal category then $\Bar_\AS$ is an algebra in $(\Bim_{\AS,\AS},\diamond,\oone_\diamond)$.\qed
\end{proposition}
The above is well-known to experts (see \cite{EilenbergZilber,LodayQuillen,GHW}) and is a generalization of the familiar fact that the bar complex of a commutative algebra has the structure of a dg algebra (indeed, a commutative algebra can be regarded as a monoidal category with one object).   We will sketch some of the main ideas.  First, the unit map $\eta\colon \oone_\diamond \to \Bar_\AS$ is simply the inclusion of $\one_\diamond = \ibim{\AS}(-,\one_\AS)\otimes\ibim{\AS}(\one_\AS,-)$ in $\Bar_\AS$.

The multiplication $\mu_2\colon \Bar_\AS\diamond \Bar_\AS \to \Bar_\AS$ is given by the Eilenberg-Zilber shuffle product.  The standard reference for this shuffle product is \cite[Theorems 5.1 and 5.2]{EilMac1}), though it takes a little bit of translating to see the connection to bar complexes.  In the interest of thoroughness, let us say this a bit more carefully.

Let $\Coalg$ be the monoidal category freely generated by a comonoid object $C$ (in the topologists language, $\Coalg$ is the opposite of the augmented simplex category).  We may form a $\k$-linear category $\k\Coalg$ by linearizing all the morphism sets. Consider the complex
\begin{equation}
\PB_C:= \cdots \to C^3\to C^2\to \underline{C}  \in \Ch(\k\Coalg)
\end{equation}
in which the differential is an alternating sum of counit maps (i.e.~face maps).
One might call this the ``universal simplicial chain complex'' because if $X\colon \Coalg\to \mathbf{Set}$ is a simplicial set, then the image of $\PB_C$ under $X$ (after linearizing and extending to complexes) is precisely the simplicial chain complex $C_\bullet(X)$, with coefficient ring $\k$. Now, the Eilenberg-Zilber product is a chain map $C_\bullet(X)\otimes C_\bullet(Y)\to C_\bullet(X\times Y)$ (in fact, a homotopy equivalence) satisfying an appropriate version of unit and associativity axioms. This chain map is inherited from a ``universal'' version, which takes the form of a chain map
\begin{equation}
\PB_C\otimes \PB_C \rightarrow \PB_{C\otimes C}
\end{equation}
inside the tensor product of dg categories $\Ch(\k\Coalg)\otimes \Ch(\k\Coalg)$.  This universal Eilenberg-Zilber map then induces similar such maps in a wide variety of contexts.  
Of particular interest to us is the following.  If $\AS$ is a dg category then there is a monoidal functor $\k\Coalg\to \Bim_{\AS,\AS}$ which sends $C\mapsto \bigoplus_X \ibim{\AS}(-,X)\otimes \ibim{\AS}(X,-)$ and $\PB_C\mapsto \Bar_\AS$.  Then the Eilenberg-Zilber map in this context is a closed degree zero bimodule map
\begin{equation} \label{EZpart1}
\Bar_{\AS}\otimes \Bar_{\BS}\to \Bar_{\AS\otimes \BS}.
\end{equation}

Now take $\AS=\BS$ to be a dg monoidal category. Recalling that 
\[ \Bar_{\AS}^{\diamond 2} = \mergebim{}\otimes_{\AS\otimes \AS} (\Bar_{\AS}\otimes \Bar_{\AS})\otimes_{\AS\otimes\AS} \splitbim{}, \]
we can apply the map from \eqref{EZpart1} to the middle factor to get a morphism
\begin{equation} \label{EZpart2}
\Bar_{\AS}^{\diamond 2}\to \mergebim{}\otimes_{\AS\otimes \AS} \Bar_{\AS\otimes \AS}\otimes_{\AS\otimes\AS} \splitbim{}.
\end{equation}
This is the first step towards defining the multiplication map on $\Bar_{\AS}$.

Now suppose we have a dg functor $F \colon \AS \to \BS$. This does not induce a map $\Bar_{\AS} \to \Bar_{\BS}$, and indeed it could not, since these bar complexes live in different categories. Instead, it induces a natural map
\begin{equation} \BS \otimes_{\AS} \Bar_{\AS} \otimes_{\AS} \BS \to \Bar_{\BS}, \end{equation}
where we have abused notation and written $\BS$ for the $(\BS,\AS)$-bimodule or the $(\AS,\BS)$-bimodule where $\AS$ acts via $F$. That is, the left-hand side is the ``induction'' of $\Bar_{\AS}$ to a $(\BS, \BS)$-bimodule. We omit the details.

We apply this construction when $\AS$ is a dg monoidal category and $F$ is the monoidal structure map $\AS\otimes \AS\to \AS$. In this case, the induction from an $(\AS \otimes \AS, \AS \otimes \AS)$-bimodule to an $(\AS,\AS)$-bimodule is achieved using $\mergebim{}$ and $\splitbim{}$. Thus we obtain a morphism
\begin{equation} \label{EZpart3} \mergebim{}\otimes_{\AS\otimes \AS} \Bar_{\AS\otimes \AS}\otimes_{\AS\otimes\AS} \splitbim{} \to \Bar_{\AS}. \end{equation}

Composing the maps from \eqref{EZpart2} and \eqref{EZpart3} gives a morphism
\begin{equation} \label{EZcomposition}
\Bar_{\AS}^{\diamond 2}\to \Bar_{\AS}
\end{equation}
which is the associative multiplication structure on $\Bar_{\AS}$.

We do not need to know any specific formulas for this product, except that in degree zero the product is given by
\begin{equation}
(\Bar_\AS^0\diamond \Bar_\AS^0)(X,X')\to \Bar_\AS^0(X,X') \ , \qquad a\otimes ([f_0,f_1]\otimes [g_0,g_1])\otimes b\mapsto [a\circ(f_0\star g_0)\, ,\, (f_1\star g_1)\circ b]
\end{equation}
where:
\begin{itemize}
\item $a\in \mergebim{\AS}(X,X_0\star Y_0)$,
\item $f_0\otimes f_1\in\ibim{\AS}(X_0,X_1)\otimes \ibim{\AS}(X_1,X_2)\subset \Bar_\AS(X_0,X_2)$
\item $g_0\otimes g_1\in\ibim{\AS}(Y_0,Y_1)\otimes \ibim{\AS}(Y_1,Y_2)\subset \Bar_\AS(Y_0,Y_2)$,
\item $b\in \splitbim{\AS}(X_2\star Y_2,X')$.
\end{itemize}

In the special case where all the above maps are identity maps $a=b=\id_{X\star Y}$, $f_0=f_1=\id_X$, and $g_0=g_1=\id_Y$ we see that
\begin{equation} \label{degzeromult}
\id_{X\star Y}\otimes([\id_X,\id_X]\otimes[\id_Y,\id_Y])\otimes \id_{X\star Y} \mapsto [\id_{X\star Y},\id_{X\star Y}].
\end{equation}

\begin{remark}
The above formula essentially tells us that $\mu_2\colon \Bar_{\AS}^{\diamond 2} \to \Bar_\AS$ lifts the algebra structure on $\ibim{\AS}$ considered in Lemma \ref{lemma:drinfeld center and bimode}, in a suitable sense.
\end{remark}

\subsection{$A_\infty$-Drinfeld centralizers}
\label{ss:infty drinfeld}
Throughout this section we let $(\AS,\star,\one)$ be a dg monoidal category and $\MS$ a dg category with the structure of a $\star$-bimodule category over $\AS$.  To formulate our notion of $A_\infty$ Drinfeld centralizer, we will use the language of $A_\infty$ algebras and modules, specifically as developed in the Appendix \ref{ss:a infty}.

Recall that an honest algebra can always be regarded as an $A_\infty$ algebra with $\mu_n=0$ for $n\neq 2$.  Likewise, we may view Proposition \ref{prop:bar is alg} as giving $\Bar_\AS$ the structure of a (strictly unital) $A_\infty$ algebra object in $(\Bim_{\AS,\AS},\diamond,\oone_\diamond)$, with $\mu_n=0$ for $n\neq 2$.    Note that the vanishing of $\mu_1$ does not mean the differential on $\Bar_\AS$ is zero, but rather that the differential of $\Bar_\AS$ is purely \emph{internal}.  We can now present our $A_\infty$ generalization of the Drinfeld centralizer.

\begin{definition}\label{def:infty drinfeld}
The $A_\infty$-Drinfeld centralizer of $\AS$ in $\MS$, denoted $\ZS_\infty(\AS,\MS)$, is the category of strictly unital $A_\infty$ modules over $\Bar_\AS$ in $\Qmod{\MS}$.  Precisely, an object of $\ZS_\infty(\AS)$ is a pair $(Z,\boldsymbol{\nu})$ in with $Z\in \AS$ and $\boldsymbol{\nu}$ is a family of maps
\[
\tau_n\colon \Bar_\AS^{\diamond n}\to E_Z \qquad (n\geq 1)
\]
of degree $1-n$, satisfying the $A_\infty$ module relations
\begin{equation}
d(\nu_m) + \sum_{i+j+2=m} (-1)^{i} \nu_{i+j+1}(\id^{\diamond i}\diamond \mu_2\diamond \id^{\diamond j}) + \sum_{i+j=m}(-1)^i\mu_2\circ(\nu_i \diamond \nu_j)=0
\end{equation}
for all $n\geq 1$, and the strictly unital relation
\begin{equation}
\nu_{i+j+1}\circ(\id^{\diamond i}\diamond \eta \diamond \id^{\diamond j})=\begin{cases}
u_Z & \text{ if $i=j=0$}\\  0 & \text{ else}\end{cases}
\end{equation}
where $\eta\colon \oone_\diamond\to \Bar_{\AS}$ is the unit of the bar complex and $u_Z\colon \oone_\diamond\to E_Z$ denotes the unit map of $E_Z$ from \eqref{eq:unit of B}.

 A degree $l$ morphism $(Z',\boldsymbol{\nu}')\buildrel\mathbf{f}\over\rightarrow (Z,\boldsymbol{\nu})$ in $\ZS_\infty(\AS,\MS)$ is a family of maps
\[
f_n\colon \Bar_\AS^{\diamond n}\to \bimhom{Z}{Z'} \qquad (n\geq 1)
\]
of degree $l-n$, with composition law $\ast$ defined by
\begin{equation}
(\mathbf{f}\ast\mathbf{g})_m = \sum_{i+j=m} (-1)^{i|\mathbf{g}|}\mu_2\circ (f_i\diamond g_j)
\end{equation}
and differential defined so that the $m$-th component of $d\mathbf{f}$ is
\begin{equation}
d(f_m)+\sum_{i+j+2=m}(-1)^{l+1+i} f_{m-1}\circ (\id^{\diamond i}\diamond \mu_2\diamond \id^{\diamond j})\\
+\mu_2\circ \sum_{i+j=m}\Big((-1)^{il}\nu_i\diamond f_j + (-1)^{i+l}f_i\diamond \nu_i\Big)
\end{equation}
Here we are using $\mu_2$ to denote either the multiplication map $\Bar_\AS^{\diamond 2}\to \Bar_\AS$ or the multiplication map $E_Z^{\diamond 2}\to E_Z$.
\end{definition}
%
%
%The monoidal structure on the $A_\infty$-Drinfeld center is defined by $(Z,\ttau)\star (Z',\ttau):=(Z\star Z', \ttau\star \ttau')$ where $\ttau\star \ttau'$ is defined using \eqref{eq:tensor of B}.

\subsection{The $A_\infty$-Drinfeld lifting lemma}
\label{ss:infty drinfeld lifting}
The purpose of this section is to prove the following, which is an $A_\infty$ version of our Drinfeld lifting lemma (Theorem \ref{thm:lifting drinfeld}).

\begin{theorem}\label{thm:infty drinfeld}
Let $\AS$ be an additive $\k$-linear monoidal category, and $\MS$ an additive $\k$-linear $\star$-bimodule category over $(\AS,\AS)$. Suppose all hom spaces in $\AS$ are projective $\k$-modules.  Let $Z\in \Ch^b(\MS)$ be a complex such that the complex of homs $\uHom_{\MS}(Z\star X',X\star Z)$ has zero homology in negative cohomological degrees for all $X,X'\in \AS$.  Then any object $(Z,\tau)\in\ZS(\AS,\KC^b(\MS))$ lifts to a unique object of $\ZS_\infty(\AS,\Ch^b(\MS))$, which lifts to a unique object of $\ZS_\infty(\Ch^b(\AS),\Ch^b(\MS))$.
\end{theorem}

Throughout this section we maintain the assumptions of Theorem \ref{thm:infty drinfeld}.

The crux of the proof will be to carefully understand the complex $\Hom(\Bar_\AS^{\diamond n}, E_Z)$ for all $n \ge 1$. In particular, we need to know the cohomology of this
complex in degree $2-n$. We begin with some lemmas.

\begin{lemma}\label{lemma:bar diamonds are projective}
Let $\AS$ be an additive $\k$-linear monoidal category whose hom spaces are projective $\k$-modules.  Then $\Bar_\AS^{\diamond n}$ is a complex of projective $(\AS,\AS)$-bimodules for all $n \ge 1$.
\end{lemma}

\begin{proof}
In cohomological degree $-r$ the bimodule $(\Bar_\AS^{\diamond n})^{-r}$ is a direct sum over arrays of objects $X^{(i)}_j$ with $1\leq i\leq n$ and (for fixed $i$) $0\leq j\leq r_i$ for some non-negative integers $r_i$ such that $r_1+\cdots+r_n=r$.  The term associated to such an array is
\begin{equation} \label{terminarray}
\ibim{\AS}(-,Y) \otimes M \otimes \ibim{\AS}(Y',-).
\end{equation}
where $Y=X^{(1)}_0\star\cdots\star X^{(n)}_0$ and $Y'=X^{(1)}_{r_1}\star\cdots\star X^{(n)}_{r_n}$, and the multiplicity space is
\begin{equation}
M=\left(\bigotimes_{i=1}^n\bigotimes_{j=1}^{n_i} \ibim{\AS}\Big(X^{(i)}_{j-1},X^{(i)}_j\Big)\right)
\end{equation}
Recall that $\ibim{\AS}(X,X') = \Hom_{\AS}(X',X)$. Thus $M$ is the tensor product of projective $\k$-modules, so $M$ is also projective over $\k$. Since $\ibim{\AS}(-,Y)\otimes \ibim{\AS}(Y',-)$ is a projective $(\AS,\AS)$-bimodule (c.f. Lemma \ref{lemma:hom from ketbra}), each term \eqref{terminarray} is projective.
\end{proof}

\begin{lemma} \label{lem:nonegsyay} If $N \in \Bim_{\AS,\AS}$ is a complex of bimodules such that $H^k(N(X,X'))=0$ for all $X, X' \in \AS$ and all $k < 0$, then $H^k(\Hom_{\Bim_{\AS,\AS}}(\Bar_\AS^{\diamond n}, N)) = 0$ for all $n \ge 1$ and all $k < 0$. \end{lemma}

\begin{proof} By construction $\Bar_{\AS}$ is a complex supported in purely negative homological degrees, and thus the same is true of $\Bar_{\AS}^{\diamond n}$. There is a spectral sequence
abutting to $\Hom_{\Bim_{\AS,\AS}}(\Bar_\AS^{\diamond n}, N)$, whose $E_0$-page is $\prod_{r \ge 0} \Hom_{\Bim_{\AS,\AS}}((\Bar_\AS^{\diamond n})^{-r},N)$. As noted in the
previous lemma, $(\Bar_\AS^{\diamond n})^{-r}$ is a sum of terms as in \eqref{terminarray}. As in Lemma \ref{lemma:hom from ketbra}, with $M$ as in \eqref{terminarray}, we have
\begin{equation} \Hom(\ibim{\AS}(-,Y) \otimes M \otimes \ibim{\AS}(Y',-),N) = \Hom_\k(M, N(Y,Y')). \end{equation} Note that $M$ is a $\k$-module, not a complex of $\k$-modules. In particular, $\Hom_\k(M , N(Y,Y'))$ has no cohomology in negative
degrees. With some bookkeeping, one verifies that the $E_1$-page is already zero in any place which could contribute to negative degree cohomology in
$\Hom_{\Bim_{\AS,\AS}}(\Bar_\AS^{\diamond n}, N)$. \end{proof}

\begin{lemma} \label{lem:diamondsquareh0}
Suppose $N\in \Bim_{\AS,\AS}$ satisfies $H^k(N(X,X'))\cong 0$ for all $k<0$ and all $X,X'\in \AS$.  Then
\begin{equation}
H^0(\Hom_{\Bim_{\AS,\AS}}(\Bar_{\AS}^{\diamond 2}, N)) \cong \Hom_{\Bim_{\AS,\AS}}(\ibim{\AS}^{\diamond 2}, H^0(N)).
\end{equation}
\end{lemma}

\begin{proof} Using \eqref{diamondadjunction} we have
\begin{equation}\label{barsquared} \Hom_{\Bim_{\AS,\AS}}(\Bar_{\AS}^{\diamond 2}, N)) \cong \Hom_{\Bim_{\AS \otimes \AS,\AS \otimes \AS}}(\Bar_{\AS} \otimes \Bar_{\AS}, \splitbim{\AS} \otimes_{\AS} N \otimes_{\AS} \mergebim{\AS}) \end{equation}
and
\begin{equation}\label{ibimsquared} \Hom_{\Bim_{\AS,\AS}}(\ibim{\AS}^{\diamond 2}, N) \cong \Hom_{\Bim_{\AS \otimes \AS,\AS \otimes \AS}}(\ibim{\AS} \otimes \ibim{\AS}, \splitbim{\AS} \otimes_{\AS} N \otimes_{\AS} \mergebim{\AS}). \end{equation}
Note that $I := \ibim{\AS}\otimes \ibim{\AS}$ is supported in homological degree zero. Since $\Bar_{\AS}$ is a projective resolution of $\ibim{\AS}$, $P := \Bar_{\AS} \otimes \Bar_{\AS}$ is a projective resolution of $\ibim{\AS}\otimes \ibim{\AS}$ (in the abelian category of $(\AS \otimes \AS,\AS \otimes \AS)$-bimodules). Note that $Q := \splitbim{\AS} \otimes_{\AS} N \otimes_{\AS} \mergebim{\AS}$ has homology supported in non-negative degrees, by our hypothesis on $N$.

Now the result follows from basic arguments in homological algebra, which we state in a simpler form within an abelian category. Let $P$ be a projective resolution of an object
$I$, and let $Q$ be a complex with $H^k(Q) = 0$ for $k < 0$. Then $H^0(\uHom(P,Q)) \cong \Hom(I,H^0(Q))$. Note that we are looking for (degree zero) chain maps $P \to Q$, and
since $P$ is supported in degrees $\le 0$, we may replace $Q$ with its truncation (i.e. replace $Q^i$ with $0$ for $i > 0$, and replace $Q^0$ with the kernel of $Q^0 \to Q^1$).
Now $Q$ is a resolution (but not a projective resolution) of $H^0(Q)$. The result now follows from the comparison theorem, \cite[Theorem 2.2.6]{Wei94}.
\end{proof}

%\begin{lemma}
%The bimodules $\splitbim{\AS}$ and $\mergebim{\AS}$ are projective $\AS$-modules from the right and left, respectively.
%\end{lemma}
%\begin{proof}
%The first statement holds since for fixed $X_1,X_2\in \AS$ we have $\splitbim{\AS}(X_1\otimes X_2,-) = \ibim{\AS}(X_1\star X_2,-)$ which is projective from the right.  The second statement is similar.
%\end{proof}

\begin{proof}[Proof of Theorem \ref{thm:infty drinfeld}]
The first step is to use our previous lifting results (not involving $A_{\infty}$-algebras) to construct $\nu_1$.
	
The hypothesis on $Z$ is equivalent to $H^k(E_Z)=0$ for $k<0$.  Let $L_Z,R_Z\colon \AS \to \Ch^b(\MS)$ be the functors given by $Z\star - $ and $-\star Z$, as usual. Observe that the bimodule which represents natural transformations from $L_Z$ to $R_Z$ satisfies $M^{R_Z}_{L_Z}\cong E_Z$.  The degree zero homology of this bimodule is
\begin{align}
H^0(E_Z(X,X'))&=H^0(\uHom_{\MS}(X\star Z\leftarrow Z'\star X'))\\ \nonumber
&\cong \Hom_{\KC^b(\MS)}(X\star Z\leftarrow Z'\star X')\\ \nonumber
&= M^{H^0(R_Z)}_{H^0(L_Z)}(X,X').
\end{align}
By hypothesis, we are given an invertible natural transformation $\tau\colon H^0(L_Z)\to H^0(R_Z)$.  From the above computation together with Lemma \ref{lemma:representing nat trans}, $\tau$ corresponds to a morphism $\tilde{\nu} \colon \ibim{\AS} \to H^0(E_Z)$, where $\tilde{\nu}(\id_X) = \tau_X$.  From the hom vanishing hypothesis, we may use Lemma \ref{lemma:infty trans 2} to lift $\tilde{\nu}$ to a closed degree zero morphism $\nu_1\colon \Bar_\AS\to E_Z$ in $\Bim_{\AS,\AS}$. Note that $\nu_1$ is unique up to homotopy by Lemma \ref{lemma:infty trans 2}.

Next, we need to construct a degree $-1$ morphism $\nu_2\colon \Bar_\AS^{\diamond 2}\to E_Z$ such that
\begin{equation} \label{nu2condition}
d(\nu_2) = -\nu_1\circ \mu + \mu\circ (\nu_1 \diamond \nu_1).
\end{equation}
Let $\Theta(\nu_1)$ denote the right-hand side of \eqref{nu2condition}. Then $\Theta(\nu_1)$ is a closed degree zero morphism $\Bar_\AS^{\diamond 2} \to E_Z$ which measures the failure of $\nu_1$ to be multiplicative. It descends to an element of $H^0(\Hom(\Bar_\AS^{\diamond 2},E_Z))$. If the homology class of $\Theta(\nu_1)$ is zero, then it is nulhomotopic, so we may choose a homotopy $\nu_2$ satisfying \eqref{nu2condition}. Note that $\nu_2$ is unique up to homotopy, since $H^{-1} \Hom_{\Bim_{\AS,\AS}}(\Bar_\AS^{\diamond 2}, E_Z) = 0$ by Lemma \ref{lem:nonegsyay}. So we need to show that $\Theta(\nu_1)$ is zero in homology.

By Lemma \ref{lem:diamondsquareh0}, $H^0(\Hom(\Bar_\AS^{\diamond 2},E_Z)) \cong \Hom(\ibim{\AS}\diamond\ibim{\AS}, H^0(E_Z))$. The morphism $\ibim{\AS}\diamond\ibim{\AS}\to H^0(E_Z)$ corresponding to $\Theta(\nu_1)$ is $\Theta(\tilde{\nu})$, defined by an analogous formula but with $\tilde{\nu}$ replacing $\nu_1$. We claim that $\Theta(\tilde{\nu})$ is zero. Note that $\ibim{\AS}\diamond\ibim{\AS}$ is generated as a bimodule by identity maps $\id_X \otimes \id_Y$ as $X, Y$ vary amongst objects of $\AS$. So we need only show that $\Theta(\tilde{\nu})$ kills identity maps. But $\Theta(\tilde{\nu})$ sends identity maps to $\Theta(\tau)$ (the analogous formula with $\tau$ replacing $\nu_1$), which is zero by the multiplicativity condition for $\tau$, see \eqref{nuisanalgebramap}. See also \eqref{degzeromult} for corroboration that this calculation does match the multiplication structure in the bar complex. %\BE{Is this sufficient? Was I too fast in the second sentence?}

Now, we fix an integer $m\geq 3$ and we assume by induction that we have constructed $\nu_1,\ldots,\nu_{m-1}$.  We need to construct a degree $1-m$ morphism $\nu_{m}\colon \Bar_\AS^{\diamond m}\to E_Z$ such that
\begin{equation} \label{numcondition}
d(\nu_{m}) = - \sum_{i+j+2=m} (-1)^{i} \nu_{i+j+1}(\id^{\diamond i}\diamond \mu_2\diamond \id^{\diamond j}) - \sum_{i+j=m}(-1)^i\mu_2\circ(\nu_i \diamond \nu_j)=0
\end{equation}
(using notation from Definition \ref{def:a infty module}).  One can check that the right-hand side above is a closed degree $2-m$ morphism $\Bar_\AS^{\diamond m}\to E_Z$.  By Lemma \ref{lem:nonegsyay}, we must have that the right-hand side of \eqref{numcondition} is null-homotopic.  Thus we may choose $\nu_{m}$ satisfying \eqref{numcondition}, and it is unique up to homotopy again by Lemma \ref{lem:nonegsyay}. This completes the proof of the theorem.
\end{proof}

\appendix

\section{An enriched category perspective on $A_\infty$-algebras and modules}
\label{ss:a infty}
Let $\BB$ be a dg monoidal category.  Denote the monoidal structure on $\BB$ by $\diamond$, and the monoidal identity by $\oone_\diamond$.  Let $\QS$ be a category enriched in $\BB$.  We will denote the homs from $Z'$ to $Z$ in $\QS$ using the notation $\bimhom{Z}{Z'}\in \BB$ and $E_Z=\bimhom{Z}{Z}$.

An $A_\infty$ algebra in $\BB$ is an object $A\in \BB$ equipped with maps
\[
\mu_n\in A^{\diamond n}\to A
\]
of degree $2-n$ (indexed by $n\in \Z_{\geq 0}$) satisfying
\begin{equation}
\sum_{i+j+n=m} (-1)^{i+nj} \mu_{i+j+1}\circ (\id^{\diamond i}\diamond \mu_n\diamond \id^{\diamond j}) = 0 
\end{equation}
for all $m\in \Z_{\geq 0}$.  Equivalently, consider the tensor coalgebra $T(A[1]):=\bigoplus_{n\geq 0} (A[1])^{\diamond n}\cong \bigoplus_{n\geq 0} A^{\diamond n}[n]$.  The $A_\infty$ algebra structure is equivalent to a degree 1 coderivation $\boldsymbol{\mu}\in \End_{\BB}(T(A[1]))$ satisfying the Maurer-Cartan equation $d(\boldsymbol{\mu})+\boldsymbol{\mu}\circ \boldsymbol{\mu} = 0$.

\begin{remark}
The tradition is to define a unital $A_\infty$ algebra structures in the following way.  Let $A$ be an object of $\BB$ equipped with direct sum decomposition $A=A'\oplus \one$ with respect to which the inclusion of $\one$ in $A$ is closed and degree zero, but the projection onto $\one$ need not be closed.  Then a unital $A_\infty$ algebra structure on $A$ is the same thing as an ordinary $A_\infty$ algebra structure on $A'$.
\end{remark}

\begin{remark}
The element $\mu_0\in \Hom_\BB^2(\oone,A)$ is called the curvature of $A$.
The element $\mu_1\in \End_\BB^1(A)$ satisfies the Maurer-Cartan equation $d(\mu_1)+\mu_1^2 =\mu_2\circ(\mu_0\diamond \id) - \mu_2\circ (\id\diamond \mu_0)$.  Frequently $\mu_1$ is added to the existing differential on $A$, obtaining a total differential $d_A+\mu_1$ (with curvature, if $\mu_0\neq 0$), but we will not do this.
\end{remark}

\begin{remark}
Below, for simplicity we will assume $\mu_0=0$.
\end{remark}

\newcommand{\TSA}{C}

\begin{definition}\label{def:pre ainfty}
Let $A$ be an $A_\infty$ algebra in $\BB$, and let $\TSA:=T(A[1])$ be the tensor coalgebra equipped with its differential $\boldsymbol{\mu}$. Let $\QS$ be a category enriched in $\BB$.  Define $\premod{A}{\QS}$ to be the dg category with the same objects as $\QS$, and morphism complexes given by
\[
\Hom_{\premod{A}{\QS}}(N,M):=\Hom_{\BB}(\TSA,\bimhom{M}{N})
\]
The composition of morphisms $\mathbf{f}\in \Hom_{\premod{A}{\QS}}(M,L)$, $\mathbf{g}\in \Hom_{\premod{A}{\QS}}(N,M)$ is denoted $\mathbf{f}\ast \mathbf{g}$, and is defined to be the composition
\begin{equation}\label{eq:f ast g}
\begin{tikzpicture}[anchorbase]
\node (a) at (0,0) {$C$};
\node (b) at (2,0) {$C\diamond C$};
\node (c) at (5,0) {$\bimhom{L}{M}\diamond \bimhom{M}{N}$};
\node (d) at (7,0) {$\bimhom{L}{N}$};
\path[-stealth,very thick]
(a) edge node[above] {$\Delta$} (b)
(b) edge node[above] {$\mathbf{f}\diamond \mathbf{g}$} (c)
(c) edge node[above] {} (d)
(a) edge[bend right] node[above] {$\mathbf{f}\ast \mathbf{g}$} (d);
\end{tikzpicture}
\end{equation}
where $\Delta$ is the comultiplication on $C$ and the last arrow is the composition of morphisms in $\QS$.  In terms of components, an element $\mathbf{f}\in \Hom_{\premod{A}{\QS}}^l(N,M)$ is a collection of maps
\[
f_n\colon A^{\diamond n}\to \bimhom{M}{N}
\]
for each $n \ge 0$, having degree $l-n$.  The components of $d(\mathbf{f})$ are
\begin{equation}
d(\mathbf{f})_m = d(f_m) - (-1)^l\sum_{i+j+n=m} (-1)^{i+nj} f_{i+j+1}\circ(\id^{\diamond i}\diamond \mu_n\diamond \id^{\diamond j})
\end{equation}
where $d(f_m)$ is the differential of $f_m$ computed in the hom complex $\Hom_\BB(A^{\diamond m},M)$.

The component $(\mathbf{f}\ast \mathbf{g})_m$ is the sum over pairs of indices $i,j$ with $i+j=m$ of composition of morphisms 
\begin{equation} \label{thisguyhere}
\begin{tikzpicture}
\node (a) at (0,0) {$A^{\diamond m}$};
\node (b) at (2,0) {$A^{\diamond i}\diamond A^{\diamond j}$};
\node (c) at (6,0) {$\bimhom{L}{M}\diamond \bimhom{N}{L}$};
\node (d) at (8,0) {$\bimhom{M}{L}$};
\path[-stealth,very thick]
(a) edge node[above] {$\cong$} (b)
(b) edge node[above] {$(-1)^{i|\mathbf{g}|}f_i\diamond g_j$} (c)
(c) edge node[above] {$\circ $} (d);
\end{tikzpicture}.
\end{equation}
The sign above is to account for the discrepancy in the Koszul signs involved in interpreting $f_i\diamond g_i$ as a morphism $A^{\diamond i}[i]\diamond A^{\diamond j}[j]\to \bimhom{L}{M}\diamond \bimhom{M}{N}$ versus as a morphism $A^{\diamond i}\diamond A^{\diamond j}\to \bimhom{L}{M}\diamond \bimhom{M}{N}$.  We write the composition in \eqref{thisguyhere} as
\begin{equation}\label{eq:f ast g component}
(\mathbf{f}\ast \mathbf{g})_m = \sum_{i+j=m} (-1)^{i|g|} \mu_2^B\circ (f_i\diamond g_j)
\end{equation}
where $\mu_2^B$ in the above equation refers to the composition law $\bimhom{L}{M}\diamond \bimhom{M}{N}\to \bimhom{L}{N}$.
\end{definition}

We refer to the dg category $\premod{A}{\QS}$ as the category of pre-$A_\infty$-modules over $A$.
We think of an object of $\premod{A}{\QS}$ as an $A_\infty$ module on which $A$ acts by zero (recall that $A$ is not assumed unital).  A more general $A_\infty$ module is defined to be a twist of an object $M$ inside $\premod{A}{\QS}$.

% by definition $\Tw(\premod{A}{\QS})$
  
\begin{definition}\label{def:a infty module}
The category of $A_\infty$ modules over $A$ is defined as follows.  An object of this category is a pair $(M,\boldsymbol{\nu})$ in which $M$ is an object of $\QS$ and $\boldsymbol{\nu}\in \Hom_{\BB}(C,\bimhom{M}{M})$ is a degree 1 element satisfying the Maurer-Cartan equation $d(\boldsymbol{\nu})+\boldsymbol{\nu}\ast \boldsymbol{\nu}=0$. We denote the pair $(M,\boldsymbol{\nu})$ as $\tw_{\boldsymbol{\nu}}(M)$.

Elaborating component by component, an $A_\infty$ module is an object $M\in \QS$ equipped with a family of maps
\[
\nu_n\colon A^{\diamond n}\to \bimhom{M}{M} \qquad (n\geq 1)
\]
of degree $1-n$, satisfying
\begin{equation}
d(\nu_m) + \sum_{i+j+n=m} (-1)^{i+nj} \nu_{i+j+1}(\id^{\diamond i}\diamond \mu_n^A\diamond \id^{\diamond j}) + \sum_{i+j=m}(-1)^i\mu_2^B\circ(\nu_i \diamond \nu_j)=0
\end{equation}
for all $m\in \Z_{\geq 0}$.

The hom complex $\tw_{\boldsymbol{\nu}^M}(M)$ to $\tw_{\boldsymbol{\nu}^N}(N)$  is by definition $\Hom_\BB(C,\bimhom{N}{M})$ with differential
\begin{equation}
d(\mathbf{f}) + \boldsymbol{\nu}^N\ast \mathbf{f} + (-1)^{|\mathbf{f}|} \mathbf{f}\ast \boldsymbol{\nu}^M.
\end{equation}
Elaborating component by component, the $m$-th component of the differential of a degree $l$ morphism $\mathbf{f}$ is a sum of terms
\begin{itemize}
\item $d(f_m)$ (calculated in the hom complex $\Hom_\BB(A^{\diamond m} , \bimhom{N}{M})$)
\item $\sum_{i+j+n=m}(-1)^{l+1+i+nj} f_{i+j+1}\circ (\id^{\diamond i}\diamond \mu_n^A\diamond \id^{\diamond j})$.
\item $\sum_{i+j=m}(-1)^{il}\mu_2^B\circ (\nu_i^N\diamond f_j)$.
\item $\sum_{i+j=m}(-1)^{i+l}\mu_2^B\circ (f_i\diamond \nu_i^M)$.
\end{itemize}

If $A$ comes equipped with a strict unit $\eta\colon \oone_{\diamond}\to A$, then we say that $\tw_{\boldsymbol{\nu}}(M)$ is \emph{strictly unital} if 
\begin{equation}
\nu_{i+j+1}\circ(\id^{\diamond i}\diamond \eta \diamond \id^{\diamond j})=\begin{cases}
u_Z & \text{ if $i=j=0$}\\  0 & \text{ else}\end{cases}
\end{equation}
where $u_M\colon \oone_{\diamond}\to \bimhom{M}{M}$ is the unit map of $M$.
\end{definition}

%\bibliographystyle{myalpha}
%\bibliography{GR4allSB}

\bibliographystyle{alpha}
\bibliography{mastercopy}

\end{document}